\documentclass[sn-mathphys,Numbered]{sn-jnl}

\usepackage{graphicx}%
\usepackage{multirow}%
\usepackage{amsmath,amssymb,amsfonts}%
\usepackage{amsthm}%
\usepackage{mathrsfs}%
\usepackage[title]{appendix}%
\usepackage{xcolor}%
\usepackage{textcomp}%
\usepackage{manyfoot}%
\usepackage{booktabs}%
\usepackage{algorithm}%
\usepackage{algorithmicx}%
\usepackage{algpseudocode}%
\usepackage{listings}%

\theoremstyle{thmstyleone}%
\newtheorem{theorem}{Theorem}
\newtheorem{proposition}[theorem]{Proposition}%

\theoremstyle{thmstyletwo}%
\newtheorem{remark}{Remark}%

\theoremstyle{thmstylethree}%
\newtheorem{definition}{Definition}%

\newtheorem{assumption}{Assumption}
\newtheorem{corollary}{Corollary}
\usepackage{placeins}
\usepackage{graphicx}
\usepackage{subfigure}
\usepackage{bm}
\usepackage{arydshln}
\hypersetup{hidelinks}

\def\Mbound{\varsigma}
\renewcommand{\vartheta}{\rho}

\def\blue#1{\textcolor{blue}{#1}}

\def\cS{{\cal S}}
\def\norm#1{\|#1\|}
\renewcommand{\iota}{{\footnotesize{l}}}

\begin{document}

\title[Article Title]{On solving a rank regularized minimization problem via equivalent factorized column-sparse regularized models}


\author[1]{\fnm{Wenjing} \sur{Li}}\email{liwenjingsx@163.com}

\author*[1]{\fnm{Wei} \sur{Bian}}\email{bianweilvse520@163.com}

\author[2]{\fnm{Kim-Chuan} \sur{Toh}}\email{mattohkc@nus.edu.sg}

\affil[1]{\orgdiv{School of Mathematics}, \orgname{Harbin Institute of Technology}, \orgaddress{\city{Harbin},\postcode{150001}, \country{China}}}

\affil[2]{\orgdiv{Department of Mathematics}, \orgname{National University of Singapore}, \orgaddress{\street{10 Lower Kent Ridge Road}, \postcode{119076}, \country{Singapore}}}


\abstract{Rank regularized minimization problem is an ideal model for the low-rank matrix completion/recovery problem. The matrix factorization approach can transform the high-dimensional rank regularized problem to a low-dimensional factorized column-sparse regularized problem. The latter can greatly facilitate fast computations in applicable algorithms, but needs to overcome the simultaneous non-convexity of the loss and regularization functions. In this paper, we consider the factorized column-sparse regularized model. Firstly, we optimize this model with bound constraints, and establish a certain equivalence between the optimized factorization problem and rank regularized problem. Further, we strengthen the optimality condition for stationary points of the factorization problem and define the notion of strong stationary point. Moreover, we establish the equivalence between the factorization problem and its  nonconvex relaxation in the sense of global minimizers and strong stationary points. To solve the factorization problem, we design two types of algorithms and give an adaptive method to reduce their computation. The first algorithm is from the relaxation point of view and its iterates own some properties from global minimizers of the factorization problem after finite iterations. We give some analysis on the convergence of its iterates to a strong stationary point. The second algorithm is designed for directly solving the factorization problem. We improve the PALM algorithm introduced by Bolte et al. (Math Program Ser A 146:459-494, 2014) for the factorization problem and give its improved convergence results. Finally, we conduct numerical experiments to show the promising performance of the proposed model and algorithms for low-rank matrix completion.}




\maketitle

\section{Introduction}
Matrix completion is to recover a matrix from  an incomplete data matrix with a few observed entries. In general, it is impossible to achieve accurate matrix completion from partially observed entries because the unobserved entries can take any values without further assumption. However, in many areas such as data mining and machine learning, since the high-dimensional data can be represented or reconstructed by low-dimensional features, the data matrices are often of low-rank. The low-rank property of data matrix not only helps to avoid the ill-posedness, but also open a viable avenue for matrix completion. Thus, over the past decade, low-rank matrix completion has attracted significant attention due to its wide applications, including collaborative filtering \cite{WestIEEE2016}, system identification \cite{LiuSIMA2009}, image inpainting \cite{LiuTIP2016} and classification \cite{CabralPAMI2015}.

Let $Z^\diamond\in\mathbb{R}^{m\times n}$ be a low-rank matrix to be recovered (which we call as a target matrix) with $\text{rank}(Z^\diamond)=r\ll\min\{m,n\}$, and $\varGamma\subseteq\{1,\ldots,m\}\times\{1,\ldots,n\}$ denote the index set of observed entries. To recover the missing entries of the low-rank matrix $Z^\diamond$, a natural idea is to solve the rank minimization problem:
\begin{equation*}
\min_{Z\in\mathbb{R}^{m\times n}}~\text{rank}(Z),~~\text{s.t.}~{\mathcal{P}_{\varGamma}(Z)=\mathcal{P}_{\varGamma}(Z^\diamond)},
\end{equation*}
where the operator $\mathcal{P}_{\varGamma}:\mathbb{R}^{m\times n}\rightarrow \mathbb{R}^{m\times n}$ satisfies $(\mathcal{P}_{\varGamma}(Z))_{ij}=Z_{ij}\ \text{if}\ (i,j)\in\varGamma\ \text{and}\ 0\ \text{otherwise}$. However, the rank minimization problem is NP-hard in general due to the nature of the rank function. In real-world applications, the observed entries in data matrices typically contain a small amount of noise. Hence, it is necessary to complete the low-rank matrix from partially sampled noisy entries. Let $\mathcal{E}$ represent the unknown noise matrix. Suppose we observe $\mathcal{P}_{\varGamma}(\bar{Z}^\diamond)$ with $\bar{Z}^\diamond=Z^\diamond+\mathcal{E}$. Since there is no prior knowledge about the noise in practice, the low-rank matrix completion model can be reasonably expressed by the following rank regularized least squares problem:
\begin{equation}\label{lstsql}
\min_{Z\in\mathbb{R}^{m\times n}}\dfrac{1}{2}\Vert \mathcal{P}_{\varGamma}(Z-\bar{Z}^\diamond)\Vert_F^2+\lambda\,\text{rank}(Z),
\end{equation}
where $\lambda$ is a positive parameter to trade off the low-rank requirement and data fidelity. Different from most of the existing works focusing on the relaxation forms of model (\ref{lstsql}), in this paper, we are particularly interested in the following original and more general rank regularized problem on matrix recovery:
\begin{equation}\label{LRMP}
\min_{Z\in\mathbb{R}^{m\times n}}\frac{1}{2}f(Z)+\lambda\,\text{rank}(Z),
\end{equation}
where the loss function $f:\mathbb{R}^{m\times n}\rightarrow\mathbb{R}$ is assumed to be a locally Lipschitz continuous function, and it could possibly depend on 
some data such as $\bar{Z}^\diamond$. In general, the formulation of $f$ is related to the observation $\mathcal{P}_{\varGamma}(\bar{Z}^\diamond)$ of the target matrix $Z^\diamond$. Model (\ref{LRMP}) is not only applicable to recover a low-rank matrix from sparse corruptions, but can also accommodate the robust principal component analysis problem \cite{Fan2020TNNLS}. For the case with Gaussian noise, the squared loss $f(Z)=\|\mathcal{P}_{\varGamma}(Z-\bar{Z}^\diamond)\|_F^2/2$ is extensively utilized. Moreover, to address the cases with non-Gaussian noise, sparse noise or outliers, it is often necessary to choose $f$ as $\ell_1$-loss, Huber loss \cite{huber1973robust} or quantile loss \cite{Koenker2001QR} functions.

In order to avoid the computational intractability of the rank regularized problem, a popular approach is to replace the rank function by a tractable relaxation. Given a matrix $X\in\mathbb{R}^{m\times n}$, some common relaxations of $\text{rank}(X)$ are nuclear norm (also called trace norm or Ky Fan-norm) $\|X\|_*:=\sum_{i=1}^{\min\{m,n\}}\sigma_i(X)$ \cite{Srebro2005}, weighted nuclear norm \cite{Gu_2014_CVPR}, truncated nuclear norm \cite{HuPAMI2013}, max-norm $\|X\|_{\max}:=\min_{X=UV^{\mathbb{T}}}\|U\|_{2,\infty}\|V\|_{2,\infty}$ \cite{Srebro2005} and Schatten-$p$ ($0<p<1$) norm $\|X\|_{S_p}:=\big(\sum_{i=1}^{\min\{m,n\}}\sigma_i^p(X)\big)^{1/p}$ \cite{ShangAAAI2016}, where $\sigma_i(X)$ denotes the $i$-th largest singular value of $X$. In particular, $\|X\|_{S_p}^p$ equals to $\|X\|_*$ when $p=1$ and equals to $\text{rank}(X)$ when $p=0$. Based on the convexity and tractability of nuclear norm, there exist various algorithms for solving the nuclear norm regularized problems \cite{Cai2010SIOPT,Toh2010accelerated}. Nevertheless, for the nuclear norm regularized problems with noiseless data, the high probability recovery often depends on the uniform sampling distribution \cite{Candes2009exact}, which is unlikely to be satisfied in practical applications. Under non-uniform sampling, both the max-norm and weighted nuclear norm regularizations have been shown to outperform the vanilla nuclear norm regularization through empirical comparisons \cite{Cai2016ES,Fang2018MP,Srebro2010collaborative}. Moreover, different from the equal treatment of all singular values in the nuclear norm, the weighted nuclear norm, truncated nuclear norm and Schatten-$p$ ($0<p<1$) norm regularizations are often used to reduce the bias in the singular values of the constructed matrices. 
As one can observe from the above discussions, all the regularizers that are based on singular
values require the computation of the 
singular value decomposition (SVD) of the underlying matrix,
which is computationally expensive for large-scale matrices.
Even though based on the closed forms of some proximal operators, one only needs to compute 
partial SVD's of the iteration matrices when the 
the solution is low-rank, computing a partial SVD based on the usual  
Lanczos-type methods through 
matrix-vector products \cite{Toh2010accelerated,yao2018large} can 
still incur considerable computational cost for large scale problems.
To overcome the computational challenges posed by SVD based regularized models, we
turn to factorization models in this paper.
 

For any given matrix $Z$ with $\text{rank}(Z)\leq d$, it can be represented by $Z=XY^{\mathbb{T}}$ with factors $X\in\mathbb{R}^{m\times d}$ and $Y\in\mathbb{R}^{n\times d}$. This matrix factorization technique can reduce the dimensions of variables in 
low-rank optimization models and is a popular SVD-free rank surrogate method to recover low-rank matrices of large dimension. When $\text{rank}(Z)\leq d$, a well-known factored formulation of the nuclear norm is $\|Z\|_*=\min_{XY^{\mathbb{T}}=Z}\frac{1}{2}(\|X\|_F^2+\|Y\|_F^2)$ with $X\in\mathbb{R}^{m\times d}$ and $Y\in\mathbb{R}^{n\times d}$. Recently, a similar factorization characterization of the Schatten-$p$ ($0<p<1$) norm is developed and can be regarded as a relaxation of the factored formulation of the rank function \cite{Fan2019factor}. Moreover, some promising performance of these factored formulations has been shown in experiments. Inspired by the above ideas, we will consider the equivalent factorization as a surrogate for the rank function in the rank regularized problem (\ref{LRMP}).

Denote $\text{nnzc}(A)$ as the number of nonzero columns of matrix $A$.
When $\text{rank}(Z)\leq d$, it has the following equivalent factored formulation with respect to $X\in\mathbb{R}^{m\times d}$ and $Y\in\mathbb{R}^{n\times d}$:
\begin{equation}\label{rank-rel}
\text{rank}(Z)=\min_{XY^{\mathbb{T}}=Z}\text{nnzc}(X)=\min_{XY^{\mathbb{T}}=Z}\text{nnzc}(Y)=\min_{XY^{\mathbb{T}}=Z}\big(\text{nnzc}(X)+\text{nnzc}(Y)\big)/2.
\end{equation}
Moreover, based on $\{Z\in\mathbb{R}^{m\times n}:\text{rank}(Z)\leq d\}=\{XY^{\mathbb{T}}:X\in\mathbb{R}^{m\times d},Y\in\mathbb{R}^{n\times d}\}$, one can easily verify that problem \eqref{LRMP} with $\text{rank}(Z)\leq d$, i.e.,
\begin{equation}\label{LRMPR}
\min_{Z\in\mathbb{R}^{m\times n},\text{rank}(Z)\leq d}\frac{1}{2}f(Z)+\lambda\,\text{rank}(Z),
\end{equation}
can equivalently be expressed by the following factorized column-sparse regularized problem in the sense of optimal value:
\begin{equation}\label{FGL0}
\min_{(X,Y)\in\mathbb{R}^{m\times d}\times\mathbb{R}^{n\times d}}F_0(X,Y):=f(XY^{\mathbb{T}})+\lambda\big(\text{nnzc}(X)+\text{nnzc}(Y)\big).
\end{equation}
Without loss of generality, we assume that {$m\geq n$} throughout this paper.
Note that both $f(XY^{\mathbb{T}})$ and $\text{nnzc}(X)+\text{nnzc}(Y)$ are nonconvex functions. It is not hard to verify the following equivalent transformations in the sense of global minimizers between problems (\ref{LRMPR}) and (\ref{FGL0}).
\begin{equation*}
\begin{split}
(E1) &\text{ If $(X,Y)$ is a global minimizer of (\ref{FGL0}), then $XY^{\mathbb{T}}$ is a global minimizer of (\ref{LRMPR}).}\\
(E2) &\text{ If $Z$ is a global minimizer of (\ref{LRMPR}), then $(X,Y)\in\mathbb{R}^{m\times d}\times\mathbb{R}^{n\times d}$ satisfying $XY^{\mathbb{T}}=$}\\
&\text{ $Z$ and ${\text{nnzc}}(X)={\text{nnzc}}(Y)=\text{rank}(Z)$ is a global minimizer of (\ref{FGL0}).}
\end{split}
\end{equation*}
This implies that the regularizer $\lambda\big(\text{nnzc}(X)+\text{nnzc}(Y)\big)$ in (\ref{FGL0}) plays the same role of rank reduction as $\lambda\,\text{rank}(Z)$ in (\ref{LRMPR}). Define $\mathcal{I}:\mathbb{R}\rightarrow\mathbb{R}$ by $\mathcal{I}(t)=1$ if $t\neq0$ and $\mathcal{I}(t)=0$ otherwise. For any matrix $A\in\mathbb{R}^{r\times d}$,
the  function $\text{nnzc}(A)$ can equivalently be replaced by its $\ell_{p,0}$ $(p\geq 0)$ norm, i.e. $\ell_{p,0}(A):=\sum_{i=1}^d\mathcal{I}(\|A_i\|_p)$, where $A_i$ is the $i$th column of $A$ and $\|\cdot\|_p$ is the vector $p$-norm. By adding suitable perturbative quadratic terms of $X$ and $Y$ to the objective function of (\ref{FGL0}) and replacing $\text{nnzc}$ by the $\ell_{2,0}$ norm, the corresponding problem is considered in \cite{Pan2022factor}, where the quadratic perturbation terms are indispensable throughout the theoretical analysis in that paper. Meanwhile, majorized alternating proximal algorithms with convergence to the limiting-critical point of this perturbed problem are proposed when the gradient of $f$ is Lipschitz continuous. Theoretically, there exists no direct equivalence result between the perturbed problem and the rank regularized problem (\ref{LRMPR}), even though one can consider the perturbed model as an approximation of problem (\ref{LRMPR}). Moreover, to the best of our knowledge, there is no research on model (\ref{FGL0}) including the analysis on its global minimizers and the relationships with problem (\ref{LRMPR}).

In many applications of matrix recovery, such as image restoration, score sheet completion and collaborative filtering (such as the well-known Netflix problem), it is usually accessible to obtain a positive upper bound for the norm of the target matrix $Z^\diamond$. The estimated prior bound can be imposed on the considered unconstrained optimization model \eqref{FGL0} as a constraint. In fact, the bound-constrained model may rule out some undesirable global minimizers of the unconstrained model, which implies that an appropriate bound constraint is likely to improve the likelihood of recovering the target matrix from the proposed model.
Thus, we make the reasonable assumption on the target matrix $Z^\diamond$ such that $\|Z^\diamond\|\leq\varsigma^2$ for an a priori given $\varsigma$. With the above assumption, it motivated us to improve (\ref{FGL0}) by considering the following discontinuous bound-constrained model:
\begin{equation}\label{FGL0C}
\min_{(X,Y)\in\Omega^1\times\Omega^2}F_0(X,Y),
\end{equation}
where $\Omega^1:=\{X\in\mathbb{R}^{m\times d}:\|X_i\|\leq\Mbound,i=1,\ldots,d\}$ and $\Omega^2:=\{Y\in\mathbb{R}^{n\times d}:\|Y_i\|\leq\Mbound,i=1,\ldots,d\}$ with $X_i$ and $Y_i$ being the $i$-th column of $X$ and $Y$. 
We will show that the target matrix $Z^\diamond$ belongs to $\{XY^{\mathbb{T}}:(X,Y)\in\Omega^1\times\Omega^2\}$ 
and analyze the relationships of the global minimizers and local minimizers between models (\ref{LRMPR}), (\ref{FGL0}) and (\ref{FGL0C}).

Recall the well-known $\ell_0$ norm of vector $x\in\mathbb{R}^n$, that is $\|x\|_0=\sum_{i=1}^{n}\mathcal{I}(x_i)$. Obviously, $\ell_0$ norm can be regarded as a special case of the $\ell_{p,0}$ $(p\geq 0)$ norm. There are many continuous relaxation functions to the $\ell_0$ norm such as capped-$\ell_1$ penalty \cite{Peleg2008}, smoothly clipped absolute deviation (SCAD) penalty \cite{Fan2001} and minimax concave penalty (MCP) \cite{Zhang2010}, which  can also be extended to the continuous relaxations of the $\ell_{p,0}$ $(p\geq 0)$ norm. For any given $\vartheta>0$ and $I\in\{1,2\}^d$, define
\begin{equation*}
\Theta_\vartheta(C):=\sum_{i=1}^d\theta_\vartheta(\|C_i\|)\mbox{ and }\Theta_{\vartheta,I}(C):=\sum_{i=1}^d\theta_{\vartheta,I_i}(\|C_i\|),~\forall C\in\mathbb{R}^{r\times d},
\end{equation*}
where $C_i$ is the $i$-th column of $C$, $\theta_\vartheta(t):=\min_{j\in\{1,2\}}\{\theta_{\vartheta,j}(t)\}$ with $\theta_{\vartheta,1}(t):=t/\vartheta$ and $\theta_{\vartheta,2}(t):=1$ for any $t\geq0$. This relaxation can be regarded as the capped-$\ell_1$ relaxation to the $\ell_{2,0}$ norm. By definition, it is easy to see that $\Theta_\vartheta(C) \leq \Theta_{\vartheta,I}(C)$. Next, we will further consider the corresponding continuous relaxation model of problem (\ref{FGL0C}) as follows:
\begin{equation}\label{FGL0R}
\min_{(X,Y)\in\Omega^1\times\Omega^2}F(X,Y):=f(XY^{\mathbb{T}})+\lambda\big(\Theta_\nu(X)+\Theta_\nu(Y)\big),
\end{equation}
where $\nu$ is a given positive constant. Due to the fact that $\lim_{\nu\downarrow0}(\Theta_\nu(X)+\Theta_\nu(Y))=\text{nnzc}(X)+\text{nnzc}(Y)$ and $\Theta_\nu(X)+\Theta_\nu(Y)\leq\text{nnzc}(X)+\text{nnzc}(Y),\,\forall(X,Y)\in\mathbb{R}^{m\times d}\times\mathbb{R}^{n\times d}$, problem (\ref{FGL0R}) can be viewed as an approximation of problem (\ref{FGL0C}) from below. We will establish the equivalence between (\ref{FGL0C}) and (\ref{FGL0R}) with suitable $\nu$ in this paper.

For some specific $\ell_{p,0}$ regularized problems with convex loss functions, there exist some equivalent relaxation theories \cite{Bian2020,LeThi2015,Li2022,Soubies2017}. Under some appropriate choices of the parameter in the relaxation function \cite{Bian2020}, $\ell_0$ regularized problem and its corresponding capped-$\ell_1$ relaxation problem own some equivalent relationships. Further, for the sparse group $\ell_0$ regularized problems with overlapping group sparsity, a class of stationary points to its capped-$\ell_1$ relaxation problem are defined in \cite{Li2022} and proved to be equivalent to some local minimizers of the considered $\ell_{1,0}$/$\ell_{2,0}$ regularized problem. These equivalent relaxations provide some theoretical basis for designing algorithms to obtain premium local minimizers for this class of $\ell_{p,0}$ regularized problems with convex loss functions. Considering the equivalence between the function $\text{nnzc}(\cdot)$ and $\ell_{p,0}$ norm, it is natural to ask the question: does problem (\ref{FGL0C}) own some equivalence with its relaxation (\ref{FGL0R})? Different from the convex loss function considered in \cite{Li2022}, the nonconvexity of $f(XY^{\mathbb{T}})$ is the main difficulty in the analysis and plays a crucial role in its applications. To the best of our knowledge, all the existing stationary points of (\ref{FGL0C}) and (\ref{FGL0R}) own weaker optimality property than their local minimizers in general. By exploring some properties of their global minimizers, we will strengthen the conditions of the existing stationary points for (\ref{FGL0C}) and (\ref{FGL0R}) to establish their equivalence. Moreover, the strengthened conditions may rule out non-global local minimizers of (\ref{FGL0C}) and (\ref{FGL0R}), respectively. These stationary points are called as strong stationary points. Besides, in terms of computation, the internal structure of $f(XY^{\mathbb{T}})$ also motivates us to further reduce the computational cost for the designed algorithms on the basis of variable dimension reduction from (\ref{LRMPR}) to (\ref{FGL0C}).

Both problems (\ref{FGL0C}) and (\ref{FGL0R}) belong to the class of nonconvex and nonsmooth problems. When $f$ is differentiable and its gradient is Lipschitz continuous, the classical proximal alternating linearized minimization (PALM) algorithm \cite{Bolte2014Proximal}, its inertial version (iPALM) \cite{SabachSIIMS2016} and block prox-linear (BPL) algorithm \cite{YinJSC2017} can be applied to solve problems (\ref{FGL0C}) and (\ref{FGL0R}), because the proximal operators of the regularizers can be calculated effectively. But in theory, they only has the convergence to the limiting-critical point sets of (\ref{FGL0C}) and (\ref{FGL0R}), respectively. Moreover, a recently proposed DC (Difference of Convex functions) algorithm \cite{LeThi2021arxiv} is also able to solve an equivalent DC problem of the relaxation problem (\ref{FGL0R}), but it only owns the convergence to the critical point set of the DC problem. Notice that both the above mentioned limiting-critical point and critical point (of the equivalent DC relaxation problem) own weaker optimality conditions than our defined strong stationary points in this paper. Next, in order to obtain the strong stationary points of \eqref{FGL0C}, we design two algorithms where the first attempts to solve \eqref{FGL0C} through its relaxation reformulation and 
the second to solve  \eqref{FGL0C} directly. The good performance, similarities and differences of the proposed two algorithms will be analyzed in detail.

\medskip

The main contributions of this paper are summarized as follows.
\begin{itemize}
\item {\bf{Enhance (\ref{FGL0}) to (\ref{FGL0C}) and analyze the necessary and sufficient conditions for the local minimizers of \eqref{LRMPR} and \eqref{FGL0C}.}} We enhance model \eqref{FGL0} to \eqref{FGL0C} by adding a prior bound on the target matrix as a constraint, and further analyze the relationships on the global minimizers and local minimizers between problems (\ref{LRMPR}), (\ref{FGL0}) and (\ref{FGL0C}), as shown in Fig.\,\ref{probs-rela}. In addition, for the local minimizers of problems (\ref{LRMPR}) and (\ref{FGL0C}), we establish their necessary and sufficient conditions.
	
\item {\bf{Analyze and strengthen the optimality conditions for problem (\ref{FGL0C}) and its continuous relaxation (\ref{FGL0R}).}} We explore some properties of the global minimizers and incorporate them into the definitions of  strong stationary points for problem (\ref{FGL0C}) and its relaxation (\ref{FGL0R}). Notably, these properties are not necessary conditions for their local minimizers. Furthermore, in terms of the optimality condition, the defined strong stationary points are stronger than the existing stationary points for both problems (\ref{FGL0C}) and (\ref{FGL0R}). In particular, we establish the equivalence between (\ref{FGL0C}) and (\ref{FGL0R}) in the sense of their global minimizers and the defined strong stationary points, respectively.
	
\item {\bf{Propose algorithms to solve problem (\ref{FGL0C}) (as well as (\ref{FGL0R})) from two perspectives.}} First, based on the relaxation of (\ref{FGL0C}) given by (\ref{FGL0R}), we design an alternating proximal gradient algorithm with adaptive indicator and prove its convergence to the strong stationary point set of (\ref{FGL0C}). After finite iterations, the iterates generated by the proposed algorithm own the derived properties of global minimizers of (\ref{FGL0C}) and their column sparsity no longer changes. When $f$ is semialgebraic, the convergence rate of the iterates is at least R-sublinear. Second, in terms of directly solving problem (\ref{FGL0C}) without relaxation, we make some improvements to the PALM algorithm in \cite{Bolte2014Proximal} and give some further progress on the convergence of PALM, which improves the existing algorithms on the attainable solution quality for problem \eqref{FGL0C}. In particular, we also show the difference on the convergence results between the two proposed algorithms.
	
\item {\bf{Reduce the computational cost from both modelling and algorithmic aspects.}} 
Based on the matrix factorization technique, the 
problem (\ref{FGL0C}) can be of much smaller scale than the rank regularized problem (\ref{LRMPR}). By taking into account of the sparsity of the iterates, we provide an adaptive dimension reduction method to reduce the computational cost of the proposed algorithms and maintain the same theoretical results for problem (\ref{FGL0C}). Moreover, we propose an adaptive tuning strategy on the regularization parameter $\lambda$ for (\ref{FGL0C}) and show its effectiveness in some numerical experiments.
\end{itemize}
\begin{figure}
\centering
\includegraphics[width=4.5in]{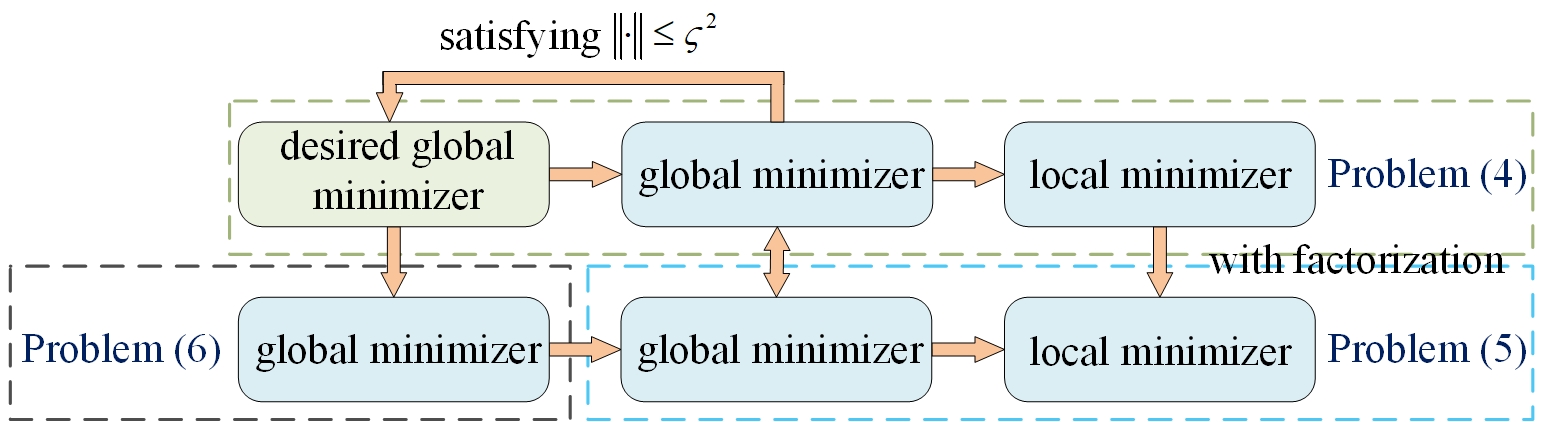}
\caption{Relationships between problems (\ref{LRMPR}), (\ref{FGL0}) and (\ref{FGL0C})}\label{probs-rela}
\end{figure}

The remaining part of this paper is organized as follows. In Section \ref{section2}, we introduce some necessary notations, definitions and preliminaries. In Section \ref{section3}, we analyze the equivalence and inclusion relationships on the global mimizers and local minimizers for problems (\ref{LRMPR}), (\ref{FGL0}) and (\ref{FGL0C}). In Section \ref{section4}, we analyze some properties of the global minimizers, local minimizers and our defined strong stationary points for problem (\ref{FGL0C}) and its relaxation problem (\ref{FGL0R}). Moreover, we establish their equivalence in the sense of global minimizers and  strong stationary points, respectively. In Section \ref{section5}, we propose two algorithms to solve (\ref{FGL0C}) as well as (\ref{FGL0R}), and give some convergence analysis. Finally, we present some numerical 
experiments to verify the theoretical results in Section \ref{section6} and show the good performance of the proposed algorithms for the low-rank matrix completion problem in
the form of \eqref{FGL0C}.

\section{Notations and Preliminaries}\label{section2}
In this section, we introduce some notations, definitions and preliminaries that will be used in 
this paper. 
\subsection{Notations}
For any positive integer $d$, let $[d]:=\{1,\ldots,d\}$. We use $\mathbb{N}$ to denote the set of all nonnegative integers. Let $\mathbb{R}^{m\times n}$ be the space of 
$m\times n$ real matrices endowed with the trace inner product $\langle X,Y\rangle={\rm trace}(X^{\mathbb{T}}Y)$ and its induced Frobenius norm $\|\cdot\|_F$. Denote $E$ as an identity matrix. For $t\in\mathbb{R}$ and $\xi\in\mathbb{R}^d$, denote $\lceil t\rceil$ as the rounded integer of $t$, $\bm{t}$ as the matrix of all $t$ and $\xi^\downarrow$ as the vector in $\mathbb{R}^d$ consisting of $\xi_i,i\in[d]$ arranged in a nonincreasing order. For $C\in\mathbb{R}^{r\times d}$, denote its spectral norm as $\|C\|$, $\sigma_i(C)$ the $i$-th largest singular value of $C$, $C_i\in\mathbb{R}^r$ the $i$-th column of $C$, $\mathcal{I}_{C}:=\{i\in[d]:C_i\neq\bm{0}\}$, ${B}_{\varepsilon}(C):=\{Z\in\mathbb{R}^{r\times d}:\|Z-C\|_F\leq\varepsilon\}$ and ${B}_{\varepsilon}:={B}_{\varepsilon}(\bm{0})$. For $(A,B)\in\mathbb{R}^{m\times d}\times\mathbb{R}^{n\times d}$ and $\varepsilon>0$, denote ${B}_{\varepsilon}(A,B):=\{(X,Y)\in\mathbb{R}^{m\times d}\times\mathbb{R}^{n\times d}:\|[X,Y]-[A,B]\|_F\leq\varepsilon\}$. 
For $C\in\mathbb{R}^{r\times d}$, $\mathcal{C}\subseteq\mathbb{R}^{r\times d}$ and $J\subseteq[d]$, $C_J:=\big[C_{j_1},\ldots,C_{j_{|J|}}\big]\in\mathbb{R}^{r\times|J|}$, $C_J^{\mathbb{T}}:=(C_J)^{\mathbb{T}}$ and $J^c=[d]\setminus{J}$, where $|J|$ is the cardinality of $J$, $j_i\in J$, $\forall i\in[|J|]$ with $j_1<\ldots<j_{|J|}$, and $\mathcal{C}_J:=\{Z_J:Z\in\mathcal{C}\}$. For a nonempty closed convex set $\cS\subseteq\mathbb{R}^{m\times n}$, its interior is defined by ${\rm{int}}\cS$, and the indicator function $\delta_\cS$ is defined by $\delta_\cS(Z)=0$ if $Z\in \cS$ and $\delta_\cS(Z)=\infty$ otherwise. 
The normal cone to $\cS$ at $Z$ is defined by $N_\cS(Z):=\{U\in\mathbb{R}^{m\times n}:\langle U,\bar{Z}-Z\rangle\leq 0,\forall \bar{Z}\in \cS\}$, and 
the projection operator to $\cS$ at $Z$ is defined by 
$\text{proj}_{\cS}(Z):={\rm{argmin}}_{X\in \cS}\|X-Z\|_F^2$. For a proper lower semicontinuous function $h:\mathbb{R}^{m\times n}\rightarrow\overline{\mathbb{R}}:=(-\infty,\infty]$, the proximal operator of $h$ is defined by ${\rm{prox}}_h(Z):={\rm{argmin}}_{X\in\mathbb{R}^{m\times n}}\{h(X)+\frac{1}{2}\|X-Z\|_F^2\}$. For a linear operator $\mathcal{A}:\mathbb{R}^{m\times n}\rightarrow\mathbb{R}^p$, denote its adjoint as $\mathcal{A}^*$. 
Denote $\partial f(\cdot)$ the limiting subdifferential of $f$ and $\partial_Zf(XY^{\mathbb{T}}):=\partial f(Z)|_{Z=XY^{\mathbb{T}}}$. Denote the sets of global minimizers and local minimizers as $\mathcal{G}_{(\vartriangle)}$ and $\mathcal{L}_{(\vartriangle)}$ for any problem ($\vartriangle$). 
The abbreviation \emph{iff} denotes ``if and only if''.
\subsection{Preliminaries}
First, recall the definitions of two classes of subdifferentials for a proper function, where we call an extended real-valued function $h: \mathbb{R}^{m\times n}\rightarrow\overline{\mathbb{R}}$ proper if its domain ${\rm dom}\, h:=\{X\in\mathbb{R}^{m\times n}: h(X)<\infty\}\neq \emptyset$. 
\begin{definition}[\bf Subdifferentials]\cite{Rockafellar1998,Mordukhovich2006}\label{subdefs}
For a proper function $h:\mathbb{R}^{m\times n}\rightarrow\overline{\mathbb{R}}$, the regular (Fr\'echet) subdifferential of $h$ at $X\in {\rm dom}\, h$ is defined by\vspace{-0.1cm}
\[\vspace{-0.1cm}
\widehat{\partial} h(X)=\Big\{\zeta\in\mathbb{R}^{m\times n} :\; \liminf_{Z\rightarrow X,Z\neq X}\frac{h(Z)-h(X)-\left<\zeta,Z-X\right>}{\|Z-X\|_F}\ge 0 \Big\},
\]
and the limiting-subdifferential of $h$ at $X\in {\rm dom}\, h$  is defined by\vspace{-0.1cm}
\[\vspace{-0.1cm}
\partial h(X)=\Big\{\zeta\in\mathbb{R}^{m\times n} :\; \exists\, X^{k}\rightarrow X, h(X^{k})\rightarrow h(X) \mbox{ and } \widehat{\partial} h(X^k)\ni\zeta^{k}\rightarrow \zeta \mbox{ as } k\rightarrow\infty\Big\}.
\]
\end{definition}

For a proper function $h:\mathbb{R}^{m\times n}\rightarrow\overline{\mathbb{R}}$ and point $X\in{\rm dom}\, h$, it holds that $\widehat{\partial}h(X)\subseteq\partial h(X)$ by \cite[Theorem 8.6]{Rockafellar1998}, $\widehat{\partial}h(X)=\{\nabla h(X)\}$ if $h$ is differentiable at $X$, and $\partial h(X)=\{\nabla h(X)\}$ if $h$ is smooth near $X$ by \cite[Exercise 8.8]{Rockafellar1998}. Moreover, if $h$ is convex, it follows from \cite[Proposition 8.12]{Rockafellar1998} that $\widehat{\partial}h(X)$ and $\partial h(X)$ coincide with the subdifferential of $h$ in the sense of convex analysis.

Set $\partial h(X)=\widehat\partial h(X)=\emptyset$ if $X\notin {\rm dom}\, h$ and ${\rm dom}\,\partial h:= \{X\in\mathbb{R}^{m\times n}:\partial h(X)\neq\emptyset\}$.

Next, recall the definition of Kurdyka-{\L}ojasiewicz (KL) property in \cite{Attouch2010MOR}, which is adopted extensively for the convergence analysis of the whole sequence generated by first-order algorithms.
\begin{definition}[\bf Kurdyka-{\L}ojasiewicz property]\cite{Attouch2010MOR}\label{KLprop}
For a proper lower semicontinuous function $h:\mathbb{R}^{m\times n}\rightarrow\overline{\mathbb{R}}$, we say that $h$ has the Kurdyka-{\L}ojasiewicz (KL) property at $\hat{Z}\in{\rm dom}\,\partial h$, if there exist $\delta\in(0,\infty]$, a continuous concave function $\varphi:[0,\delta)\rightarrow[0,\infty)$ and a neighborhood $\mathcal{V}$ of $\hat{Z}$ such that
\begin{itemize}
\item[(i)] $\varphi(0)=0$, $\varphi$ is smooth and $\varphi'>0$ on $(0,\delta)$;
\item[(ii)] for any $Z\in\mathcal{V}$ with $h(\hat{Z})<h(Z)<h(\hat{Z})+\delta$, it holds 
\begin{equation*}
\varphi'(h(Z)-h(\hat{Z})){\rm{dist}}(\bm{0},\partial h(Z))\geq 1.
\end{equation*}
\end{itemize}
\end{definition}
If $h$ has the KL property at $\hat{Z}\in{\rm{dom}}\partial h$ with $\varphi(t)=\kappa_0 t^{1-\alpha}$ for some $\kappa_0>0$ and $\alpha\in[0,1)$ in Definition \ref{KLprop}, we say that $h$ has the KL property at $\hat{Z}$ with exponent $\alpha$. If $h$ has the KL property at every point of $\mbox{dom}\partial{h}$ (with exponent $\alpha$), we say that $h$ is a KL function (with exponent $\alpha$). Note that there are many KL functions in applications, such as semi-algebraic, subanalytic and log-exp functions \cite{Attouch2010MOR,Attouch2013}.

Next, we recall a frequently used class of stationary points for  nonsmooth optimization problems.
\begin{definition}[\bf Limiting-critical point]\cite{Attouch2010MOR,Bolte2014Proximal}\label{lim-cri-def}
For a proper function $h:\mathbb{R}^{m\times n}\rightarrow\overline{\mathbb{R}}$ and a closed convex set $\Omega\subseteq\mathbb{R}^{m\times n}$, we call $X\in\mathbb{R}^{m\times n}$ satisfying $\bm{0}\in\partial\big(h(X)+\delta_\Omega(X)\big)$ the limiting-critical point of the problem $\min_{X\in\Omega}h(X)$.
\end{definition}
Given a proper function $h$, according to the Fermat's rule in \cite[Theorem 10.1]{Rockafellar1998}, any local minimizer is a limiting-critical point of the problem $\min_{X}h(X)$.

For the convenience of subsequent analysis, we now give the characterization of subdifferentials for $\text{nnzc}+\delta_{\Omega^1\times\Omega^2}$ and an equivalent formulation for the limiting-critical points of problem (\ref{FGL0C}).
\begin{proposition}\label{subd-cri-rela}
Let $h_1:\mathbb{R}^{m\times d}\rightarrow\overline{\mathbb{R}}$ and $h_2:\mathbb{R}^{n\times d}\rightarrow\overline{\mathbb{R}}$ be defined by $h_1:=\lambda\,\emph{\text{nnzc}}+\delta_{\Omega^1}$ and $h_2:=\lambda\,\emph{\text{nnzc}}+\delta_{\Omega^2}$. For any $(X,Y)\in\Omega^1\times \Omega^2$, it holds that
\begin{itemize}
\item[(i)] $h_1$ and $h_2$ are proper and lower semicontinuous, $\widehat\partial\big(h_1(X)+h_2(Y)\big)=\partial\big(h_1(X)+h_2(Y)\big)=\partial_Xh_1(X)\times\partial_Yh_2(Y)$, $\widehat\partial_Xh_1(X)=\partial_Xh_1(X)=\big\{U\in\mathbb{R}^{m\times d}:U_{\mathcal{I}_X}\in N_{\Omega^1_{\mathcal{I}_X}}(X_{\mathcal{I}_X})\big\}$ and $\widehat\partial_Yh_2(Y)=\partial_Yh_2(Y)=\big\{U\in\mathbb{R}^{n\times d}:U_{\mathcal{I}_Y}\in N_{\Omega^2_{\mathcal{I}_Y}}(Y_{\mathcal{I}_Y})\big\}$;
\item[(ii)] when $f$ is smooth, $(X,Y)$ is a limiting-critical point of (\ref{FGL0C}) iff it satisfies
\begin{equation}\label{cri-incl}
\bm{0}\in \nabla_Zf(XY^{\mathbb{T}})Y_{\mathcal{I}}+N_{\Omega^1_{\mathcal{I}}}(X_{\mathcal{I}}) \;\mbox{and}\; \bm{0}\in \nabla_Zf(XY^{\mathbb{T}})^{\mathbb{T}}X_{\mathcal{I}}+N_{\Omega^2_{\mathcal{I}}}(Y_{\mathcal{I}})\mbox{ with }\mathcal{I}=\mathcal{I}_{X}\cap\mathcal{I}_{Y}.
\end{equation}
\end{itemize}
\end{proposition}
\begin{proof}
(i) Due to the lower semicontinuity of $\text{nnzc}$ and the nonempty closedness of the sets $\Omega^1$ and $\Omega^2$, it holds that $h_1$ and $h_2$ are proper and lower semicontinuous.
	
Define $g_{q,\alpha}:\mathbb{R}^q\rightarrow\mathbb{R}$ by $g_{q,\alpha}(v)=\lambda\,\text{nnzc}(v)+\delta_{B_{\alpha}}(v)$ with $\alpha>0$. Based on Definition \ref{subdefs}, we deduce that for any $v\in B_{\alpha}$, $\widehat\partial g_{q,\alpha}(v)=\widehat\partial\delta_{B_{\alpha}}(v)=N_{B_{\alpha}}(v)$ if $v\neq\bm{0}$, and $\widehat\partial g_{q,\alpha}(v)=\widehat\partial\lambda\,\text{nnzc}(v)=\mathbb{R}^q$ if $v=\bm{0}$. Further, it follows from \cite[Proposition 8.5]{Rockafellar1998} that $\widehat\partial_X h_1(X)=\Pi_{i=1}^d\widehat\partial g_{m,\Mbound}(X_i)=\big\{U\in\mathbb{R}^{m\times d}:U_{\mathcal{I}_X}\in N_{\Omega^1_{\mathcal{I}_X}}(X_{\mathcal{I}_X})\big\}$, $\widehat\partial_Y h_2(Y)=\Pi_{i=1}^d\widehat\partial g_{n,\Mbound}(Y_i)=\big\{U\in\mathbb{R}^{n\times d}:U_{\mathcal{I}_Y}\in N_{\Omega^2_{\mathcal{I}_Y}}(Y_{\mathcal{I}_Y})\big\}$ and $\widehat\partial (h_1(X)+h_2(Y))=\widehat\partial_X h_1(X)\times\widehat\partial_Y h_2(Y)$.
	
Considering that there exists an $\varepsilon>0$ such that $\mathcal{I}_X\subseteq\mathcal{I}_{U},\forall U\in B_\varepsilon(X)\cap\Omega^1$, by the closedness of $N_{\Omega^1_{\mathcal{I}_X}}$ in \cite[Lemma 2.42]{Andrzej2006nonlinear} and Definition \ref{subdefs}, we get $\partial_Xh_1(X)\subseteq\widehat\partial_Xh_1(X)$. Further, by \cite[Theorem 8.6]{Rockafellar1998}, it holds $\partial_Xh_1(X)\supseteq\widehat\partial_Xh_1(X)$ and hence $\partial_Xh_1(X)=\widehat\partial_Xh_1(X)$. Similarly, we deduce that $\partial_Yh_2(Y)=\widehat\partial_Yh_2(Y)$ and $\partial\big(h_1(X)+h_2(Y)\big)=\widehat\partial\big(h_1(X)+h_2(Y)\big)$.
	
(ii) Based on Definition \ref{lim-cri-def}, the smoothness of $f$, \cite[Exercise 8.8]{Rockafellar1998} and (i), it follows that $(X,Y)$ is a limiting-critical point of (\ref{FGL0C}), that is $\bm{0}\in\partial \big(F_0(X,Y)+\delta_{\Omega^1}(X)+\delta_{\Omega^2}(Y)\big)$ \emph{iff}
\begin{equation}\label{equv_sim}
\bm{0}\in\nabla_Zf(XY^{\mathbb{T}})Y+\partial_X h_1(X)\text{ and }\bm{0}\in\nabla_Zf(XY^{\mathbb{T}})^{\mathbb{T}}X+\partial_Y h_2(Y). 
\end{equation}
	
Next, we prove that (\ref{equv_sim}) is equivalent to (\ref{cri-incl}). First, suppose $(X,Y)$ satisfies (\ref{equv_sim}). It follows from (i) that $(X,Y)$ fulfills (\ref{cri-incl}). Conversely, consider that $(X,Y)$ satisfies (\ref{cri-incl}). Based on $Y_{\mathcal{I}_X\setminus\mathcal{I}}=\bm{0}$, $X_{\mathcal{I}_Y\setminus\mathcal{I}}=\bm{0}$ and (i), we deduce that $\bm{0}\in[\nabla_Zf(XY^{\mathbb{T}})Y+\partial_X h_1(X)]_{[d]\setminus\mathcal{I}}$ and $\bm{0}\in[\nabla_Zf(XY^{\mathbb{T}})^{\mathbb{T}}X+\partial_Y h_2(Y)]_{[d]\setminus\mathcal{I}}$. This together with (\ref{cri-incl}) yields (\ref{equv_sim}). In conclusion, the statement (ii) holds.
\end{proof}

\section{Analysis and Relationships Among Models  (\ref{LRMPR}), (\ref{FGL0}) and (\ref{FGL0C})}\label{section3}
Based on (E1) and (E2) in the Introduction, we have the equivalent transformation on the global minimizers between problems (\ref{LRMPR}) and (\ref{FGL0}). In this section, we further analyze the relationships between problems (\ref{LRMPR}), (\ref{FGL0}), and (\ref{FGL0C}). Moreover, we give an equivalent formulation for the local minimizers of problem (\ref{LRMPR}), and show the inclusion relation on the local minimizers between problems (\ref{LRMPR}) and (\ref{FGL0}).

\medskip

First, we show a basic property for those matrices satisfying the same bound $\varsigma^2$ 
as the target matrix $Z^\diamond$.
\begin{proposition}\label{rank_const_trans}
For any $Z\in\mathbb{R}^{m\times n}$ with $\emph{\text{rank}}(Z)\leq d$ and $\|Z\|\leq\varsigma^2$, there exists $(\check{X},\check{Y})\in\Omega^1\times\Omega^2$ satisfying $\check{X}\check{Y}^{\mathbb{T}}=Z$ and $\emph{\text{nnzc}}(\check{X})=\emph{\text{nnzc}}(\check{Y})=\emph{\text{rank}}(Z)$, and hence $\emph{\text{rank}}(Z)=
\min\big\{
\frac{1}{2}\big(\emph{\text{nnzc}}(X)+\emph{\text{nnzc}}(Y)\big) \mid {(X,Y)\in\Omega^1\times\Omega^2:XY^{\mathbb{T}}=Z}\big\}$.
\end{proposition}
\begin{proof}
For $Z\in\mathbb{R}^{m\times n}$ as supposition, by its singular value decomposition, there exist two orthogonal matrices $U\in\mathbb{R}^{m\times m}$ and $V\in\mathbb{R}^{n\times n}$ such that $Z=U\Sigma V^\mathbb{T}$, where $\Sigma=[{\rm diag}(\sigma),0_{n\times (m-n)}]^{\mathbb{T}}\in\mathbb{R}^{m\times n}$ with $\sigma\in \mathbb{R}^m$ satisfying $\sigma_i \geq \sigma_j\,\forall i < j$ and 
\begin{equation}\label{elem_p}
\sigma_i \in(0,\|Z\|],\forall i\in[\text{rank}(Z)], \quad \sigma_i=0,\forall i\notin[\text{rank}(Z)].
\end{equation}
Let $\Sigma^{1/2}\in\mathbb{R}^{m\times n}$ be generated by $\Sigma^{1/2}_{ij}=\sqrt{\Sigma_{ij}}$, $\forall i\in[m],j\in[n]$, $\check{X}=(U\Sigma^{1/2})_{[d]}\in\mathbb{R}^{m\times d}$ and $\check{Y}=(V(\Sigma^{1/2})^\mathbb{T})_{[d]}\in\mathbb{R}^{n\times d}$. It follows from $\text{rank}(Z)\leq d$ that $Z=\check{X}\check{Y}^{\mathbb{T}}$. Moreover, since $U$ and $V$ are orthogonal, $\Sigma$ satisfies (\ref{elem_p}) and $\|Z\|\leq\varsigma^2$, we deduce $(\check{X},\check{Y})\in\Omega^1\times\Omega^2$ and $\text{nnzc}(\check{X})=\text{nnzc}(\check{Y})=\text{rank}(Z)$. Combining with $\text{rank}(XY^{\mathbb{T}})\leq\frac{1}{2}\big(\text{nnzc}(X)+\text{nnzc}(Y)\big),\forall X,Y$, we get $\text{rank}(Z)=\min_{(X,Y)\in\Omega^1\times\Omega^2:XY^{\mathbb{T}}=Z}\frac{1}{2}(\text{nnzc}(X)+\text{nnzc}(Y))$.
\end{proof}

Further, we show some properties of the global minimizers of problem (\ref{FGL0}).
\begin{proposition}\label{FGL0-prps}
Let $(\bar{X},\bar{Y})\in\mathcal{G}_{(\ref{FGL0})}$. The following properties hold for problem (\ref{FGL0}).
\begin{itemize}
\item[(i)] $\mathcal{I}_{\bar{X}}=\mathcal{I}_{\bar{Y}}$ and $\emph{\text{nnzc}}(\bar{X})=\emph{\text{nnzc}}(\bar{Y})=\emph{\text{rank}}(\bar{X}\bar{Y}^{\mathbb{T}})$.
\item[(ii)] $\{(\gamma \bar{X},\frac{1}{\gamma}\bar{Y}):\gamma\neq 0\}\subseteq\mathcal{G}_{(\ref{FGL0})}$.
\item[(iii)] If $\|\bar{X}\bar{Y}^{\mathbb{T}}\|\leq\varsigma^2$, then there exists $(\check{X},\check{Y})\in\mathcal{G}_{(\ref{FGL0})}$ such that $\check{X}\check{Y}^{\mathbb{T}}=\bar{X}\bar{Y}^{\mathbb{T}}$ and $(\check{X},\check{Y})\in\Omega^1\times\Omega^2$.
\end{itemize}
\end{proposition}
\begin{proof}
(i) We prove this statement by contradiction. Suppose $\mathcal{I}_{\bar{X}}\neq\mathcal{I}_{\bar{Y}}$. Let $\bar{\mathcal{I}}=\mathcal{I}_{\bar{X}}\cap\mathcal{I}_{\bar{Y}}$, and $(\tilde{X},\tilde{Y})\in\mathbb{R}^{m\times d}\times\mathbb{R}^{n\times d}$ satisfy $\tilde{X}_{\bar{\mathcal{I}}}=\bar{X}_{\bar{\mathcal{I}}},\tilde{X}_{\bar{\mathcal{I}}^c}=\bm{0}$ and $\tilde{Y}_{\bar{\mathcal{I}}}=\bar{Y}_{\bar{\mathcal{I}}},\tilde{Y}_{\bar{\mathcal{I}}^c}=\bm{0}$. It follows that $\bar{X}\bar{Y}^{\mathbb{T}}=\tilde{X}\tilde{Y}^{\mathbb{T}}$, $\bar{\mathcal{I}}=\mathcal{I}_{\tilde{X}}=\mathcal{I}_{\tilde{Y}}$ and $\text{nnzc}(\tilde{X})+\text{nnzc}(\tilde{Y})<\text{nnzc}(\bar{X})+\text{nnzc}(\bar{Y})$. Hence, $F_0(\tilde{X},\tilde{Y})<F_0(\bar{X},\bar{Y})$, which contradicts $(\bar{X},\bar{Y})\in\mathcal{G}_{(\ref{FGL0})}$. Therefore, $\mathcal{I}_{\bar{X}}=\mathcal{I}_{\bar{Y}}$.
	
Let $\Gamma = \{ (X,Y) : XY^{\mathbb{T}}=\bar{X}\bar{Y}^{\mathbb{T}}\}$. Due to $(\bar{X},\bar{Y})\in\mathcal{G}_{(\ref{FGL0})}$, we have $F_0(\bar{X},\bar{Y})\leq\min_{(X,Y) \in \Gamma}F_0(X,Y)$, which implies that $\text{nnzc}(\bar{X})+\text{nnzc}(\bar{Y})=\min_{(X,Y)\in \Gamma}\big\{\text{nnzc}(X)+\text{nnzc}(Y)\big\}$. By $\mathcal{I}_{\bar{X}}=\mathcal{I}_{\bar{Y}}$, we get $\text{nnzc}(\bar{X})=\text{nnzc}(\bar{Y})=\frac{1}{2}\min_{(X,Y)\in \Gamma}\big\{\text{nnzc}(X)+\text{nnzc}(Y)\big\}$. Further, by (\ref{rank-rel}), we have $\text{nnzc}(\bar{X})=\text{nnzc}(\bar{Y})=\text{rank}(\bar{X}\bar{Y}^{\mathbb{T}})$.
	
(ii) In view of $\gamma\bar{X}\frac{1}{\gamma}\bar{Y}^{\mathbb{T}}=\bar{X}\bar{Y}^{\mathbb{T}}$, $\mathcal{I}_{\bar{X}}=\mathcal{I}_{\gamma\bar{X}}$ and $\mathcal{I}_{\bar{Y}}=\mathcal{I}_{\frac{1}{\gamma}\bar{Y}}$ for any $\gamma\neq0$, we deduce that $(\gamma\bar{X},\frac{1}{\gamma}\bar{Y})\in\mathcal{G}_{(\ref{FGL0})},\forall\gamma\neq 0$ by $(\bar{X},\bar{Y})\in\mathcal{G}_{(\ref{FGL0})}$.
	
(iii) Considering $\text{rank}(\bar{X}\bar{Y}^{\mathbb{T}})\leq d$ and $\|\bar{X}\bar{Y}^{\mathbb{T}}\|\leq\varsigma^2$, by Proposition \ref{rank_const_trans}, we obtain that there exists $(\check{X},\check{Y})\in\Omega^1\times\Omega^2$ satisfying $\check{X}\check{Y}^{\mathbb{T}}=\bar{X}\bar{Y}^{\mathbb{T}}$ and $\text{nnzc}(\check{X})=\text{nnzc}(\check{Y})=\text{rank}(\bar{X}\bar{Y}^{\mathbb{T}})$. Combining with (i), it holds that $\text{nnzc}(\check{X})=\text{nnzc}(\check{Y})=\text{nnzc}(\bar{X})=\text{nnzc}(\bar{Y})$. Then, by $\check{X}\check{Y}^{\mathbb{T}}=\bar{X}\bar{Y}^{\mathbb{T}}$ and $(\bar{X},\bar{Y})\in\mathcal{G}_{(\ref{FGL0})}$, we deduce that $(\check{X},\check{Y})\in\mathcal{G}_{(\ref{FGL0})}$.
\end{proof}

Although problem (\ref{LRMPR}) and problem (\ref{FGL0}) own the same optimal value, it is worth noting from Proposition \ref{FGL0-prps} (ii) that problem (\ref{FGL0}) has uncountable number of global minimizers, even if problem (\ref{LRMPR}) has a unique global minimizer. This motivates us to consider the bound-constrained problem (\ref{FGL0C}) and prove that the desired global miminizer $Z^\star$ of (\ref{LRMPR}) with $\|Z^\star\|\leq\varsigma^2$ corresponds to a global minimizer of problem (\ref{FGL0C}). Next, we analyze the relationships on the global minimizers between problems (\ref{LRMPR}), (\ref{FGL0}) and (\ref{FGL0C}).

\begin{proposition}\label{bd-rela}
Denote $\mathcal{G}_{(\ref{LRMPR})}^\star:=\{Z\in\mathcal{G}_{(\ref{LRMPR})}:\|Z\|\leq\varsigma^2\}$. The following statements hold.
\begin{itemize}
\item[(i)] $\{XY^{\mathbb{T}}:(X,Y)\in\mathcal{G}_{(\ref{FGL0})}\}=\mathcal{G}_{(\ref{LRMPR})}$.
\item[(ii)] If $Z^\star\in\mathcal{G}_{(\ref{LRMPR})}^\star$, then there exists $(X^\star,Y^\star)\in\mathcal{G}_{(\ref{FGL0})}$ such that $X^\star {Y^\star}^{\mathbb{T}}=Z^\star$ and $(X^\star,Y^\star)\in\Omega^1\times\Omega^2$.
\item[(iii)] If $\mathcal{G}_{(\ref{LRMPR})}^\star\neq\emptyset$, then $\mathcal{G}_{(\ref{FGL0C})}\subseteq\mathcal{G}_{(\ref{FGL0})}$ and $\mathcal{G}_{(\ref{LRMPR})}^\star\subseteq\{XY^{\mathbb{T}}:(X,Y)\in\mathcal{G}_{(\ref{FGL0C})}\}\subseteq\mathcal{G}_{(\ref{LRMPR})}$. Further, $\{XY^{\mathbb{T}}:(X,Y)\in\mathcal{G}_{(\ref{FGL0C})}\}=\mathcal{G}_{(\ref{LRMPR})}$ when $\mathcal{G}_{(\ref{LRMPR})}^\star=\mathcal{G}_{(\ref{LRMPR})}$.
\end{itemize}
\end{proposition}
\begin{proof}
(i) In view of the relationships in (E1) and (E2), we get $\{XY^{\mathbb{T}}:(X,Y)\in\mathcal{G}_{(\ref{FGL0})}\}=\mathcal{G}_{(\ref{LRMPR})}$.
	
(ii) Let $Z^\star\in\mathcal{G}_{(\ref{LRMPR})}^\star$. By (E2), for any $(\bar{X},\bar{Y})\in\mathbb{R}^{m\times d}\times\mathbb{R}^{n\times d}$ with $\bar{X}\bar{Y}^{\mathbb{T}}=Z^\star$ and ${\text{nnzc}}(\bar{X})={\text{nnzc}}(\bar{Y})=\text{rank}(Z^\star)$, we have $(\bar{X},\bar{Y})\in\mathcal{G}_{(\ref{FGL0})}$ and $\|\bar{X}\bar{Y}^{\mathbb{T}}\|\leq\varsigma^2$.
Further, by Proposition \ref{FGL0-prps} (iii), the results hold.
	
(iii) Due to $\mathcal{G}_{(\ref{LRMPR})}^\star\neq\emptyset$, we arbitrarily choose $Z^\star\in\mathcal{G}_{(\ref{LRMPR})}^\star$. By (ii), there exists $(X^\star,Y^\star)\in\mathcal{G}_{(\ref{FGL0})}$ such that $X^\star {Y^\star}^{\mathbb{T}}=Z^\star$ and $(X^\star,Y^\star)\in\Omega^1\times\Omega^2$. This implies $(X^\star,Y^\star)\in\mathcal{G}_{(\ref{FGL0C})}$ and
\begin{equation*}
\min_{(X,Y)\in\mathbb{R}^{m\times d}\times\mathbb{R}^{n\times d}}F_0(X,Y)=\min_{(X,Y)\in\Omega^1\times\Omega^2}F_0(X,Y).
\end{equation*}
Then, we deduce $\mathcal{G}_{(\ref{FGL0C})}\subseteq\mathcal{G}_{(\ref{FGL0})}$. Further, combining with (i), we obtain $Z^\star=X^\star {Y^\star}^{\mathbb{T}}\in\{XY^{\mathbb{T}}:(X,Y)\in\mathcal{G}_{(\ref{FGL0C})}\}\subseteq\mathcal{G}_{(\ref{LRMPR})}$. Therefore, the statements in (iii) hold.
\end{proof}

It is reasonable to expect that there exists at least a global minimizer of problem (\ref{LRMPR}) satisfying the upper bound $\varsigma^2$ of the target matrix $Z^\diamond$ when $\varsigma$ is taken to be sufficiently large. However,
if all global minimizers of (\ref{LRMPR}) do not satisfy the upper bound $\varsigma^2$, then
it is not meaningful to solve (\ref{LRMPR}) for the recovery or estimation of $Z^\diamond$. Note that if the bound $\varsigma^2$ holds for all global minimizers of (\ref{LRMPR}), then $\mathcal{G}_{(\ref{LRMPR})}=\{XY^\mathbb{T}:(X,Y)\in\mathcal{G}_{(\ref{FGL0})}\}=\{XY^\mathbb{T}:(X,Y)\in\mathcal{G}_{(\ref{FGL0C})}\}$.

In what follows, we give some necessary and sufficient conditions for the local minimizers of problem (\ref{LRMPR}), and show the inclusion relation for the local minimizers between problems (\ref{LRMPR}) and (\ref{FGL0}).

\begin{proposition}\label{lcm-rela}
$\bar{Z}\in\mathcal{L}_{(\ref{LRMPR})}$ iff it is a local minimizer of $f$ on $\{Z\in\mathbb{R}^{m\times n}:\emph{\text{rank}}(Z)\leq d\}\cap\Omega_{\bar{Z}}\,\big(\Omega_{\bar{Z}}:=\{Z\in\mathbb{R}^{m\times n}:\emph{\text{rank}}(Z)=\emph{\text{rank}}(\bar{Z})\}\big)$, which is also equivalent to that it is a local minimizer of $f$ on $\{Z\in\mathbb{R}^{m\times n}:\emph{\text{rank}}(Z)\leq d\}\cap\{Z\in\mathbb{R}^{m\times n}:\emph{\text{rank}}(Z)\leq\emph{\text{rank}}(\bar{Z})\}$. Moreover, for any $Z\in\mathcal{L}_{(\ref{LRMPR})}$, there exists $(X,Y)\in\mathcal{L}_{(\ref{FGL0})}$ satisfying $XY^{\mathbb{T}}=Z$ and $\emph{\text{nnzc}}(X)=\emph{\text{nnzc}}(Y)=\emph{\text{rank}}(Z)$.
\end{proposition}
\begin{proof}
For a given $\bar{Z}\in\mathbb{R}^{m\times n}$, there exists an $\varepsilon>0$ such that $\text{rank}(Z)\geq\text{rank}(\bar{Z}),\forall Z\in B_\varepsilon(\bar{Z})$, which means $\{Z\in B_\varepsilon(\bar{Z}):\text{rank}(Z)\leq\text{rank}(\bar{Z})\}=\{Z\in B_\varepsilon(\bar{Z}):\text{rank}(Z)=\text{rank}(\bar{Z})\}$. Therefore, $\bar{Z}$ is a local minimizer of $f$ on $\{Z\in\mathbb{R}^{m\times n}:\text{rank}(Z)\leq\text{rank}(\bar{Z})\}$ \emph{iff} it is a local minimizer of $f$ on $\Omega_{\bar{Z}}$.
	
When $\bar{Z}$ is a local minimizer of $f$ on $\{Z:{\text{rank}}(Z)\leq d\}\cap\Omega_{\bar{Z}}$, there exists an $\hat\varepsilon>0$ such that $f(\bar{Z})+2\lambda\,\text{rank}(\bar{Z})\leq f(Z)+2\lambda\,\text{rank}(Z),\forall Z\in B_{\hat\varepsilon}(\bar{Z})\cap\{Z:{\text{rank}}(Z)\leq d\}\cap\Omega_{\bar{Z}}$. 
Moreover, there exists an $\bar\epsilon\in(0,\min\{\hat\varepsilon,\varepsilon\}]$ 
such that $f(\bar{Z})\leq f(Z)+2\lambda,\forall Z\in B_{\bar\epsilon}(\bar{Z})$ by the continuity of $f$, which implies $f(\bar{Z})+2\lambda\,\text{rank}(\bar{Z})\leq f(Z)+2\lambda+2\lambda\,\text{rank}(\bar{Z})\leq f(Z)+2\lambda\,\text{rank}(Z),\forall Z\in B_{\bar\epsilon}(\bar{Z})\cap\{Z:{\text{rank}}(Z)\leq d\}\setminus\Omega_{\bar{Z}}$. Thus, $\bar{Z}\in\mathcal{L}_{(\ref{LRMPR})}$.
	
Conversely, suppose $\bar{Z}\in\mathcal{L}_{(\ref{LRMPR})}$. There exists an $\epsilon>0$ such that
\begin{equation}\label{lcmn_ieq}
f(\bar{Z})+2\lambda\,\text{rank}(\bar{Z})\leq f(Z)+2\lambda\,\text{rank}(Z),\quad \forall Z\in B_\epsilon(\bar{Z})\cap\{Z:{\text{rank}}(Z)\leq d\}.
\end{equation}
This implies $f(\bar{Z})\leq f(Z), \forall Z\in B_{\epsilon}(\bar{Z})\cap\{Z:{\text{rank}}(Z)\leq d\}\cap\Omega_{\bar{Z}}$, and hence $\bar{Z}$ is a local minimizer of $f$ on $\{Z:{\text{rank}}(Z)\leq d\}\cap\Omega_{\bar{Z}}$.
	
For the given $\bar{Z}$, it is not hard to construct $(\bar{X},\bar{Y})\in\mathbb{R}^{m\times d}\times\mathbb{R}^{n\times d}$ such that $\text{nnzc}(\bar{X})=\text{nnzc}(\bar{Y})=\text{rank}(\bar{Z})$ and $\bar{X}\bar{Y}^{\mathbb{T}}=\bar{Z}$. Then, by the continuity of $\|XY^{\mathbb{T}}\|$, we infer that there exists an $\epsilon_1>0$ such that $XY^{\mathbb{T}}\in B_{\epsilon}(\bar{Z}),\forall (X,Y)\in B_{\epsilon_1}(\bar{X},\bar{Y})$. Further, by (\ref{rank-rel}), (\ref{lcmn_ieq}) and $\text{rank}(XY^{\mathbb{T}})\leq d,\forall (X,Y)\in\mathbb{R}^{m\times d}\times\mathbb{R}^{n\times d}$, we have $F_0(\bar{X},\bar{Y})=f(\bar{Z})+2\lambda\,\text{rank}(\bar{Z})\leq f(XY^{\mathbb{T}})+2\lambda\,\text{rank}(XY^{\mathbb{T}})\leq F_0(X,Y),\forall (X,Y)\in B_{\epsilon_1}(\bar{X},\bar{Y})$. This means $(\bar{X},\bar{Y})\in\mathcal{L}_{(\ref{FGL0})}$.
\end{proof}
The prior bound ($\|Z^\diamond\|\leq\varsigma^2$) of the target matrix $Z^\diamond$ can be exploited by the constraint $\Omega^1\times\Omega^2$ in problem (\ref{FGL0C}). This makes $Z^\diamond$ belong to $\{XY^{\mathbb{T}}:(X,Y)\in\Omega^1\times\Omega^2\}$ by Proposition \ref{rank_const_trans}. Moreover, based on Proposition \ref{bd-rela} and Proposition \ref{lcm-rela}, problems (\ref{LRMPR}), (\ref{FGL0}) and (\ref{FGL0C}) own the relationships as shown in Fig.\,\ref{probs-rela}. As we can see, the problem (\ref{FGL0C}) can be of much smaller scale than problem (\ref{LRMPR}), and the set of its global minimizers covers all global minimizers of (\ref{LRMPR}) that are bounded by $\varsigma^2$.

\section{Model Analysis on Factorization and Relaxation}\label{section4}
In this section, we define some classes of stationary points to problem (\ref{FGL0C}) and its relaxation problem (\ref{FGL0R}), and then analyze some properties of their global minimizers, local minimizers and  stationary points. In particular, we find two 
consistency properties of the global minimizers of (\ref{FGL0C}) and (\ref{FGL0R}), and incorporate these properties to define their notions of strong stationary points. Further, we establish the equivalence between (\ref{FGL0C}) and (\ref{FGL0R}) in the sense of global minimizers and the defined strong stationary points.

\subsection{Factorization model analysis}
First, we define a property that will be later shown to be possessed by the global minimizers of problems (\ref{FGL0C}) and (\ref{FGL0R}).
\begin{definition}[\bf Column consistency property]\label{col_sons_pr}
We say that a matrix couple $(X,Y)\in\mathbb{R}^{m\times d}\times\mathbb{R}^{n\times d}$ has the column consistency property if $\mathcal{I}_{X}=\mathcal{I}_{Y}$.
\end{definition}
\begin{definition}[\bf Stationary point of (\ref{FGL0C})]\label{Lsta}
We call $(X,Y)\in\Omega^1\times\Omega^2$ as a stationary point of (\ref{FGL0C}) if there exists an $H\in\partial_Zf(XY^{\mathbb{T}})$ such that
\begin{equation}\label{Lsta-incl}
\bm{0}\in HY_{\mathcal{I}}+N_{\Omega^1_{\mathcal{I}}}(X_{\mathcal{I}})\mbox{ and }\bm{0}\in H^{\mathbb{T}}X_{\mathcal{I}}+N_{\Omega^2_{\mathcal{I}}}(Y_{\mathcal{I}})\mbox{ with }\mathcal{I}=\mathcal{I}_{X}\cap\mathcal{I}_{Y}.
\end{equation}
The set of stationary points of problem (\ref{FGL0C}) is denoted as $\mathcal{S}_{(\ref{FGL0C})}$.
\end{definition}
Note that (\ref{Lsta-incl}) is equivalent to $\bm{0}\in HY_{\mathcal{I}_X}+N_{\Omega^1_{\mathcal{I}_X}}(X_{\mathcal{I}_X})$ and $\bm{0}\in H^{\mathbb{T}}X_{\mathcal{I}_Y}+N_{\Omega^2_{\mathcal{I}_Y}}(Y_{\mathcal{I}_Y})$, because of the fact that $Y_{\mathcal{I}_X\setminus\mathcal{I}}=\bm{0}$ and $X_{\mathcal{I}_Y\setminus\mathcal{I}}=\bm{0}$. 

In nonconvex factorized column-sparse regularized models, the loss function $f$ is often assumed to be smooth and the iterates generated by the designed algorithms are often proved to be convergent to the limiting-critical point set of the regularized models \cite{Bolte2014Proximal,Pan2022factor,YangL2018}. In view of \cite[Exercise 10.10, Proposition 10.5]{Rockafellar1998}, \cite[Theorem 6.23]{Penot2013} and $X_{\mathcal{I}}Y_{\mathcal{I}}^{\mathbb{T}}=XY^{\mathbb{T}}$, we have that any limiting-critical point is a stationary point for (\ref{FGL0C}). It is worth noting that when $f$ is smooth, it follows from Proposition \ref{subd-cri-rela} (ii) that the stationary point defined in Definition \ref{Lsta} is equivalent to the limiting-critical point for (\ref{FGL0C}).

In what follows, we show a necessary and sufficient condition for the local minimizers of problem (\ref{FGL0C}), and further analyze some properties on its global minimizers, local minimizers and stationary points, respectively.
\begin{proposition}\label{FGL0C-prps}
For problem (\ref{FGL0C}), the following statements hold.
\begin{itemize}
\item[(i)] For any $(X,Y)\in\mathcal{G}_{(\ref{FGL0C})}$, $(X,Y)$ owns the column consistency property.
\item[(ii)] For any $(X,Y)\in\Omega^1\times\Omega^2$, $(X,Y)\in\mathcal{L}_{(\ref{FGL0C})}$ iff $(X_{\mathcal{I}},Y_{\mathcal{I}})$ is a local solution of
\begin{equation*}
\min_{M,N}\,f(MN^{\mathbb{T}}),~~\emph{s.t.}~(M,N)\in\Omega_{X,Y}:=\{(M,N):M\in\Omega^1_{\mathcal{I}},N\in\Omega^2_{\mathcal{I}}\}
\end{equation*}
with $\mathcal{I}=\mathcal{I}_{X}\cap\mathcal{I}_{Y}$.
Moreover, $(X,Y)\in\mathcal{L}_{(\ref{FGL0C})}$ if $XY^{\mathbb{T}}$ is a local solution of \begin{equation*}
\min_{Z}\,f(Z),~~\emph{s.t.}~Z\in\{MN^{\mathbb{T}}:(M,N)\in\Omega_{X,Y}\}.
\end{equation*}
\item[(iii)] $\mathcal{L}_{(\ref{FGL0C})}\subseteq\mathcal{S}_{(\ref{FGL0C})}$. Moreover, if $(X,Y)\in\mathcal{L}_{(\ref{FGL0C})}$ (or $\mathcal{S}_{(\ref{FGL0C})}$), then $(X^s,Y^s)$ defined by
\begin{equation}\label{rdc-elm}
X^s_{\mathcal{I}}=X_{\mathcal{I}}, \; X^s_{\mathcal{I}^c}=\bm{0}, \; 
Y^s_{\mathcal{I}}=Y_{\mathcal{I}},\; Y^s_{\mathcal{I}^c}=\bm{0}\mbox{ with }\mathcal{I}=\mathcal{I}_{X}\cap\mathcal{I}_{Y},
\end{equation}
satisfies $(X^s,Y^s)\in\mathcal{L}_{(\ref{FGL0C})}$ (or $\mathcal{S}_{(\ref{FGL0C})}$) and $F_0(X^s,Y^s)\leq F_0(X,Y)$.
\item[(iv)] If $(X,Y)\in\mathcal{G}_{(\ref{FGL0C})}$ (or $\mathcal{L}_{(\ref{FGL0C})}$), then $(\gamma X,\frac{1}{\gamma}Y)\in\mathcal{G}_{(\ref{FGL0C})}$ (or $\mathcal{L}_{(\ref{FGL0C})}$) for any $\gamma\neq 0$ satisfying $(\gamma X,\frac{1}{\gamma}Y)\in\Omega^1\times\Omega^2$.
\end{itemize}
\end{proposition}
\begin{proof}
(i) Similar to the proof of Proposition \ref{FGL0-prps} (i) and by the geometry structure of $\Omega^1\times\Omega^2$, it holds that any global minimizer $(X,Y)$ of (\ref{FGL0C}) satisfies $\mathcal{I}_{X}=\mathcal{I}_{Y}$.
	
(ii) For a given $(\bar{X},\bar{Y})\in\Omega^1\times\Omega^2$, define $\bar{\mathcal{I}}:=\mathcal{I}_{\bar{X}}\cap\mathcal{I}_{\bar{Y}}$ and $\bar{\Omega}:=\{(X,Y)\in\Omega^1\times\Omega^2:X_{\mathcal{I}^c_{\bar{X}}}=\bm{0},Y_{\mathcal{I}^c_{\bar{Y}}}=\bm{0}\}$. There exists an $\varepsilon>0$ such that
\begin{equation}\label{ieqd1}
\begin{split}
&\mathcal{I}_{X}=\mathcal{I}_{\bar{X}},\mathcal{I}_{Y}=\mathcal{I}_{\bar{Y}},\forall\big(X,Y\big)\in B_{\varepsilon}(\bar{X},\bar{Y})\cap\bar\Omega,\\
\mbox{and }&\mathcal{I}_{X}\supsetneq\mathcal{I}_{\bar{X}}\text{ or }\mathcal{I}_{Y}\supsetneq\mathcal{I}_{\bar{Y}},\forall\big(X,Y\big)\in B_{\varepsilon}(\bar{X},\bar{Y})\setminus\bar\Omega.
\end{split}
\end{equation}
This implies $\big\{MN^{\mathbb{T}}:(M,N)\in B_{\delta}(\bar{X}_{\bar{\mathcal{I}}},\bar{Y}_{\bar{\mathcal{I}}})\cap\Omega_{\bar{X},\bar{Y}}\big\}=\big\{XY^{\mathbb{T}}:(X,Y)\in B_{\delta}(\bar{X},\bar{Y})\cap\bar\Omega\big\},\forall\delta\in(0,\varepsilon]$.
In view of $\bar{X}_{\bar{\mathcal{I}}}\bar{Y}_{\bar{\mathcal{I}}}^{\mathbb{T}}=\bar{X}\bar{Y}^{\mathbb{T}}$, $(\bar{X}_{\bar{\mathcal{I}}},\bar{Y}_{\bar{\mathcal{I}}})\in\Omega_{\bar{X},\bar{Y}}$ and $(\bar{X},\bar{Y})\in\bar\Omega$, we deduce that for given $(\bar{X},\bar{Y})$, 
\begin{align}\label{inn-eqv}
&\text{$(\bar{X}_{\bar{\mathcal{I}}},\bar{Y}_{\bar{\mathcal{I}}})$ is a local minimizer of $f(MN^{\mathbb{T}})$ on $\Omega_{\bar{X},\bar{Y}}$}\,\Longleftrightarrow\notag\\
&\text{$(\bar{X},\bar{Y})$ is a local minimizer of $f(XY^{\mathbb{T}})$ on $\bar\Omega$}
\end{align}
	
If $(\bar{X},\bar{Y})\in\mathcal{L}_{(\ref{FGL0C})}$, there exists an $\varepsilon_1\in(0,\varepsilon]$ such that $F_0(\bar{X},\bar{Y})\leq F_0(X,Y),\forall (X,Y)\in B_{\varepsilon_1}(\bar{X},\bar{Y})\cap(\Omega^1\times\Omega^2)$. It follows from (\ref{ieqd1}) that $f(\bar{X}\bar{Y}^{\mathbb{T}})\leq f(XY^{\mathbb{T}}),\forall (X,Y)\in B_{\varepsilon_1}(\bar{X},\bar{Y})\cap\bar\Omega$ and hence (\ref{inn-eqv}) holds.
	
Conversely, consider that (\ref{inn-eqv}) holds. There exists an $\varepsilon_2\in(0,\varepsilon]$ such that $f(\bar{X}\bar{Y}^{\mathbb{T}})\leq f(XY^{\mathbb{T}}),\forall (X,Y)\in B_{\varepsilon_2}(\bar{X},\bar{Y})\cap\bar\Omega$.
Further, based on (\ref{ieqd1}), we have
\begin{equation}\label{ieqd4}
F_0(\bar{X},\bar{Y})\leq F_0(X,Y),\forall (X,Y)\in B_{\varepsilon_2}(\bar{X},\bar{Y})\cap\bar\Omega.
\end{equation}
Moreover, by the continuity of $f(XY^{\mathbb{T}})$, there exists an $\varepsilon_3\in(0,\varepsilon_2]$ such that $f(\bar{X}\bar{Y}^{\mathbb{T}})\leq f(XY^{\mathbb{T}})+\lambda,\forall (X,Y)\in B_{\varepsilon_3}(\bar{X},\bar{Y})$. Combining with (\ref{ieqd1}), we deduce
that 
\begin{equation}\label{ieqd5}
\begin{split}
&F_0(\bar{X},\bar{Y})\leq f(XY^{\mathbb{T}})+\lambda\big(1+\text{nnzc}(\bar{X})+\text{nnzc}(\bar{Y})\big)\\
\leq\,&F_0(X,Y),\,\,\forall (X,Y)\in B_{\varepsilon_3}(\bar{X},\bar{Y})\cap(\Omega^1\times\Omega^2)\setminus\bar\Omega.
\end{split}
\end{equation}
This together with (\ref{ieqd4}) implies $F_0(\bar{X},\bar{Y})\leq F_0(X,Y),\forall (X,Y)\in B_{\varepsilon_3}(\bar{X},\bar{Y})\cap(\Omega^1\times\Omega^2)$. Thus, $(\bar{X},\bar{Y})\in\mathcal{L}_{(\ref{FGL0C})}$.
	
By $\bar{X}\bar{Y}^{\mathbb{T}}=\bar{X}_{\bar{\mathcal{I}}}\bar{Y}_{\bar{\mathcal{I}}}^{\mathbb{T}}$ and the continuity of $\|MN^{\mathbb{T}}\|$ in $\Omega_{\bar{X},\bar{Y}}$, the last statement in (ii) naturally holds.
	
(iii) Let $(X,Y)\in\mathcal{L}_{(\ref{FGL0C})}$. Then, by (ii), $(X_{\mathcal{I}},Y_{\mathcal{I}})$ is a local minimizer of $f(MN^{\mathbb{T}})$ on $\Omega_{X,Y}$. It follows from Fermat's rule in \cite[Theorem 10.1]{Rockafellar1998} and \cite[Exercise 10.10]{Rockafellar1998} that
$\bm{0}\in\partial_{(X_{\mathcal{I}},Y_{\mathcal{I}})}f(X_{\mathcal{I}}Y_{\mathcal{I}}^{\mathbb{T}})+N_{\Omega_{X,Y}}(X_{\mathcal{I}},Y_{\mathcal{I}})$. Further, by \cite[Theorem 6.23]{Penot2013}, \cite[Proposition 10.5]{Rockafellar1998} and $X_{\mathcal{I}}Y_{\mathcal{I}}^{\mathbb{T}}=XY^{\mathbb{T}}$, we have that there exists an $H\in\partial_Zf(XY^{\mathbb{T}})$ such that $\bm{0}\in HY_{\mathcal{I}}+N_{\Omega^1_{\mathcal{I}}}(X_{\mathcal{I}})$ and $\bm{0}\in H^{\mathbb{T}}X_{\mathcal{I}}+N_{\Omega^2_{\mathcal{I}}}(Y_{\mathcal{I}})$. Therefore, $(X,Y)\in\mathcal{S}_{(\ref{FGL0C})}$. Moreover, by the definition of $(X^s,Y^s)$, we have $X^s_{\mathcal{I}_{X^s}\cap\mathcal{I}_{Y^s}}=X_{\mathcal{I}}$, $Y^s_{\mathcal{I}_{X^s}\cap\mathcal{I}_{Y^s}}=Y_{\mathcal{I}}$ and $\Omega_{X^s,Y^s}=\Omega_{X,Y}$. Then, by the equivalence in (ii), $(X^s_{\mathcal{I}_{X^s}\cap\mathcal{I}_{Y^s}},Y^s_{\mathcal{I}_{X^s}\cap\mathcal{I}_{Y^s}})$ is a local minimizer of $f(MN^{\mathbb{T}})$ on $\Omega_{X^s,Y^s}$ and hence $(X^s,Y^s)\in\mathcal{L}_{(\ref{FGL0C})}$.
	
Suppose $(X,Y)\in\mathcal{S}_{(\ref{FGL0C})}$. Considering $X^s_{\mathcal{I}}=X_{\mathcal{I}}$, $Y^s_{\mathcal{I}}=Y_{\mathcal{I}}$ and $\mathcal{I}=\mathcal{I}_{X^s}=\mathcal{I}_{Y^s}$, we have that (\ref{Lsta-incl}) holds with $X=X^s$, $Y=Y^s$ and $\mathcal{I}=\mathcal{I}_{X^s}\cap\mathcal{I}_{Y^s}$. This yields $(X^s,Y^s)\in\mathcal{S}_{(\ref{FGL0C})}$.
	
Since $\text{nnzc}(X^s)\leq\text{nnzc}(X)$, $\text{nnzc}(Y^s)\leq\text{nnzc}(Y)$ and $X^s(Y^s)^{\mathbb{T}}=XY^{\mathbb{T}}$, we get $F_0(X^s,Y^s)\leq F_0(X,Y)$.
	
(iv) For a given $(X,Y)\in\Omega^1\times\Omega^2$, let $\gamma\neq 0$ satisfy $(\gamma X,\frac{1}{\gamma}Y)\in\Omega^1\times\Omega^2$. This implies $\Omega_{\gamma X,Y/\gamma}=\Omega_{X,Y}$, $\mathcal{I}_{X}=\mathcal{I}_{\gamma X}$ and $\mathcal{I}_{Y}=\mathcal{I}_{\frac{1}{\gamma}Y}$. Further, if $(X,Y)\in\mathcal{G}_{(\ref{FGL0C})}$, then by $\gamma X\frac{1}{\gamma}Y^{\mathbb{T}}=XY^{\mathbb{T}}$, we obtain $F_0(X,Y)=F_0(\gamma X,\frac{1}{\gamma}Y)$, which implies $(\gamma X,\frac{1}{\gamma}Y)\in\mathcal{G}_{(\ref{FGL0C})}$.
	
Suppose $(X,Y)\in\mathcal{L}_{(\ref{FGL0C})}$. Based on (ii), there exists a $\delta>0$ such that
\begin{equation}\label{eqd2}
f(X_{\mathcal{I}}Y_{\mathcal{I}}^{\mathbb{T}})\leq f(MN^{\mathbb{T}}),\forall (M,N)\in B_{\delta}(X_{\mathcal{I}},Y_{\mathcal{I}})\cap\Omega_{X,Y}.
\end{equation}
Moreover, there exists a $\bar\delta>0$ such that $(\frac{1}{{\gamma}}M,\gamma N)\in B_{\delta}(X_{\mathcal{I}},Y_{\mathcal{I}}),\forall(M,N)\in B_{\bar\delta}(\gamma X_{\mathcal{I}},\frac{1}{\gamma}Y_{\mathcal{I}})$. Combining with (\ref{eqd2}) and $\Omega_{\gamma X,Y/\gamma}=\Omega_{X,Y}$, we obtain that for any $(M,N)\in B_{\bar\delta}(\gamma X_{\mathcal{I}},Y_{\mathcal{I}}/\gamma)\cap\Omega_{\gamma X,Y/\gamma}$,
\begin{equation*}
f\big((\gamma X_{\mathcal{I}})(Y_{\mathcal{I}}/\gamma)^{\mathbb{T}}\big)=f(X_{\mathcal{I}}Y_{\mathcal{I}}^{\mathbb{T}})\leq f((M/\gamma)(\gamma N)^{\mathbb{T}})= f(MN^{\mathbb{T}}),
\end{equation*}
which implies $(\gamma X,\frac{1}{\gamma} Y)\in\mathcal{L}_{(\ref{FGL0C})}$ by (ii).
\end{proof}
\begin{remark}\label{lc-pr-id}
According to Proposition \ref{FGL0C-prps} (ii) and (iv), for any given local minimizer $(X,Y)$ of (\ref{FGL0C}), we can obtain a corresponding subset of the local minimizers of (\ref{FGL0C}), that is $\{(U,V)\in\Omega^1\times\Omega^2:U_{\mathcal{I}_{X}\cap\mathcal{I}_{Y}}=\gamma X_{\mathcal{I}_{X}\cap\mathcal{I}_{Y}},V_{\mathcal{I}_{X}\cap\mathcal{I}_{Y}}=\frac{1}{\gamma}Y_{\mathcal{I}_{X}\cap\mathcal{I}_{Y}},\gamma>0\text{ and }\mathcal{I}_{U}\cap\mathcal{I}_{V}=\mathcal{I}_{X}\cap\mathcal{I}_{Y}\}:=\Delta_{X,Y}$.
	
Further, if $|\mathcal{I}_{X}\cap\mathcal{I}_{Y}|<d$, then there exists a local minimizer $(\hat{X},\hat{Y})\in\Delta_{X,Y}$ of (\ref{FGL0C}) satisfying $\mathcal{I}_{\hat{X}}\neq\mathcal{I}_{\hat{Y}}$, which implies that the column consistency property is not a necessary condition for all local minimizers. Moreover, $(X^s,Y^s)$ defined in Proposition \ref{FGL0C-prps} (iii) belongs to $\Delta_{X,Y}$. By $\hat{X}\hat{Y}^{\mathbb{T}}=X^s(Y^s)^{\mathbb{T}}$ and $\emph{\text{nnzc}}(\hat{X})+\emph{\text{nnzc}}(\hat{Y})>\emph{\text{nnzc}}(X^s)+\emph{\text{nnzc}}(Y^s)$, we have $F_0(\hat{X},\hat{Y})>F_0(X^s,Y^s)$. 
Therefore, the operation in (\ref{rdc-elm}) can enforce the column consistency from the global minimizers of (\ref{FGL0C}), and potentially reduce the objective function value while maintaining the local optimality to (\ref{FGL0C}). This shows the superiority and necessity for performing the operation in (\ref{rdc-elm}).
\end{remark}

\subsection{Relaxation and equivalence}
In this subsection, we define some stationary points for problem (\ref{FGL0C}) and its relaxation problem (\ref{FGL0R}). Further, for these problems, we prove their equivalence in the sense of global minimizers and the defined strong stationary points, as well as the inclusion relations on local minimizers and  stationary points, respectively.

We first introduce the parameter $\kappa$ used in what follows. Based on the locally Lipschitz continuity of $f$ and the boundedness of $\Omega^1$ and $\Omega^2$, there exists $\kappa>0$ such that $\max\{\|H Y_i\|,\|H^{\mathbb{T}}X_i\|\}\leq\kappa$ for any $(X,Y)\in\Omega^1\times\Omega^2$, $H\in\partial_Zf(XY^{\mathbb{T}})$ and $i\in[d]$.

Throughout this section, we assume that $\nu$ fulfills the following assumption.
\begin{assumption}\label{mu-nu}
$\nu$ in (\ref{FGL0R}) satisfies $\nu\in(0,\min\{\lambda/\kappa,\Mbound\})$.
\end{assumption}

In what follows, we define a class of stationary points for the relaxation problem (\ref{FGL0R}).
\begin{definition}[\bf Stationary point of (\ref{FGL0R})]\label{Rsta}
We call $(X,Y)\in\Omega^1\times\Omega^2$ as a stationary point of (\ref{FGL0R}) if there exists an $H\in\partial_Zf(XY^{\mathbb{T}})$ such that
\begin{equation}\label{Rsta-incl}
\bm{0}\in HY+\lambda\partial_{X}\Theta_{\nu,D^X}(X)+N_{\Omega^1}(X)\mbox{ and }\bm{0}\in H^{\mathbb{T}}X+\lambda\partial_{Y}\Theta_{\nu,D^Y}(Y)+N_{\Omega^2}(Y),
\end{equation}
where $D^{X}$ and $D^{Y}$ are index vectors in $\{1,2\}^d$ defined by $D^{C}\in\{1,2\}^d$ with $D^{C}_i:=\max\{j\in\{1,2\}:\theta_{\nu,j}(\|C_i\|)=\theta_\nu(\|C_i\|)\},i\in[d]$ for any $C\in\mathbb{R}^{r\times d}$. The set of stationary points is denoted as $\mathcal{S}_{(\ref{FGL0R})}$ for problem (\ref{FGL0R}).
\end{definition}
Further, we define an important property for the matrix couple. It will be proved that this property stems from the stationary points of problem (\ref{FGL0R}) and plays a pivotal role in establishing the equivalence between the global minimizers of (\ref{FGL0C}) and (\ref{FGL0R}).
\begin{definition}[\bf  Column bound-$\alpha$ property]\label{col-bd-p-def}
For a given constant $\alpha\in(0,\nu]$, we say that a matrix couple $(X,Y)\in\mathbb{R}^{m\times d}\times\mathbb{R}^{n\times d}$ has the column bound-$\alpha$ property if $\|X_i\|\notin(0,\alpha)$ and $\|Y_i\|\notin(0,\alpha),\forall i\in[d]$. That is, either $\norm{X_i} \geq \alpha$ or $\norm{X_i} = 0$ for all $i\in [d]$; similarly for $Y$.
\end{definition}

Next, we show some properties of the local minimizers and stationary points of problem (\ref{FGL0R}).
\begin{proposition}\label{gl-lc-sta}
Problem (\ref{FGL0R}) owns the following properties.
\begin{itemize}
\item[(i)] $\mathcal{L}_{(\ref{FGL0R})}\subseteq\mathcal{S}_{(\ref{FGL0R})}$.
\item[(ii)] Let $(X,Y)\in\mathcal{S}_{(\ref{FGL0R})}$. Then, $(X,Y)$ has the column bound-$\nu$ property.
\item[(iii)] Let $(X^s,Y^s)$ be defined by (\ref{rdc-elm}). If $(X,Y)\in\mathcal{L}_{(\ref{FGL0R})}$ (or $\mathcal{S}_{(\ref{FGL0R})}$), then $(X^s,Y^s)\in\mathcal{L}_{(\ref{FGL0R})}$ (or $\mathcal{S}_{(\ref{FGL0R})}$) and $F(X^s,Y^s)\leq F(X,Y)$.
\end{itemize}
\end{proposition}
\begin{proof}
(i) Let $(\bar{X},\bar{Y})\in\mathcal{L}_{(\ref{FGL0R})}$. Then, there exists a $\delta>0$ such that $F(\bar{X},\bar{Y})\leq F(X,Y),\forall(X,Y)\in B_\delta(\bar{X},\bar{Y})\cap(\Omega^1\times\Omega^2)$. Define $g(X,Y)=\Theta_{\nu,D^{\bar{X}}}(X)+\Theta_{\nu,D^{\bar{Y}}}(Y)$. Since $\Theta_\nu(X)+\Theta_\nu(Y)\leq g(X,Y),\forall (X,Y)$ and $\Theta_\nu(\bar{X})+\Theta_\nu(\bar{Y})=g(\bar{X},\bar{Y})$, we have $f(\bar{X}\bar{Y}^{\mathbb{T}})+\lambda g(\bar{X},\bar{Y})=F(\bar{X},\bar{Y})\leq F(X,Y)\leq f(XY^{\mathbb{T}})+\lambda g(X,Y),\forall(X,Y)\in B_\delta(\bar{X},\bar{Y})\cap(\Omega^1\times\Omega^2)$. Then, by the optimality condition, (\ref{Rsta-incl}) holds, which means $(\bar{X},\bar{Y})\in\mathcal{S}_{(\ref{FGL0R})}$.
	
(ii) Let $(X,Y)\in\mathcal{S}_{(\ref{FGL0R})}$. If there exists an $i_0\in[d]$ such that $\|X_{i_0}\|\in(0,\nu)$, it follows that $D^{X}_{i_0}=1$. Combining with Assumption \ref{mu-nu}, we have $X_{i_0}\in{\rm{int}}\Omega^1_{i_0}$ (hence $N_{\Omega^1_{i_0}}(X_{i_0})=\{0\}$) and
\begin{equation*}
\|\lambda\nabla_{X_{i_0}}\Theta_{\nu,D^{X}}(X)\|={\lambda}/{\nu}>\kappa\geq\|HY_{i_0}\|,\forall H\in\partial_Zf(XY^{\mathbb{T}}).
\end{equation*}
Then, we infer that $\bm{0}\notin\partial_Zf(XY^{\mathbb{T}})Y_{i_0}+\lambda\nabla_{X_{i_0}}\Theta_{\nu,D^{X}}(X)+N_{\Omega^1_{i_0}}(X_{i_0})$, which contradicts (\ref{Rsta-incl}). Hence, $\|X_i\|\notin(0,\nu),\forall i\in[d]$. Similarly, it holds $\|Y_i\|\notin(0,\nu),\forall i\in[d]$. Thus, $(X,Y)$ has the column bound-$\nu$ property.
	
(iii) Suppose $(X,Y)\in\mathcal{S}_{(\ref{FGL0R})}$. By (ii) and (\ref{Rsta-incl}), there exists an $H\in\partial_Zf(XY^{\mathbb{T}})$ such that
\begin{equation}\label{Rsta-incl-tr}
\bm{0}\in HY+\frac{\lambda}{\nu}\sum_{i\in\mathcal{I}^c_{X}}\partial_{X}\|X_i\|+N_{\Omega^1}(X)\mbox{ and }\bm{0}\in H^{\mathbb{T}}X+\frac{\lambda}{\nu}\sum_{i\in\mathcal{I}^c_{Y}}\partial_{Y}\|Y_i\|+N_{\Omega^2}(Y).
\end{equation}
Let $\mathcal{I}=\mathcal{I}_{X}\cap\mathcal{I}_{Y}$ and $(X^s,Y^s)$ be defined by (\ref{rdc-elm}). Combining $X^s_{\mathcal{I}}=X_{\mathcal{I}}$, $Y^s_{\mathcal{I}}=Y_{\mathcal{I}}$ and (\ref{Rsta-incl-tr}), we get
\begin{equation}\label{Rsta-incl-tr2}
\begin{split}
&\bm{0}\in \Big[HY^s+\frac{\lambda}{\nu}\sum_{i\in\mathcal{I}^c}\partial_{X}\|X^s_i\|+N_{\Omega^1}(X^s)\Big]_{\mathcal{I}}\\
\mbox{and }&\bm{0}\in \Big[H^{\mathbb{T}}X^s+\frac{\lambda}{\nu}\sum_{i\in\mathcal{I}^c}\partial_{Y}\|Y^s_i\|+N_{\Omega^2}(Y^s)\Big]_{\mathcal{I}}.
\end{split}
\end{equation}
Since $X^s(Y^s)^{\mathbb{T}}=XY^{\mathbb{T}}$, we have $H\in\partial_Zf(X^s(Y^s)^{\mathbb{T}})$. According to Assumption \ref{mu-nu}, we deduce that $\bm{0}\in \big[HY^s+\frac{\lambda}{\nu}\sum_{i\in\mathcal{I}^c}\partial_{X}\|X^s_i\|+N_{\Omega^1}(X^s)\big]_{\mathcal{I}^c}$ and $\bm{0}\in \big[H^{\mathbb{T}}X^s+\frac{\lambda}{\nu}\sum_{i\in\mathcal{I}^c}\partial_{Y}\|Y^s_i\|+N_{\Omega^2}(Y^s)\big]_{\mathcal{I}^c}$. This combined with (\ref{Rsta-incl-tr2}) and the column bound-$\nu$ property of $(X^s,Y^s)$ yields $\bm{0}\in HY^s+\lambda\partial_{X}\Theta_{\nu,D^{X^s}}(X^s)+N_{\Omega^1}(X^s)$ and $\bm{0}\in H^{\mathbb{T}}X^s+\lambda\partial_{Y}\Theta_{\nu,D^{Y^s}}(Y^s)+N_{\Omega^2}(Y^s)$. Therefore, $(X^s,Y^s)\in\mathcal{S}_{(\ref{FGL0R})}$. Moreover, by $\mathcal{I}_{X^s}\subseteq\mathcal{I}_{X}$, $\mathcal{I}_{Y^s}\subseteq\mathcal{I}_{Y}$ and $X^s(Y^s)^{\mathbb{T}}=XY^{\mathbb{T}}$, we get $F(X^s,Y^s)\leq F(X,Y)$.
	
Further, consider the case $(X,Y)\in\mathcal{L}_{(\ref{FGL0R})}$. By (i), we have $(X,Y)\in\mathcal{S}_{(\ref{FGL0R})}$ and hence $F(X^s,Y^s)\leq F(X,Y)$. For any $i\in\mathcal{I}_{X}\setminus\mathcal{I}$ and $j\in\mathcal{I}_{Y}\setminus\mathcal{I}$, we have $0=\theta_\nu(\|X^s_i\|) \leq\theta_\nu(\|\tilde{X}_i\|),\forall \tilde{X}\in\mathbb{R}^{m\times d}$ and $0=\theta_\nu(\|Y^s_j\|)\leq\theta_\nu(\|\tilde{Y}_j\|),\forall \tilde{Y}\in\mathbb{R}^{n\times d}$. This together with $X^s_i=X_i,\forall i\in([d]\setminus{\mathcal{I}_X})\cup\mathcal{I}$, $Y^s_j=Y_j,\forall j\in([d]\setminus{\mathcal{I}_Y})\cup\mathcal{I}$ and $X^s(Y^s)^{\mathbb{T}}=XY^{\mathbb{T}}$ implies $(X^s,Y^s)\in\mathcal{L}_{(\ref{FGL0R})}$.
\end{proof}
Proposition \ref{gl-lc-sta} not only states that any stationary point of (\ref{FGL0R}) owns the column bound-$\nu$ property, but also explains that the operation of (\ref{rdc-elm}) on any stationary point (or local minimizer) can give a stationary point (or local minimizer) with no larger objective function value. Notably, this superiority of (\ref{rdc-elm}) is also applicable to the stationary points and local minimizers of problem (\ref{FGL0C}) by Proposition \ref{FGL0C-prps} (iii).

In what follows, we define a class of strong stationary points for problems (\ref{FGL0C}) and (\ref{FGL0R}), which are stationary points satisfying certain properties of their
global minimizers.

\begin{definition}[\bf strong stationary points of (\ref{FGL0C}) and (\ref{FGL0R})]\label{CR-sstr}
For problem (\ref{FGL0C}) (or (\ref{FGL0R})), we call $(X,Y)\in\Omega^1\times\Omega^2$ as a strong stationary point  if it  satisfies 
\begin{itemize}
\item[(i)] $(X,Y)\in\mathcal{S}_{(\ref{FGL0C})}$ (or $\mathcal{S}_{(\ref{FGL0R})})$;
\item[(ii)] the column consistency property in Definition \ref{col_sons_pr};
\item[(iii)] the column bound-$\nu$ property in Definition \ref{col-bd-p-def}.
\end{itemize}
The sets of all strong stationary points of (\ref{FGL0C}) and (\ref{FGL0R}) are denoted as $\mathcal{S}^s_{(\ref{FGL0C})}$ and $\mathcal{S}^s_{(\ref{FGL0R})}$, respectively.
\end{definition}

As shown in the next theorem, the properties in (ii) and (iii) of Definition \ref{CR-sstr} 
are satisfied by the global minimizers of problems (\ref{FGL0C}) and (\ref{FGL0R}).
\begin{theorem}\label{stat-loc}
The following statements hold for problems (\ref{FGL0C}) and (\ref{FGL0R}).
\begin{itemize}
\item[(i)] $\mathcal{G}_{(\ref{FGL0R})}=\mathcal{G}_{(\ref{FGL0C})}$ and $\mathcal{L}_{(\ref{FGL0R})}\subseteq\mathcal{L}_{(\ref{FGL0C})}$. Moreover, the optimal values of (\ref{FGL0C}) and (\ref{FGL0R}) are the same.
\item[(ii)] For any $(X,Y)\in\mathcal{G}_{(\ref{FGL0C})}$ (or $\mathcal{G}_{(\ref{FGL0R})}$), $(X,Y)$ owns the column consistency property and column bound-$\nu$ property.
\item[(iii)] $\mathcal{G}_{(\ref{FGL0C})}\subseteq\mathcal{S}^s_{(\ref{FGL0C})}$ and $\mathcal{G}_{(\ref{FGL0R})}\subseteq\mathcal{S}^s_{(\ref{FGL0R})}$. Moreover, $\mathcal{S}^s_{(\ref{FGL0C})}=\mathcal{S}^s_{(\ref{FGL0R})}\subseteq\mathcal{S}_{(\ref{FGL0R})}\subseteq\mathcal{S}_{(\ref{FGL0C})}$.
\end{itemize}
\end{theorem}
\begin{proof}
(i) For $(X,Y)\in\mathcal{G}_{(\ref{FGL0C})}$, if $(X,Y)\notin\mathcal{G}_{(\ref{FGL0R})}$, then there exists $(\hat{X},\hat{Y})\in\mathcal{G}_{(\ref{FGL0R})}$ such that $F_0(\hat{X},\hat{Y})=F(\hat{X},\hat{Y})<F(X,Y)$, where the first equality is due to Proposition \ref{gl-lc-sta} (i) and (ii). It follows from $\Theta_\nu\leq \text{nnzc}$ that $F(X,Y)\leq F_0(X,Y)$, which leads to a contradiction. As a result, $(X,Y)\in\mathcal{G}_{(\ref{FGL0R})}$.
	
Let $(\bar{X},\bar{Y})\in\mathcal{G}_{(\ref{FGL0R})}$ (or $\mathcal{L}_{(\ref{FGL0R})}$). Then, (there exists an $\varepsilon>0$ such that) $F(\bar{X},\bar{Y})\leq F(X,Y),\forall (X,Y)\in (\Omega^1\times\Omega^2)\cap B_\varepsilon(\bar{X},\bar{Y})$. By Proposition \ref{gl-lc-sta} (i) and (ii), we have that $(\bar{X},\bar{Y})$ own the column bound-$\nu$ property, which implies $F(\bar{X},\bar{Y})=F_0(\bar{X},\bar{Y})$. This together with $F(X,Y)\leq F_0(X,Y),\forall X,Y$ yields $F_0(\bar{X},\bar{Y})=F(\bar{X},\bar{Y}) \leq F(X,Y) \leq F_0(X,Y),\forall (X,Y)\in (\Omega^1\times\Omega^2)\cap B_\varepsilon(\bar{X},\bar{Y})$. Thus, $(\bar{X},\bar{Y})\in\mathcal{G}_{(\ref{FGL0C})}$(or $\mathcal{L}_{(\ref{FGL0C})}$) and hence $F_0(\bar{X},\bar{Y})=F(\bar{X},\bar{Y})$ by Proposition \ref{gl-lc-sta} (i) and (ii), which implies that the optimal values of (\ref{FGL0C}) and (\ref{FGL0R}) are the same.
	
(ii) Combining (i), Proposition \ref{FGL0C-prps} (i), Proposition \ref{gl-lc-sta} (i) and (ii), we get the statement (ii).
	
(iii) By Proposition \ref{FGL0C-prps} (iii) and Proposition \ref{gl-lc-sta} (i), we have $\mathcal{G}_{(\ref{FGL0C})}\subseteq\mathcal{S}_{(\ref{FGL0C})}$ and $\mathcal{G}_{(\ref{FGL0R})}\subseteq\mathcal{S}_{(\ref{FGL0R})}$. Moreover, it follows from (ii) that $\mathcal{G}_{(\ref{FGL0C})}\subseteq\mathcal{S}^s_{(\ref{FGL0C})}$ and $\mathcal{G}_{(\ref{FGL0R})}\subseteq\mathcal{S}^s_{(\ref{FGL0R})}$.
	
Suppose $(X,Y)\in\mathcal{S}_{(\ref{FGL0R})}$. By Proposition \ref{gl-lc-sta} (ii), $(X,Y)$ owns the column bound-$\nu$ property. Then, $D^{X}_i=2,\forall i\in\mathcal{I}_{X}$ and $D^{Y}_i=2,\forall i\in\mathcal{I}_{Y}$, which means $[\nabla_{X}\Theta_{\nu,D^{X}}(X)]_{\mathcal{I}_{X}}=\bm{0}\mbox{ and }[\nabla_{Y}\Theta_{\nu,D^{Y}}(Y)]_{\mathcal{I}_{Y}}=\bm{0}$. Combining with (\ref{Rsta-incl}), we get that there exists an $H\in\partial_Zf(XY^{\mathbb{T}})$ such that $\bm{0}\in HY_{\mathcal{I}_X}+N_{\Omega^1_{\mathcal{I}_X}}(X_{\mathcal{I}_X})$ and $\bm{0}\in H^{\mathbb{T}}X_{\mathcal{I}_Y}+N_{\Omega^2_{\mathcal{I}_Y}}(Y_{\mathcal{I}_Y})$, which implies that (\ref{Lsta-incl}) holds. Hence, $(X,Y)\in\mathcal{S}_{(\ref{FGL0C})}$. Further, $\mathcal{S}^s_{(\ref{FGL0R})}\subseteq\mathcal{S}^s_{(\ref{FGL0C})}$.
	
Conversely, let $(X,Y)\in\mathcal{S}^s_{(\ref{FGL0C})}$. By Definition \ref{CR-sstr}, we only need to prove $(X,Y)\in\mathcal{S}_{(\ref{FGL0R})}$. Due to its column bound-$\nu$ property, we have $\big[\nabla_{X}\Theta_{\nu,D^{X}}(X)\big]_{\mathcal{I}_{X}}=\bm{0}\mbox{ and }\big[\nabla_{Y}\Theta_{\nu,D^{Y}}(Y)\big]_{\mathcal{I}_{Y}}=\bm{0}$. Combining with (\ref{Lsta-incl}), we get that there exists $H\in\partial_Zf(XY^{\mathbb{T}})$ such that $\bm{0}\in\big[HY+\lambda\nabla_{X}\Theta_{\nu,D^{X}}(X)+N_{\Omega^1}(X)\big]_{\mathcal{I}_{X}}$ and $\bm{0}\in\big[H^{\mathbb{T}}X+\lambda\nabla_{Y}\Theta_{\nu,D^{Y}}(Y)+N_{\Omega^2}(Y)\big]_{\mathcal{I}_{Y}}$. By Assumption \ref{mu-nu}, we have $\bm{0}\in\big[HY+\lambda\partial_{X}\Theta_{\nu,D^{X}}(X)+N_{\Omega^1}(X)\big]_{\mathcal{I}^c_{X}}$ and $\bm{0}\in\big[H^{\mathbb{T}}X+\lambda\partial_{Y}\Theta_{\nu,D^{Y}}(Y)+N_{\Omega^2}(Y)\big]_{\mathcal{I}^c_{Y}}$. As a result, (\ref{Rsta-incl}) holds and then $(X,Y)\in\mathcal{S}_{(\ref{FGL0R})}$.
\end{proof}
\begin{remark}
Recall from Proposition \ref{FGL0C-prps} (i) and Remark \ref{lc-pr-id}, any global minimizer of (\ref{FGL0C}) owns the column consistency property, but some local minimizers may not. Similar to Remark \ref{lc-pr-id}, under the condition $|\mathcal{I}_{X}\cap\mathcal{I}_{Y}|<d$, it is not hard to construct a local minimizer of (\ref{FGL0C}) without the column bound-$\nu$ property. This combined with Theorem \ref{stat-loc} (ii) yields that the column consistency property and column bound-$\nu$ property are necessary conditions for the global minimizers of problem (\ref{FGL0C}), but not for its non-global local minimizers. Moreover, continuing the same assumption $|\mathcal{I}_{X}\cap\mathcal{I}_{Y}|<d$ for a given stationary point $(X,Y)$ of (\ref{FGL0C}), we define $(U,V)\in\Omega^1\times\Omega^2$ by $U_{\mathcal{I}}=X_{\mathcal{I}}$, $U_{\mathcal{I}^c}=\bm{0}$, $V_{\mathcal{I}}=Y_{\mathcal{I}}$ and $V_{\mathcal{I}^c}=\bm{1}$ with $\mathcal{I}=\mathcal{I}_{X}\cap\mathcal{I}_{Y}$. It follows that $(U,V)\in\mathcal{S}_{(\ref{FGL0C})}$ but $(U,V)\notin\mathcal{S}^s_{(\ref{FGL0C})}$, which means $\mathcal{S}^s_{(\ref{FGL0C})}\varsubsetneq\mathcal{S}_{(\ref{FGL0C})}$.
\end{remark}
\begin{remark}
In problem (\ref{FGL0R}), the relaxation $\theta_\vartheta$ is defined by the capped-$\ell_1$ function on $[0,+\infty)$. Based on Theorem \ref{stat-loc} (i) and similar  analysis of \cite[Proposition 5]{LeThi2015}, we can also establish the equivalence on the global minimizers between problems (\ref{FGL0C}) and (\ref{FGL0R}) when $\theta_\vartheta$ in (\ref{FGL0R}) is replaced by the transformed SCAD or MCP function. However, due to the quadratic non-convex part of SCAD and MCP functions, it is difficult to establish the equivalence on their strong stationary points. In addition, for problems (\ref{FGL0C}) and (\ref{FGL0R}), it is worthwhile to delve into some properties of global minimizers and introduce another type of ``strong stationary point'', which is potentially helpful to  exclude some non-global stationary points. One avenue to pursue this enhancement is by incorporating second-order necessary optimality conditions and it is worthy of further study in the future. Moreover,  Theorem \ref{stat-loc} shows that the (strong) stationary point set of (\ref{FGL0R}) with capped-$\ell_1$ relaxation contains the global optimal solution set of problem (\ref{FGL0C}). In contrast, if $\theta_\rho$ is replaced by a convex relaxation instead of capped-$\ell_1$, the optimal solution set of the convex relaxation is not theoretically guaranteed to contain the global optimal solution set of  (\ref{FGL0C})  in general.
\end{remark}
In this section, problem (\ref{FGL0C}) and its relaxation problem (\ref{FGL0R}) are proved to have the relationships as shown in Fig.\,\ref{rela}. Finally, we discuss the relationships between some stationary points of (\ref{FGL0R}) under the DC formulation of it as follows.
\begin{figure}
\centering
\includegraphics[width=5.1in]{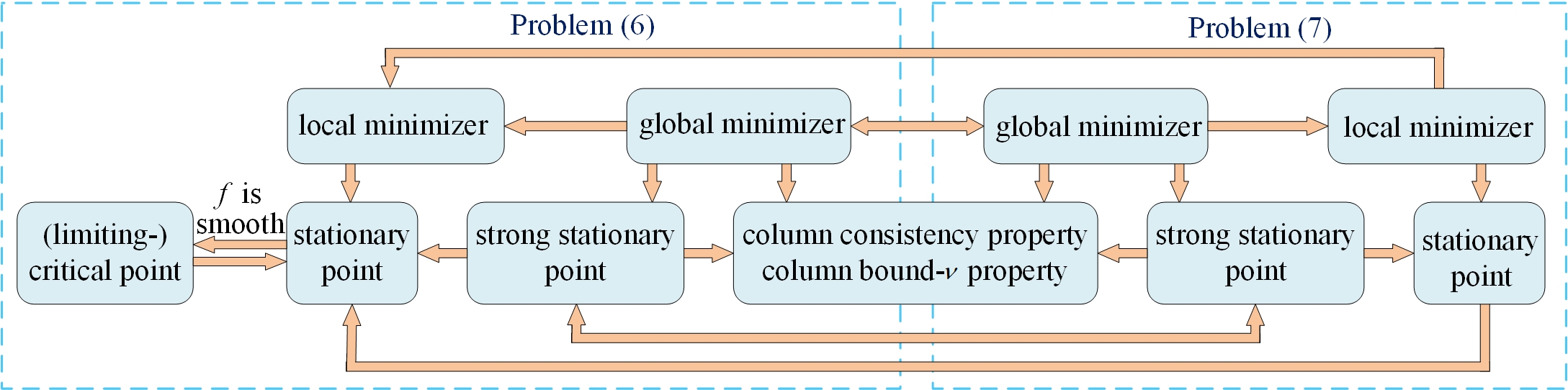}
\caption{Relationships between problem (\ref{FGL0C}) and its relaxation problem (\ref{FGL0R}) with $\nu$ satisfying Assumption \ref{mu-nu}.}\label{rela}
\end{figure}
\begin{remark}
When the loss function $f$ in (\ref{FGL0R}) satisfies the following condition: 
\begin{equation}\label{dc_cas}
\begin{split}
&\text{$f$ is smooth and }\text{there exist a strongly convex smooth function $\phi(X,Y)$, $L>0$}\\
&\text{and $l\geq0$ such that $L\phi(X,Y)-f(XY^{\mathbb{T}})$ and $f(XY^{\mathbb{T}})+l\phi(X,Y)$ are convex,}
\end{split}
\end{equation}
a DC algorithm has been proposed in \cite{LeThi2021arxiv} to obtain the critical points of an equivalent DC form of (\ref{FGL0R}). In particular, the function $f(Z)=\dfrac{1}{2}\Vert \mathcal{P}_{\varGamma}(Z-\bar{Z}^\diamond)\Vert_F^2$ in (\ref{lstsql}) fulfills (\ref{dc_cas}). For $\vartheta>0$ and $I\in[2]^d$, define $\bar\Theta_\vartheta(C):=\sum_{i=1}^d\bar\theta_\vartheta(\|C_i\|)\;\text{and}\;\bar\Theta_{\vartheta,I}(C):=\sum_{i=1}^d\bar{\theta}_{\vartheta,I_i}(\|C_i\|),\forall C\in\mathbb{R}^{r\times d}$ with $\bar\theta_{\vartheta,1}(t):=0$, $\bar\theta_{\vartheta,2}(t):=t/\vartheta-1$ and $\bar{\theta}_\vartheta(t):=\max_{j\in[2]}\{\bar\theta_{\vartheta,j}(t)\},\forall t\geq0$. It follows $\Theta_\vartheta(C)=\frac{1}{\vartheta}\sum_{i=1}^d\|C_i\|-\bar\Theta_\vartheta(C)$. Then, combining with the convexity of $f(XY^{\mathbb{T}})+l\phi(X,Y)$, problem (\ref{FGL0R}) with $f$ satisfying (\ref{dc_cas}) can be equivalently transformed into the following DC problem,
\begin{equation}\label{FGL0R-dc}
\min_{(X,Y)\in\Omega^1\times\Omega^2}f(XY^{\mathbb{T}})+l\phi(X,Y)+\lambda\sum_{i=1}^d(\|X_i\|+\|Y_i\|)-l\phi(X,Y)-\lambda\big(\bar\Theta_\nu(X)+\bar\Theta_\nu(Y)\big).
\end{equation}
Note that $\Theta_{\vartheta,I}(C)=\frac{1}{\vartheta}\sum_{i=1}^d\|C_i\|-\bar\Theta_{\vartheta,I}(C),\forall\vartheta>0$ and $I\in[2]^d$. Moreover, $\bar{\Theta}_{\nu,D^X}$ and $\bar{\Theta}_{\nu,D^Y}$ are smooth. Then we deduce that (\ref{Rsta-incl}) is equivalent to
\begin{equation*}
\bm{0}\in\partial\big(f(XY^{\mathbb{T}})+\frac{\lambda}{\vartheta}\sum_{i=1}^d(\|X_i\|+\|Y_i\|)\big)-\lambda\nabla\big(\bar\Theta_{\nu,D^{X}}(X)+\bar\Theta_{\nu,D^{Y}}(Y)\big),
\end{equation*}
with $D^X$ and $D^Y$  defined as in Definition \ref{Rsta}. This combined with $\nabla\big(\bar\Theta_{\nu,D^{X}}(X)+\bar\Theta_{\nu,D^{Y}}(Y)\big)\subseteq\partial\big(\bar\Theta_\nu(X)+\bar\Theta_\nu(Y)\big)$ and the model equivalence between (\ref{FGL0R}) and (\ref{FGL0R-dc}) implies that a strong stationary point $(X,Y)$ of (\ref{FGL0R}) owns  stronger optimality condition than the d-stationary point defined in \cite{Pang2017} when $\|X_i\|\neq\nu$ and $\|Y_i\|\neq\nu,\forall i\in[d]$, and the critical point used in \cite{LeThi2021arxiv} for (\ref{FGL0R-dc}). Note that a strong stationary point of (\ref{FGL0C}) owns the same optimality condition as that of (\ref{FGL0R}) by Theorem \ref{stat-loc} (iii).
\end{remark}

\subsection{The recovery of the rank of the target matrix}
In this part, inspired by \cite[Proposition 2.6]{Pan2022factor}, we show that some matrix couples can attain the same rank $r^\star\,(\leq d)$ as the target matrix $Z^\diamond$ under some assumptions related to $f$ and $\lambda$ for problem (\ref{FGL0C}). This result encompasses the associated (strong) stationary points and ensures the consistency of $\text{rank}(Z^\diamond)$ and $\text{rank}(XY^\mathbb{T})$ for any global minimizer $(X,Y)$ of problem (\ref{FGL0C}). In the following, we define the $2r^\star$-restricted smallest eigenvalue $\alpha$ of operator $\mathcal{A}:\mathbb{R}^{m\times n}\rightarrow\mathbb{R}^p$ as $\alpha=\min\{\|\mathcal{A}(Z)\|^2/\|Z\|_F^2:Z\neq\bm{0}\text{ and }\text{rank}(Z)\leq2r^\star\}$.

\begin{proposition}\label{rank-consis}
Let $f(Z):=h(\mathcal{A}(Z)-b)$, where $h:\mathbb{R}^p\rightarrow\mathbb{R}_+$ is a differentiable $\rho$-strongly convex function, $\mathcal{A}:\mathbb{R}^{m\times n}\rightarrow\mathbb{R}^p$ is a linear operator, and $b\in\mathbb{R}^p$ is a given vector that is dependent on $Z^\diamond$. Let $M_{r^\star}$ be a projection of $\mathcal{A}^*(b)$ onto the rank-$r^\star$ constraint set. Suppose the $2r^\star$-restricted smallest eigenvalue $\alpha$ of $\mathcal{A}$ satisfies 
$\alpha{\rho\sigma_{r^\star}(M_{r^\star})}\geq {(\sqrt{2}+1)\|\mathcal{A}^*(\nabla h(\mathcal{A}(M_{r^\star})-b))\|_F}$ and 
$\lambda$ in (\ref{FGL0C}) fulfills $\lambda\in\big[f(M_{r^\star}),\frac{\rho\alpha}{8}\big(\sigma_{r^\star}(M_{r^\star})-\frac{1}{\rho\alpha}\|\mathcal{A}^*(\nabla h(\mathcal{A}(M_{r^\star})-b)\big)\|)^2\big)$. Then, any matrix couple $(\bar{X},\bar{Y})$ with $F_0(\bar{X},\bar{Y})\leq f(M_{r^\star})+2\lambda r^\star$ satisfies that $\emph{\text{rank}}(\bar{X}\bar{Y}^\mathbb{T})=r^\star$.
\end{proposition}
\begin{proof}
Let $(\bar{X},\bar{Y})$ satisfy $F_0(\bar{X},\bar{Y})\leq f(M_{r^\star})+2\lambda r^\star$. By $h\geq0$, $\lambda\geq f(M_{r^\star})$ and (\ref{rank-rel}), we deduce that
\begin{equation*}
2\lambda\,\text{rank}(\bar{X}\bar{Y}^\mathbb{T})\leq \lambda(\text{nnzc}(\bar{X})+\text{nnzc}(\bar{Y}))\leq F_0(\bar{X},\bar{Y})\leq f(M_{r^\star})+2\lambda r^\star\leq\lambda+2\lambda r^\star.
\end{equation*}
Then, $\text{rank}(\bar{X}\bar{Y}^\mathbb{T})\leq r^\star+\frac{1}{2}$, which implies that $\text{rank}(\bar{X}\bar{Y}^\mathbb{T})\leq r^\star$.
	
Next, we only need to prove that $\text{rank}(\bar{X}\bar{Y}^\mathbb{T})<r^\star$ does not hold to obtain the statement that 
$\text{rank}(\bar{X}\bar{Y}^\mathbb{T})=r^\star$. Let $\bar{r}:=\text{rank}(\bar{X}\bar{Y}^\mathbb{T})$ and suppose $\bar{r}<r^\star$. Since $h$ is a differentiable $\rho$-strongly convex function and $\alpha$  is the $2r^\star$-restricted smallest eigenvalue of $\mathcal{A}$, we get
\begin{equation}\label{true_rank}
\begin{split}
&f(\bar{X}\bar{Y}^\mathbb{T})-f(M_{r^\star})\\
\geq&\,\langle\nabla h(\mathcal{A}(M_{r^\star})-b),\mathcal{A}(\bar{X}\bar{Y}^\mathbb{T}-M_{r^\star})\rangle+\frac{\rho}{2}\|\mathcal{A}(\bar{X}\bar{Y}^\mathbb{T}-M_{r^\star})\|^2
\\
\geq&\,\langle\mathcal{A}^*(\nabla h(\mathcal{A}(M_{r^\star})-b)),\bar{X}\bar{Y}^\mathbb{T}-M_{r^\star}\rangle+\frac{\rho\alpha}{2}\|\bar{X}\bar{Y}^\mathbb{T}-M_{r^\star}\|_F^2\\
=&\,\frac{\rho\alpha}{2}\|\bar{X}\bar{Y}^\mathbb{T}-\bar{M}\|_F^2-\frac{1}{2\rho\alpha}\|\mathcal{A}^*(\nabla h(\mathcal{A}(M_{r^\star})-b))\|_F^2\\
\geq&\,\frac{\rho\alpha}{2}\min_{\text{rank}(Z)\leq\bar{r}}\|Z-\bar{M}\|_F^2-\frac{1}{2\rho\alpha}\|\mathcal{A}^*(\nabla h(\mathcal{A}(M_{r^\star})-b))\|_F^2\\
=&\,\frac{\rho\alpha}{2}\sum_{i=\bar{r}+1}^n\sigma_i^2(\bar{M})-\frac{1}{2\rho\alpha}\|\mathcal{A}^*(\nabla h(\mathcal{A}(M_{r^\star})-b))\|_F^2\\
\geq&\,\frac{\rho\alpha}{2}(r^\star-\bar{r})\sigma_{r^\star}^2(\bar{M})-\frac{1}{2\rho\alpha}\|\mathcal{A}^*(\nabla h(\mathcal{A}(M_{r^\star})-b))\|_F^2,
\end{split}
\end{equation}
where $\bar{M}=M_{r^\star}-(\rho\alpha)^{-1}\mathcal{A}^*(\nabla h(\mathcal{A}(M_{r^\star})-b))$. By the definition of $\bar{M}$, we have that
$\sigma_{r^\star}(\bar{M}) \geq \sigma_{r^\star}(M_{r^\star}) - (\rho\alpha)^{-1}
\|\mathcal{A}^*(\nabla h(\mathcal{A}(M_{r^\star})-b))\|.$
This, together with the selection range of $\lambda$, we can deduce that $\frac{\rho\alpha}{4}(r^\star-\bar{r})\sigma^2_{r^\star}(\bar{M})\geq\frac{\rho\alpha}{4}(r^\star-\bar{r})\big[\sigma_{r^\star}(M_{r^\star})-(\rho\alpha)^{-1}\|\mathcal{A}^*(\nabla h(\mathcal{A}(M_{r^\star})-b))\|\big]^2>2\lambda(r^\star-\bar{r})$. 
Moreover, due to $\alpha\geq\frac{(\sqrt{2}+1)\|\mathcal{A}^*(\nabla h(\mathcal{A}(M_{r^\star})-b))\|_F}{\rho\sigma_{r^\star}(M_{r^\star})}$, 
we have $\frac{\rho\alpha}{4}(r^\star-\bar{r})
\sigma_{r^\star}^2(\bar{M})-\frac{1}{2\rho\alpha}\|\mathcal{A}^*(\nabla h(\mathcal{A}(M_{r^\star})-b))\|_F^2\geq0$. Combining with (\ref{true_rank}), we obtain $f(\bar{X}\bar{Y}^\mathbb{T})-f(M_{r^\star})>2\lambda(r^\star-\bar{r})$. 
It follows that $F_0(\bar{X},\bar{Y})\geq f(\bar{X}\bar{Y}^\mathbb{T})+2\lambda\bar{r}>f(M_{r^\star})+2\lambda r^\star$, which contradicts the assumption 
that 
$F_0(\bar{X},\bar{Y})\leq f(M_{r^\star})+2\lambda r^\star$. As a result, we get $\text{rank}(\bar{X}\bar{Y}^\mathbb{T})=r^\star$.
\end{proof}

Under the mild assumption that $\{Z\in\mathcal{G}_{(\ref{LRMPR})}:\|Z\|\leq\varsigma^2\}\neq\emptyset$ in Proposition \ref{bd-rela} (iii), we have that ${\text{nnzc}}(X)={\text{nnzc}}(Y)=\text{rank}(XY^\mathbb{T})$ and $XY^\mathbb{T}\in\mathcal{G}_{(\ref{LRMPR})},\,\forall(X,Y)\in\mathcal{G}_{(\ref{FGL0C})}$ by Proposition \ref{bd-rela} (iii) and Proposition \ref{FGL0-prps} (i). It follows from Proposition \ref{rank-consis} that $\text{rank}(XY^\mathbb{T})=r^\star,\forall(X,Y)\in\mathcal{G}_{(\ref{FGL0C})}$. Therefore, the condition in Proposition \ref{bd-rela} (iii) and Proposition \ref{rank-consis} can guarantee the consistency of $\text{rank}(Z^\diamond)$ and $\text{rank}(XY^\mathbb{T})$ for any $(X,Y)\in\mathcal{G}_{(\ref{FGL0C})}$.

\section{Algorithms}\label{section5}

Throughout this section, we assume that $f$ is differentiable and its gradient $\nabla f$ is Lipschitz continuous on $\{XY^{\mathbb{T}}:(X,Y)\in\Omega^1\times\Omega^2\}$ with modulus $L_f$. We propose an alternating proximal gradient algorithm and prove its  convergence to a strong stationary point of problem (\ref{FGL0C}). Different from the PALM algorithm in \cite{Bolte2014Proximal}, its inertial version in \cite{SabachSIIMS2016} and the BPL algorithm in \cite{YinJSC2017} for problems (\ref{FGL0C}) and (\ref{FGL0R}), the proximal operators in our algorithm are updated step by step. Based on this, we can obtain  higher quality stationary points in theory. In addition, we are also interested in giving further theoretical results based on the PALM algorithm for (\ref{FGL0C}). So we make some improvements on the PALM algorithm and give some progress on the relevant  properties of its iterates. In particular, we also incorporate a dimension reduction technique without affecting the theoretical results for the proposed algorithms.

Based on Proposition \ref{FGL0C-prps} and Proposition \ref{gl-lc-sta}, it holds that for problems (\ref{FGL0C}) and (\ref{FGL0R}), the operation of (\ref{rdc-elm}) on any local minimizer (stationary point) also gives a local minimizer (stationary point) with the column consistency property from the global minimizers. This inspires us to incorporate this process into the algorithms and explore the impact on their generated iterates. So, we specially define the operation in (\ref{rdc-elm}) as follows.
\begin{definition}
For any $(X,Y)\in\mathbb{R}^{m\times d}\times\mathbb{R}^{n\times d}$, we define function $S:\mathbb{R}^{m\times d}\times\mathbb{R}^{n\times d}\rightarrow\mathbb{R}^{m\times d}\times\mathbb{R}^{n\times d}$ by 
\begin{equation*}
S(X,Y)=(X^s,Y^s),
\end{equation*}
with $(X^s,Y^s)$ defined as in (\ref{rdc-elm}), and call $S(X,Y)$ the sparsity couple of $(X,Y)$.
\end{definition}
In subsequent algorithms, we refer to the use of the function $S$ as sparsity consistency processing.

\subsection{Alternating proximal gradient algorithm with adaptive indicator and line search for problem \eqref{FGL0C}
}
\label{th_rs_ag1}

We design an algorithm as shown in Algorithm \ref{FGL0alg1} (Alg.\,\ref{FGL0alg1}) for solving problem (\ref{FGL0C}) through the relaxation problem \eqref{FGL0R} based on their equivalence
of strong stationary points.

Note that when $k\geq K$, it holds that $I^k=D^{X^k}$ and $J^k=D^{Y^k}$, and when $\mathcal{K}=\emptyset$, we have $(X^{k+1},Y^{k+1})=(\bar{X}^{k+1},\bar{Y}^{k+1}),\forall k\geq 0$. Moreover, we use the line search method in each step to improve the convergence rate, which will be verified in later numerical experiments.
\begin{algorithm}[h!]
\caption{({\bf{Alg.\,\ref{FGL0alg1}}})
An alternating proximal gradient algorithm with adaptive indicator and line search for problem  (\ref{FGL0C})}
\label{FGL0alg1}
{\bf Initialization:} Choose $(X^0,Y^0)\in\Omega^1\times\Omega^2$, $0<c_{\min}<c_{\max}$, $0<\underline{\iota}<\overline{\iota}$, $\rho>1$, $\mathcal{K}\subseteq\mathbb{N}$ and $K\in\mathbb{N}$. Let $\{\nu_k\}$ be positive sequences satisfying
\begin{equation*}
\nu_k=\nu,\forall k\geq K.
\end{equation*}
Set $k=0$.
	
{\bf Step 1:} Let $I^{k}\in\{1,2\}^d$ with $I^{k}_i:=\max\{j\in\{1,2\}:\theta_{\nu_k,j}(\|X^k_i\|)=\theta_{\nu_k}(\|X^k_i\|)\},\forall i\in[d]$. Choose $\iota_{1,k}^B\in[\underline{\iota},\overline{\iota}]$ and $c_{1,k}\in [c_{\min},c_{\max}]$.
For $l=1,2,\ldots$, do\\
(1a) let $\iota_{1,k}=\iota_{1,k}^B\rho^{l-1}$;\\
(1b) compute ${\bar{X}}^{k+1}(\iota_{1,k})$ by the following problem,
\begin{equation}\label{FGL0sub1}
\begin{split}
\bar{X}^{k+1}(\iota_{1,k})=\,{{\rm{arg}}\min}_{X\in\Omega^1}S_{1,k}(X;\iota_{1,k}):=\,&f(X^k{Y^k}^{\mathbb{T}})\\&+\langle\nabla_Zf(X^k{Y^k}^{\mathbb{T}}){Y^k},X-X^k\rangle\\&+\frac{\iota_{1,k}}{2}\|X-X^k\|_F^2+\lambda\Theta_{\nu_k,I^k}(X);
\end{split}
\end{equation}
(1c) if ${\bar{X}}^{k+1}(\iota_{1,k})$ satisfies
\begin{equation}\label{objXineq}
\begin{split}
f({\bar{X}}^{k+1}(\iota_{1,k}){Y^k}^{\mathbb{T}})+\lambda\Theta_{\nu_k}({\bar{X}}^{k+1}(\iota_{1,k}))\leq\, &f(X^k{Y^k}^{\mathbb{T}})+\lambda\Theta_{\nu_k}(X^k)\\&-c_{1,k}\|{\bar{X}}^{k+1}(\iota_{1,k})-X^k\|_F^2,
\end{split}
\end{equation}
set ${\bar{X}}^{k+1}={\bar{X}}^{k+1}(\iota_{1,k})$, $\bar\iota_{1,k}=\iota_{1,k}$ and go to Step 2.\\
{\bf Step 2:} Let $J^k\in\{1,2\}^d$ with $J^{k}_i:=\max\{j\in\{1,2\}:\theta_{\nu_k,j}(\|Y^k_i\|)=\theta_{\nu_k}(\|Y^k_i\|)\},\forall i\in[d]$. Choose $\iota_{2,k}^B\in[\underline{\iota},\overline{\iota}]$ and $c_{2,k}\in[c_{\min},c_{\max}]$.
For $l=1,2,\ldots$, do\\
(2a) let $\iota_{2,k}=\iota_{2,k}^B\rho^{l-1}$;\\
(2b) compute ${\bar{Y}}^{k+1}(\iota_{2,k})$ by the following problem,
\begin{equation}\label{FGL0sub2}
\begin{split}
\bar{Y}^{k+1}(\iota_{2,k})=\,{{\rm{arg}}\min}_{Y\in\Omega^2} S_{2,k}(Y;\iota_{2,k}):=\,&f(\bar{X}^{k+1}{Y^k}^{\mathbb{T}})\\&+\langle\nabla_Zf(\bar{X}^{k+1}{Y^k}^{\mathbb{T}})^{\mathbb{T}}\bar{X}^{k+1},Y-Y^k\rangle\\&+\frac{\iota_{2,k}}{2}\|Y-Y^k\|_F^2+\lambda\Theta_{\nu_k,J^k}(Y);
\end{split}
\end{equation}
(2c) if ${\bar{Y}}^{k+1}(\iota_{2,k})$ satisfies
\begin{equation}\label{objYineq}
\begin{split}
f({\bar{X}}^{k+1}({\bar{Y}}^{k+1}(\iota_{2,k}))^{\mathbb{T}})+\lambda\Theta_{\nu_k}({\bar{Y}}^{k+1}(\iota_{2,k}))\leq\, &f({\bar{X}}^{k+1}{Y^k}^{\mathbb{T}})+\lambda\Theta_{\nu_k}(Y^k)\\&-c_{2,k}\|{\bar{Y}}^{k+1}(\iota_{2,k})-Y^k\|_F^2,
\end{split}
\end{equation}
set ${\bar{Y}}^{k+1}={\bar{Y}}^{k+1}(\iota_{2,k})$, $\bar\iota_{2,k}=\iota_{2,k}$ and go to Step 3.\\
{\bf Step 3:} Let $(X^{k+1},Y^{k+1})=S(\bar{X}^{k+1},\bar{Y}^{k+1})$ if $k\in\mathcal{K}$, and $(X^{k+1},Y^{k+1})=(\bar{X}^{k+1},\bar{Y}^{k+1})$ otherwise. Update $k\leftarrow k+1$ and return to Step 1.
\end{algorithm}

First, we analyze some basic properties of Alg.\,\ref{FGL0alg1}, which will be used in its forthcoming convergence analysis.
\begin{proposition}\label{limprop}
Let $\{(X^k,Y^k)\}$ be the iteration sequence generated by Alg.\,\ref{FGL0alg1}. Then,
\begin{itemize}
\item[(i)] $\bar{\iota}_{1,k}\leq\max\{\underline\iota,2\rho c_{1,k},\rho L_f\|Y^k\|^2\}$ and $\bar{\iota}_{2,k}\leq\max\{\underline\iota,2\rho c_{2,k},\rho L_f\|\bar{X}^{k+1}\|^2\}$;
\item[(ii)] $\lim_{k\rightarrow\infty}F(X^k,Y^k)$ exists and $\lim_{k\rightarrow\infty}\|\bar{X}^{k+1}-X^k\|_F=\lim_{k\rightarrow\infty}\|\bar{Y}^{k+1}-Y^k\|_F=0$.
\end{itemize}
\end{proposition}
\begin{proof}
(i) Suppose that
\begin{equation*}
\iota_{1,k}\geq\max\{2c_{1,k},L_f\|Y^k\|^2\}\text{ and }\iota_{2,k}\geq\max\{2c_{2,k},L_f\|\bar{X}^{k+1}\|^2\}. 
\end{equation*}
Then, we have
	
\begin{equation*}
\begin{split}
&f(X^k{Y^k}^{\mathbb{T}})+\lambda\Theta_{\nu_k}(X^k)=f(X^k{Y^k}^{\mathbb{T}})+\lambda\Theta_{\nu_k,I^k}(X^k)=S_{1,k}(X^k;\iota_{1,k})\\
\geq&\, S_{1,k}(\bar{X}^{k+1}(\iota_{1,k});\iota_{1,k})+\frac{\iota_{1,k}}{2}\|\bar{X}^{k+1}(\iota_{1,k})-X^k\|_F^2\\
\geq&\, S_{1,k}(\bar{X}^{k+1}(\iota_{1,k});\iota_{1,k})+c_{1,k}\|\bar{X}^{k+1}(\iota_{1,k})-X^k\|_F^2\\
=&\, f(X^k{Y^k}^{\mathbb{T}})+\langle\nabla_Zf(X^k{Y^k}^{\mathbb{T}}){Y^k},\bar{X}^{k+1}(\iota_{1,k})-X^k\rangle+\frac{\iota_{1,k}}{2}\|\bar{X}^{k+1}(\iota_{1,k})-X^k\|_F^2
\\&\,+\lambda\Theta_{\nu_k,I^k}(\bar{X}^{k+1}(\iota_{1,k}))+c_{1,k}\|\bar{X}^{k+1}(\iota_{1,k})-X^k\|_F^2\\
\geq&\, f(\bar{X}^{k+1}(\iota_{1,k}){Y^k}^{\mathbb{T}})+\lambda\Theta_{\nu_k,I^k}(\bar{X}^{k+1}(\iota_{1,k}))+c_{1,k}\|\bar{X}^{k+1}(\iota_{1,k})-X^k\|_F^2\\
\geq&\, f(\bar{X}^{k+1}(\iota_{1,k}){Y^k}^{\mathbb{T}})+\lambda\Theta_{\nu_k}(\bar{X}^{k+1}(\iota_{1,k}))+c_{1,k}\|\bar{X}^{k+1}(\iota_{1,k})-X^k\|_F^2,
\end{split}
\end{equation*}
where the first inequality follows from the strong convexity of $S_{1,k}$ and that $\bar{X}^{k+1}(\iota_{1,k})$ is the minimizer of the subproblem in Step (1b), the second inequality is due to $\iota_{1,k}\geq 2c_{1,k}$, the third inequality is based on that $\nabla_Zf(X{Y^k}^{\mathbb{T}})Y^k$ is Lipschitz continuous with modulus $L_f\|Y^k\|^2\leq\iota_{1,k}$, and the last inequality is from the definition of $\Theta_\nu$. Similarly, we have $f({\bar{X}}^{k+1}{Y^k}^{\mathbb{T}})+\lambda\Theta_{\nu_k}(Y^k)\geq\,f({\bar{X}}^{k+1}({\bar{Y}}^{k+1}(\iota_{2,k}))^{\mathbb{T}})+\lambda\Theta_{\nu_k}({\bar{Y}}^{k+1}(\iota_{2,k}))+c_{2,k}\|{\bar{Y}}^{k+1}(\iota_{2,k})-Y^k\|_F^2$.
Hence, we get that (\ref{objXineq}) and (\ref{objYineq}) hold when $\iota_{1,k}\geq\max\{2c_{1,k},L_f\|Y^k\|^2\}$ and $\iota_{2,k}\geq\max\{2c_{2,k},L_f\|\bar{X}^{k+1}\|^2\}$. By $\rho>1$, this implies that 
\begin{equation}\label{sub-iter-n}
\bar{\iota}_{1,k}\leq\max\{2\rho c_{1,k},\rho L_f\|Y^k\|^2\}\text{ and }\bar{\iota}_{2,k}\leq\max\{2\rho c_{2,k},\rho L_f\|\bar{X}^{k+1}\|^2\}.
\end{equation}
Due to $\iota_{1,k}^B,\iota_{2,k}^B\in[\underline{\iota},\overline{\iota}]$ and the expressions of $\iota_{1,k}$ and $\iota_{2,k}$ in (1a) and (2a) of Alg.\,\ref{FGL0alg1}, we have $\iota_{1,k}\geq\underline{\iota}$ and $\iota_{2,k}\geq\underline{\iota}$. As a result,  statement (i) holds.
	
(ii) By the definition of $(X^{k+1},Y^{k+1})$, we get $\Theta_{\nu_k}(X^{k+1})\leq\Theta_{\nu_k}({\bar{X}}^{k+1})$ and $\Theta_{\nu_k}(Y^{k+1})\leq\Theta_{\nu_k}({\bar{Y}}^{k+1})$. Combining with $X^{k+1}(Y^{k+1})^{\mathbb{T}}={\bar{X}}^{k+1}({\bar{Y}}^{k+1})^{\mathbb{T}}$, (\ref{objXineq}) and (\ref{objYineq}), we obtain that for any $k\geq K$,
\begin{equation}\label{limxk}
\begin{split}
&F(X^k,Y^k)-F(X^{k+1},Y^{k+1})\\
=&\,f(X^k{Y^k}^{\mathbb{T}})+\lambda\Theta_{\nu}(X^k)+\lambda\Theta_{\nu}(Y^k)-f(X^{k+1}(Y^{k+1})^{\mathbb{T}})-\lambda\Theta_{\nu}(X^{k+1})-\lambda\Theta_{\nu}(Y^{k+1})\\
\geq&\,f(X^k{Y^k}^{\mathbb{T}})+\lambda\Theta_{\nu}(X^k)+\lambda\Theta_{\nu}(Y^k)-f({\bar{X}}^{k+1}({\bar{Y}}^{k+1})^{\mathbb{T}})-\lambda\Theta_{\nu}({\bar{X}}^{k+1})-\lambda\Theta_{\nu}({\bar{Y}}^{k+1})\\
=&\,f(X^k{Y^k}^{\mathbb{T}})+\lambda\Theta_{\nu}(X^k)-f({\bar{X}}^{k+1}{Y^k}^{\mathbb{T}})-\lambda\Theta_{\nu}({\bar{X}}^{k+1})
\\&\,+f({\bar{X}}^{k+1}{Y^k}^{\mathbb{T}})+\lambda\Theta_{\nu}(Y^k)-f({\bar{X}}^{k+1}({\bar{Y}}^{k+1})^{\mathbb{T}})-\lambda\Theta_{\nu}({\bar{Y}}^{k+1})\\
\geq&\,\min\{c_{1,k},c_{2,k}\}\big(\|{\bar{X}}^{k+1}-X^k\|_F^2+\|{\bar{Y}}^{k+1}-Y^k\|_F^2\big)\\
\geq&\,c_{\min}\big(\|{\bar{X}}^{k+1}-X^k\|_F^2+\|{\bar{Y}}^{k+1}-Y^k\|_F^2\big),
\end{split}
\end{equation}
which implies that $F(X^k,Y^k)$ is non-increasing with respect to $k$ when $k\geq K$. This together with the boundedness of $\Omega^1\times\Omega^2$ yields the existence of $\lim_{k\rightarrow\infty}F(X^k,Y^k)$. By (\ref{limxk}) and $c_{\min}>0$, we get $\lim_{k\rightarrow\infty}\|{\bar{X}}^{k+1}-X^k\|_F=\lim_{k\rightarrow\infty}\|{\bar{Y}}^{k+1}-Y^k\|_F=0$.
\end{proof}
Next, we show the convergence of the iterates generated by Alg.\,\ref{FGL0alg1} to strong stationary points of (\ref{FGL0C}).
\begin{theorem}\label{alg1-conv-thm}
Let $\nu$ satisfy Assumption \ref{mu-nu}. For the sequence $\{(X^k,Y^k)\}$ generated by Alg.\,\ref{FGL0alg1}, the following statements hold.
\begin{itemize}
\item[(i)] There exist $\hat{K}>0$ and $\hat{\mathcal{I}}_X,\hat{\mathcal{I}}_Y\subseteq[d]$ such that for any $k\geq\hat{K}$, it holds that $\mathcal{I}_{X^k}=\hat{\mathcal{I}}_X$, $\mathcal{I}_{Y^k}=\hat{\mathcal{I}}_Y$, $(X^k,Y^k)=(\bar{X}^k,\bar{Y}^k)$ and $(X^k,Y^k)$ owns the column bound-$\nu$ property. 
\item[(ii)] If we set $|\mathcal{K}|=+\infty$, then $(X^k,Y^k)$ owns the column consistency property for all $k>\min\{k\in\mathcal{K}:k\geq\hat{K}-1\}$.
\item[(iii)] For any accumulation point $(X^*,Y^*)$ of $\{(X^k,Y^k)\}$, $(X^*,Y^*)\in\mathcal{S}^s_{(\ref{FGL0C})}$ if $|\mathcal{K}|=+\infty$, and $S(X^*,Y^*)\in\mathcal{S}^s_{(\ref{FGL0C})}$ otherwise.
\end{itemize}
\end{theorem}
\begin{proof}
(i) Based on $\lim_{k\rightarrow\infty}\|{\bar{X}}^{k+1}-X^k\|_F=\lim_{k\rightarrow\infty}\|{\bar{Y}}^{k+1}-Y^k\|_F=0$ in Proposition \ref{limprop} and Assumption \ref{mu-nu}, we have that there exists a $\bar{K}\geq K$ such that for any $k\geq\bar{K}$ and $i\in[d]$,
\begin{equation}\label{distlim}
\max\{\|{\bar{X}}^{k+1}_i-X^k_i\|,\|{\bar{Y}}^{k+1}_i-Y^k_i\|\}<\min\big\{({\lambda}/{\nu}-\kappa)/\max\{\underline{\iota},2\rho c_{\max},\rho L_fd\varsigma^2\},\Mbound-\nu\big\}.
\end{equation}
Moreover, it follows from the optimality conditions of (\ref{FGL0sub1}) and (\ref{FGL0sub2}) with $\iota_{1,k}=\bar\iota_{1,k}$ and $\iota_{2,k}=\bar\iota_{2,k}$ that
	
\begin{equation}\label{Xincl}
\bm{0}\in\nabla_Zf(X^k{Y^k}^{\mathbb{T}}){Y^k}+\bar\iota_{1,k}({\bar{X}}^{k+1}-X^k)+\lambda\partial_X\Theta_{\nu,I^k}({\bar{X}}^{k+1})+N_{\Omega^1}({\bar{X}}^{k+1}),
\end{equation}
and
\begin{equation}\label{Yincl}
\bm{0}\in\nabla_Zf({\bar{X}}^{k+1}{Y^k}^{\mathbb{T}})^{\mathbb{T}}{{\bar{X}}^{k+1}}+\bar\iota_{2,k}({\bar{Y}}^{k+1}-Y^k)+\lambda\partial_Y\Theta_{\nu,J^k}({\bar{Y}}^{k+1})+N_{\Omega^2}({\bar{Y}}^{k+1}).
\end{equation}
Suppose there exist $\bar{k}\geq\bar{K}$ and $\bar{i}\in[d]$ such that $\|X^{\bar{k}}_{\bar{i}}\|<\nu$. Then $I^{\bar{k}}_{\bar{i}}=1$ and hence $\partial_{X_{\bar{i}}}\Theta_{\nu,I^{\bar{k}}}({\bar{X}}^{\bar{k}+1})=\frac{1}{\nu}\partial_{X_{\bar{i}}}\|{\bar{X}}^{\bar{k}+1}_{\bar{i}}\|$. Combining with Assumption \ref{mu-nu}, Proposition \ref{limprop} (i) and (\ref{distlim}), we have $N_{\Omega^1_{\bar{i}}}({\bar{X}}^{\bar{k}+1}_{\bar{i}})=\{\bm{0}\}$ and $\|\nabla_Zf(X^{\bar{k}}{Y^{\bar{k}}}^{\mathbb{T}}){Y^{\bar{k}}_{\bar{i}}}+\bar\iota_{1,\bar{k}}({\bar{X}}^{\bar{k}+1}_{\bar{i}}-X^{\bar{k}}_{\bar{i}})\|<{\lambda}/{\nu}$. If ${\bar{X}}^{\bar{k}+1}_{\bar{i}}\neq\bm{0}$, we have $\bm{0}\notin\nabla_Zf(X^{\bar{k}}{Y^{\bar{k}}}^{\mathbb{T}}){Y^{\bar{k}}_{\bar{i}}}+\bar\iota_{1,{\bar{k}}}({\bar{X}}^{{\bar{k}}+1}_{\bar{i}}-X^{\bar{k}}_{\bar{i}})+\lambda\nabla_{{X}_{\bar{i}}}\Theta_{\nu,I^{\bar{k}}}({\bar{X}}^{{\bar{k}}+1})+N_{\Omega^1_{\bar{i}}}({\bar{X}}^{{\bar{k}}+1}_{\bar{i}})$, which contradicts (\ref{Xincl}). Hence, we obtain ${\bar{X}}^{\bar{k}+1}_{\bar{i}}=\bm{0}$. Combining with the definition of $X^{k+1}$ in step 3 of Alg.\,\ref{FGL0alg1}, we have $X^{\bar{k}+1}_{\bar{i}}=\bm{0}$ and hence $\mathcal{I}_{X^{k+1}}\subseteq\mathcal{I}_{\bar{X}^{k+1}}\subseteq\mathcal{I}_{X^k},\forall k\geq\bar{K}$. Similarly, if there exist $\tilde{k}\geq\bar{K}$ and $\tilde{i}\in[d]$ such that $\|Y^{\tilde{k}}_{\tilde{i}}\|<\nu$, we have ${\bar{Y}}^{\tilde{k}+1}_{\tilde{i}}=\bm{0}$ by (\ref{Yincl}). Moreover, $\mathcal{I}_{Y^{k+1}}\subseteq\mathcal{I}_{\bar{Y}^{k+1}}\subseteq\mathcal{I}_{Y^k},\forall k\geq\bar{K}$. 
As a result, there exist $\hat{K}>\bar{K}$ and $\hat{\mathcal{I}}_X,\hat{\mathcal{I}}_Y\subseteq[d]$ such that $\mathcal{I}_{\bar{X}^k}=\mathcal{I}_{X^k}=\hat{\mathcal{I}}_X$, $\mathcal{I}_{\bar{Y}^k}=\mathcal{I}_{Y^k}=\hat{\mathcal{I}}_Y$ and $(X^k,Y^k),k\geq\hat{K}$ own the column bound-$\nu$ property. It follows from $\mathcal{I}_{\bar{X}^k}=\mathcal{I}_{X^k}$, $\mathcal{I}_{\bar{Y}^k}=\mathcal{I}_{Y^k},\forall k\geq\hat{K}$ that $(X^k,Y^k)=(\bar{X}^k,\bar{Y}^k),\forall k\geq\hat{K}$.
	
(ii) Due to $(X^{k+1},Y^{k+1})=S(\bar{X}^{k+1},\bar{Y}^{k+1}),\forall k\in\mathcal{K}$, we obtain $\mathcal{I}_{X^{k+1}}=\mathcal{I}_{Y^{k+1}},\forall k\in\mathcal{K}$. This together with $\mathcal{I}_{X^k}=\hat{\mathcal{I}}_X$, $\mathcal{I}_{Y^k}=\hat{\mathcal{I}}_Y,\forall k\geq\hat{K}$ in (i) and $|\mathcal{K}|=+\infty$ implies $\mathcal{I}_{X^{k+1}}=\mathcal{I}_{Y^{k+1}}=\hat{\mathcal{I}}_X=\hat{\mathcal{I}}_Y,\forall k\geq\bar{K}:=\min\{k\in\mathcal{K}:k\geq\hat{K}-1\}$. Thus, $(X^k,Y^k)$ owns the column consistency property for all $k>\bar{K}$.
	
(iii) By (i), we have $I^k=I^{\hat{K}}$ and $J^k=J^{\hat{K}},\forall k\geq\hat{K}$. Due to the boundedness of $\Omega^1$ and $\Omega^2$, there exists at least one accumulation point. Suppose $(X^*,Y^*)$ an accumulation point of $(X^k,Y^k)$ with subsequence $\{k_l\}$. Accordingly, by (i), $D^{X^*}=I^{\hat{K}}=I^k$ and $D^{Y^*}=J^{\hat{K}}=J^k,\forall k\geq\hat{K}$. Then, in view of (\ref{Xincl}) and $X^{k}=\bar{X}^{k},\forall k\geq \hat{K}$ in (i), there exists $P^{k_l+1}\in\partial_X\Theta_{\nu,I^{\hat{K}}}(X^{k_l+1})$ such that	
\begin{equation}\label{nom-ieq}
\langle \nabla_Zf(X^{k_l}{Y^{k_l}}^{\mathbb{T}}){Y^{k_l}}+\bar\iota_{1,k_l}(X^{k_l+1}-X^{k_l})+\lambda P^{k_l+1},X-X^{k_l+1}\rangle\geq 0,\forall X\in\Omega^1.
\end{equation}
Because $\lim_{k\rightarrow\infty}\|X^{k+1}-X^k\|_F=0$ and $\lim_{l\rightarrow\infty}X^{k_l}=X^*$, we have $\lim_{k_l\rightarrow\infty}X^{k_l+1}=X^*$. It then follows from the boundedness and upper semicontinuity of $\partial_X\Theta_{\nu,I^{\hat{K}}}$ that there exists a subsequence of $\{k_l\}$ (also denoted by $\{k_l\}$) such that $\lim_{l\rightarrow\infty}P^{k_l+1}=P^*\in\partial_X \Theta_{\nu,I^{\hat{K}}}(X^*)$. Combining with (\ref{nom-ieq}), the continuity of $\nabla f$ and boundedness of $\bar\iota_{1,k_l}$, let $l\rightarrow\infty$ in (\ref{nom-ieq}), we obtain $\langle \nabla_Zf(X^*{Y^*}^{\mathbb{T}}){Y^*}+\lambda P^*,X-X^*\rangle\geq 0,\forall X\in\Omega^1$. Whence, by $D^{X^*}=I^{\hat{K}}$, we obtain
\begin{equation}\label{sst1-pr}
\bm{0}\in \nabla_Zf(X^*{Y^*}^{\mathbb{T}}){Y^*}+\lambda\partial_X \Theta_{\nu,D^{X^*}}(X^*)+N_{\Omega^1}(X^*).
\end{equation}
Similarly, by (\ref{Yincl}), we deduce

\begin{equation}\label{sst2-pr}
\bm{0}\in \nabla_Zf(X^*{Y^*}^{\mathbb{T}})^{\mathbb{T}}{X^*}+\lambda\partial_Y \Theta_{\nu,D^{Y^*}}(Y^*)+N_{\Omega^2}(Y^*).
\end{equation}
Thus, $(X^*,Y^*)\in\mathcal{S}_{(\ref{FGL0R})}$. In the case of $|\mathcal{K}|=+\infty$, it follows from (ii) that $\mathcal{I}_{X^*}=\mathcal{I}_{Y^*}$. Then, by Definition \ref{CR-sstr} and Theorem \ref{stat-loc}, $(X^*,Y^*)\in\mathcal{S}^s_{(\ref{FGL0C})}$. Next, if $|\mathcal{K}|<+\infty$ in Alg.\,\ref{FGL0alg1}, combining Proposition \ref{gl-lc-sta} (ii), the definition of $S$ and Theorem \ref{stat-loc}, we get $S(X^*,Y^*)\in\mathcal{S}^s_{(\ref{FGL0C})}$.
\end{proof}

\begin{remark}\label{chag_fsb}
Define 
\begin{equation}\label{hat_omega}
{\hat\Omega}^1:=\{X\in\Omega^1:X_{\hat{\mathcal{I}}^c_X}=\bm{0}\}\text{ and }{\hat\Omega}^2:=\{Y\in\Omega^2:Y_{\hat{\mathcal{I}}^c_Y}=\bm{0}\}
\end{equation}
with $\hat{\mathcal{I}}_X$ and $\hat{\mathcal{I}}_Y$ given as in Theorem \ref{alg1-conv-thm}. By Theorem \ref{alg1-conv-thm} (i), the iterate $(X^k,Y^k)$ of Alg.\,\ref{FGL0alg1} satisfies $\mathcal{I}_{\bar{X}^k}=\hat{\mathcal{I}}_X,\mathcal{I}_{\bar{Y}^k}=\hat{\mathcal{I}}_Y$ and $\|\bar{X}^k_i\|\geq\nu\;\forall i\in\hat{\mathcal{I}}_X,\|\bar{Y}^k_j\|\geq\nu\;\forall j\in\hat{\mathcal{I}}_Y$, for any $k\geq\hat{K}$. This implies that $({\bar{X}}^k,{\bar{Y}}^k)\in\hat\Omega^1\times\hat\Omega^2$ and $I^k_i=J^k_j=2$ for any $k\geq\hat{K}$, $i\in\hat{\mathcal{I}}_X$ and $j\in\hat{\mathcal{I}}_Y$. It then follows from Theorem \ref{alg1-conv-thm} (i) that for any $k\geq\hat{K}$, we have that ${\bar{X}}^k\in\hat\Omega^1\subseteq\Omega^1$, ${\bar{Y}}^k\in\hat\Omega^2\subseteq\Omega^2,$
and
\begin{eqnarray}\label{simp_sub1}
&&\hspace{-7mm} X^{k+1}={{\rm{arg}}\min}_{X\in{\hat\Omega}^1}\hat{S}_{1,k}(X),
\\[5pt]
\label{simp_sub2}
&&\hspace{-7mm} Y^{k+1}={{\rm{arg}}\min}_{Y\in{\hat\Omega}^2}\hat{S}_{2,k}(Y),\quad
\end{eqnarray}
with $\hat{S}_{1,k}(X):=f(X^k{Y^k}^{\mathbb{T}})+\langle\nabla_Zf(X^k{Y^k}^{\mathbb{T}}){Y^k},X-X^k\rangle+\frac{\bar\iota_{1,k}}{2}\|X-X^k\|_F^2$ and $\hat{S}_{2,k}(Y):=f(X^{k+1}{Y^k}^{\mathbb{T}})+\langle\nabla_Zf(X^{k+1}{Y^k}^{\mathbb{T}})^{\mathbb{T}}X^{k+1},Y-Y^k\rangle+\frac{\bar\iota_{2,k}}{2}\|Y-Y^k\|_F^2$.
These mean that after finite iterations, Alg.\,\ref{FGL0alg1} is actually to solve the problem of $\min_{(X,Y)\in\Omega^1_{\hat{\mathcal{I}}_X}\times\Omega^2_{\hat{\mathcal{I}}_Y}}f(XY^{\mathbb{T}})$ with dimension $|\hat{\mathcal{I}}_X|\times|\hat{\mathcal{I}}_Y|$.
\end{remark}
Finally, we show the global convergence results of Alg.\,\ref{FGL0alg1}.
\begin{corollary}\label{conv_rate_ag1}
Let $\{(X^k,Y^k)\}$ be the sequence generated by Alg.\,\ref{FGL0alg1}. If $\,l(X,Y):=f(XY^{\mathbb{T}})+\delta_{\hat\Omega^1}(X)+\delta_{\hat\Omega^2}(Y)$ is a KL function with $\hat\Omega^1$ and $\hat\Omega^2$ defined as in (\ref{hat_omega}), then 
it holds that $\sum_{k=1}^\infty\|(X^{k+1},Y^{k+1})-(X^k,Y^k)\|_F<\infty$ and $\{(X^k,Y^k)\}$ converges to a matrix couple $(X^*,Y^*)$. If $\,l(X,Y)$ has the KL property at $(X^*,Y^*)$ with exponent $\alpha\in[0,1)$, then the following convergence results hold.
\begin{itemize}
\item[(a)] When $\alpha=0$, one has that $\{(X^k,Y^k)\}$ converges finitely.
\item[(b)] When $\alpha\in(0,\frac{1}{2}]$, one has that $\exists~c>0$, $q\in[0,1)$, \emph{s.t.} $\|(X^k,Y^k)-(X^*,Y^*)\|_F\leq cq^k,\forall k\geq0$ (R-linear convegence).
\item[(c)] When $\alpha\in(\frac{1}{2},1)$, one has that $\exists~c>0$, \emph{ s.t.} $\|(X^k,Y^k)-(X^*,Y^*)\|_F\leq ck^{-\frac{1-\alpha}{2\alpha-1}},\forall k\geq0$ (R-sublinear convergence).
\end{itemize}
In particular, if $f$ is a semialgebraic function, then the convergence rate of $\{(X^k,Y^k)\}$ is at least R-sublinear.
\end{corollary}
\begin{proof}
Let $(X^*,Y^*)$ be an accumulation point of $\{(X^k,Y^k)\}$. Since $\mathcal{I}_{X^k}=\hat{\mathcal{I}}_X,\mathcal{I}_{Y^k}=\hat{\mathcal{I}}_Y$ and $\|X^k_i\|\geq\nu,\forall i\in\hat{\mathcal{I}}_X,\|Y^k_i\|\geq\nu,\forall i\in\hat{\mathcal{I}}_Y$ for any $k\geq\hat{K}$ with $\hat{K}$ defined in Theorem \ref{alg1-conv-thm} (i), we get $\Theta_\nu(X^k)=|\hat{\mathcal{I}}_X|,\Theta_\nu(Y^k)=|\hat{\mathcal{I}}_Y|$ and $X^k\in\hat\Omega^1,Y^k\in\hat\Omega^2$ for any $k\geq\hat{K}$. Further, by the existence of $\lim_{k\rightarrow\infty}F(X^k,Y^k)$ in Proposition \ref{limprop}, we have
\begin{equation}\label{convg1}
\lim_{k\rightarrow\infty}f(X^k{Y^k}^{\mathbb{T}})+\delta_{\hat\Omega^1}(X^k)+\delta_{\hat\Omega^2}(Y^k)=f(X^*{Y^*}^{\mathbb{T}})+\delta_{\hat\Omega^1}(X^*)+\delta_{\hat\Omega^2}(Y^*).
\end{equation}
Moreover, based on (\ref{limxk}), $X^k=\bar{X}^k,Y^k=\bar{Y}^k,\forall k\geq\hat{K}$ in Theorem \ref{alg1-conv-thm} (i) and $\Theta_\nu(X^k)=|\hat{\mathcal{I}}_X|,\Theta_\nu(Y^k)=|\hat{\mathcal{I}}_Y|,\forall k\geq\hat{K}$, it holds that
\begin{equation}\label{convg2}
\begin{split}
&f(X^k{Y^k}^{\mathbb{T}})+\delta_{\hat\Omega^1}(X^k)+\delta_{\hat\Omega^2}(Y^k)
\\
\geq\, &f(X^{k+1}(Y^{k+1})^{\mathbb{T}})+\delta_{\hat\Omega^1}(X^{k+1})+\delta_{\hat\Omega^2}(Y^{k+1})\\&+\frac{c_{\min}}{2}\big(\|X^{k+1}-X^k\|_F^2+\|Y^{k+1}-Y^k\|_F^2\big),\,\,\forall k\geq\hat{K}.
\end{split}
\end{equation}
Based on the optimality conditions of (\ref{simp_sub1}) and (\ref{simp_sub2}), we deduce that for any $k\geq\hat{K}$,
\begin{equation*}
\begin{split}
\eta^k_X:=&-\bar\iota_{1,k}(X^{k+1}-X^k)+\big(\nabla_Xf(X^{k+1}(Y^{k+1})^{\mathbb{T}})-\nabla_Xf(X^{k+1}{Y^{k}}^{\mathbb{T}})\big)\\
&+\big(\nabla_Xf(X^{k+1}{Y^{k}}^{\mathbb{T}})-\nabla_Xf(X^{k}{Y^{k}}^{\mathbb{T}})\big)\\
\in\,&\nabla_Xf(X^{k+1}(Y^{k+1})^{\mathbb{T}})+N_{\hat\Omega^1}(X^{k+1}),\\
\eta^k_Y:=&-\bar\iota_{2,k}(Y^{k+1}-Y^k)+\nabla_Yf(X^{k+1}(Y^{k+1})^{\mathbb{T}})-\nabla_Yf(X^{k+1}{Y^{k}}^{\mathbb{T}})\\
\in\,&\nabla_Yf(X^{k+1}(Y^{k+1})^{\mathbb{T}})+N_{\hat\Omega^2}(Y^{k+1}),
\end{split}
\end{equation*}
and hence $(\eta^k_X,\eta^k_Y)\in\partial\big(f(X^{k+1}(Y^{k+1})^{\mathbb{T}})+\delta_{\hat\Omega^1}(X^{k+1})+\delta_{\hat\Omega^2}(Y^{k+1})\big)$. Then, according to $\|Y^k\|_F\leq {\varsigma}\sqrt{d}$, $\|X^{k+1}\|_F\leq \varsigma \sqrt{d},\forall k$, the Lipschitz constant $L_f$ of $\nabla f$ and Proposition \ref{limprop} (i), there exists a constant $\bar{\beta}$ such that
\begin{equation}\label{convg3}
\|(\eta^k_X,\eta^k_Y)\|_F\leq\bar{\beta}\|(X^{k+1},Y^{k+1})-(X^k,Y^k)\|_F,\forall k\geq\hat{K}.
\end{equation}
Consider that $l(X,Y)$ is a KL function. Then, combining with (\ref{convg1}), (\ref{convg2}) and (\ref{convg3}), by \cite[Theorem 2.9]{Attouch2013}, we obtain $\sum_{k=1}^\infty\|(X^{k+1},Y^{k+1})-(X^k,Y^k)\|_F<\infty$ and $\lim_{k\rightarrow\infty}(X^k,Y^k)=(X^*,Y^*)$. Similar to the analysis of Theorem 2 in \cite{Attouch2009}, by the KL property of $l(X,Y)$ at $(X^*,Y^*)$ with exponent $\alpha\in[0,1)$, we obtain the convergence of $(X^k,Y^k)$ to $(X^*,Y^*)$ finitely, R-linearly or R-sublinearly when the {\L}ojasiewicz exponent $\alpha=0$, $\alpha\in(0,\frac{1}{2}]$ or $\alpha\in(\frac{1}{2},1)$, respectively. 
	
Finally, suppose $f$ is semialgebraic. Since $\hat\Omega^1$ and $\hat\Omega^2$ are semialgebraic sets, according to \cite[Section 4.3]{Attouch2010MOR}, $f(XY^{\mathbb{T}})+\delta_{\hat\Omega^1}(X)+\delta_{\hat\Omega^2}(Y)$ is semialgebraic and hence it is a KL function with a suitable exponent $\alpha\in[0,1)$. Therefore, the convergence rate of $\{(X^k,Y^k)\}$ is at least R-sublinear.
\end{proof}
\begin{remark}
For the matrix completion problem, the popular squared loss function $f(Z)=\frac{1}{2}\Vert \mathcal{P}_{\varGamma}(Z-\bar{Z}^\diamond)\Vert_F^2$ is a polynomial function and hence semialgebraic. Moreover, the Huber loss function $h$ popular in robust regression is a semialgebraic function and the corresponding loss function $f(Z)=\sum_{(i,j)\in\varGamma}h\big((Z-\bar{Z}^\diamond)_{ij}\big)$ is also semialgebraic. Then, based on Corollary \ref{conv_rate_ag1}, when $f$ is the squared loss function or Huber loss function, the sequence generated by Alg.\,\ref{FGL0alg1} is convergent to a strong stationary point of (\ref{FGL0C}) and its convergence rate is at least R-sublinear.
\end{remark}
\begin{remark}
By Theorem \ref{alg1-conv-thm}, we can see that the iterate $(X^k,Y^k)$ generated by Alg.\,\ref{FGL0alg1} with $|\mathcal{K}|=+\infty$ owns the column consistency property and column bound-$\nu$ property after finite iterations, which are properties of the global minimizers of (\ref{FGL0C}). Compared with the methods in \cite{Pan2022factor}, the iterates generated only by Alg.\,\ref{FGL0alg1} are stable with respect to the locations of the nonzero columns and can be convergent to a stationary point with stronger optimality conditions than the limiting-critical point of the considered problem.
\end{remark}

\subsection{Improvements to PALM for problem (\ref{FGL0C})}\label{th_rs_ag2}

In this subsection, we combine the PALM in \cite{Bolte2014Proximal} with line search and sparsity consistency processing procedures to develop an algorithm (Alg.\,\ref{FGL0alg2}) for solving problem (\ref{FGL0C}).
\begin{algorithm}[h!]
\caption{({\bf{Alg.\,\ref{FGL0alg2}}})
PALM with line search for problem (\ref{FGL0C})}\label{FGL0alg2}
{\bf Initialization:} Choose $(X^0,Y^0)\in\Omega^1\times\Omega^2$, $0<c_{\min}<c_{\max}$, $0<\underline{\iota}<\overline{\iota}$, $\rho>1$ and $\mathcal{K}\subseteq\mathbb{N}$. Set $k=0$.
	
{\bf Step 1:} Choose $\iota_{1,k}^B\in[\underline{\iota},\overline{\iota}]$ and $c_{1,k}\in[c_{\min},c_{\max}]$. \\
For $l=1,2,\ldots$, do\\
(1a) let $\iota_{1,k}=\iota_{1,k}^B\rho^{l-1}$;\\
(1b) compute ${\bar{X}}^{k+1}(\iota_{1,k})$ by the following problem,
\begin{equation}\label{FGL0sub12}
\begin{split}
\bar{X}^{k+1}(\iota_{1,k})\in{{\rm{arg}}\min}_{X\in\Omega^1}S^0_{1,k}(X;\iota_{1,k}):=\,&f(X^k{Y^k}^{\mathbb{T}})\\&+\langle\nabla_Zf(X^k{Y^k}^{\mathbb{T}}){Y^k},X-X^k\rangle\\&+\frac{\iota_{1,k}}{2}\|X-X^k\|_F^2+\lambda\,\text{nnzc}(X);
\end{split}
\end{equation}
(1c) if ${\bar{X}}^{k+1}(\iota_{1,k})$ satisfies
\begin{equation}\label{objXineq2}
\begin{split}
f({\bar{X}}^{k+1}(\iota_{1,k}){Y^k}^{\mathbb{T}})+\lambda\,\text{nnzc}({\bar{X}}^{k+1}(\iota_{1,k}))\leq\, &f(X^k{Y^k}^{\mathbb{T}})+\lambda\,\text{nnzc}(X^k)\\&-c_{1,k}\|{\bar{X}}^{k+1}(\iota_{1,k})-X^k\|_F^2,
\end{split}
\end{equation}
set ${\bar{X}}^{k+1}={\bar{X}}^{k+1}(\iota_{1,k})$, $\bar\iota_{1,k}=\iota_{1,k}$ and go to Step 2.\\
{\bf Step 2:} Choose $\iota_{2,k}^B\in[\underline{\iota},\overline{\iota}]$ and $c_{2,k}\in[c_{\min},c_{\max}]$. \\
For $l=1,2,\ldots$, do\\
(2a) let $\iota_{2,k}=\iota_{2,k}^B\rho^{l-1}$;\\
(2b) compute ${\bar{Y}}^{k+1}(\iota_{2,k})$ by the following problem,
\begin{equation}\label{FGL0sub22}
\begin{split}
\bar{Y}^{k+1}(\iota_{2,k})\in{{\rm{arg}}\min}_{Y\in\Omega^2} S^0_{2,k}(Y;\iota_{2,k}):=\,&f(\bar{X}^{k+1}{Y^k}^{\mathbb{T}})\\&+\langle\nabla_Zf(\bar{X}^{k+1}{Y^k}^{\mathbb{T}})^{\mathbb{T}}\bar{X}^{k+1},Y-Y^k\rangle\\&+\frac{\iota_{2,k}}{2}\|Y-Y^k\|_F^2+\lambda\,\text{nnzc}(Y);
\end{split}
\end{equation}
(2c) if ${\bar{Y}}^{k+1}(\iota_{2,k})$ satisfies
\begin{equation}\label{objYineq2}
\begin{split}
f({\bar{X}}^{k+1}({\bar{Y}}^{k+1}(\iota_{2,k}))^{\mathbb{T}})+\lambda\,\text{nnzc}({\bar{Y}}^{k+1}(\iota_{2,k}))\leq\, &f({\bar{X}}^{k+1}{Y^k}^{\mathbb{T}})+\lambda\,\text{nnzc}(Y^k)\\&-c_{2,k}\|{\bar{Y}}^{k+1}(\iota_{2,k})-Y^k\|_F^2,
\end{split}
\end{equation}
set ${\bar{Y}}^{k+1}={\bar{Y}}^{k+1}(\iota_{2,k})$, $\bar\iota_{2,k}=\iota_{2,k}$ and go to Step 3.\\
{\bf Step 3:} Let $(X^{k+1},Y^{k+1})=S(\bar{X}^{k+1},\bar{Y}^{k+1})$ if $k\in\mathcal{K}$, and $(X^{k+1},Y^{k+1})=(\bar{X}^{k+1},\bar{Y}^{k+1})$ otherwise. Update $k\leftarrow k+1$ and return to Step 1.
\end{algorithm}

\begin{remark}\label{proj_ox}
The subproblem (\ref{FGL0sub12}) in Alg.\,\ref{FGL0alg2} is equivalent to the following form
\begin{equation}\label{iter1_cal2}
{\bar{X}}^{k+1}(\iota_{1,k})_i={\rm{prox}}_{\frac{\lambda}{\iota_{1,k}}{\emph{\text{nnzc}}}+\delta_{\Omega^1_i}}(Q^{k+1}_i)={\rm{proj}}_{B_{\Mbound}}{\circ}~{\rm{prox}}_{\frac{\lambda}{\iota_{1,k}}{\emph{\text{nnzc}}}}(Q^{k+1}_i),\forall i\in[d]
\end{equation}
with $Q^{k+1}=X^k-\nabla_Zf(X^k{Y^k}^{\mathbb{T}})Y^k/\iota_{1,k}$, and (\ref{FGL0sub22}) in Alg.\,\ref{FGL0alg2} is equivalent to the following form
\begin{equation}\label{iter2_cal2}
{\bar{Y}}^{k+1}(\iota_{2,k})_i={\rm{prox}}_{\frac{\lambda}{\iota_{2,k}}{\emph{\text{nnzc}}}+\delta_{\Omega^2_i}}(\bar{Q}^{k+1}_i)={\rm{proj}}_{B_{\Mbound}}{\circ}~{\rm{prox}}_{\frac{\lambda}{\iota_{2,k}}{\emph{\text{nnzc}}}}(\bar{Q}^{k+1}_i),\forall i\in[d]
\end{equation}
with $\bar{Q}^{k+1}=Y^k-\nabla_Zf(\bar{X}^{k+1}{Y^k}^{\mathbb{T}})^{\mathbb{T}}\bar{X}^{k+1}/\iota_{2,k}$.
Since for any $\alpha>0$ and $\eta\in\mathbb{R}^l$, it holds
\begin{equation}\label{nnzx_prox}
{\rm{prox}}_{\alpha\,{\emph{\text{nnzc}}}}(\eta)=\left\{
\begin{split}
&\{\bm{0}\},&&\text{if }\|\eta\|<\sqrt{2\alpha}\\
&\big\{\bm{0},\eta\},&&\text{if }\|\eta\|=\sqrt{2\alpha}\\
&\{\eta\},&&\text{if }\|\eta\|>\sqrt{2\alpha},
\end{split}
\right.
\end{equation}
we get $\|\bar{X}^{k+1}_i\|\geq\min\{\Mbound,\sqrt{2\lambda/\bar{\iota}_{1,k}}\},\forall i\in\mathcal{I}_{\bar{X}^{k+1}}$ and $\|\bar{Y}^{k+1}_j\|\geq\min\{\Mbound,\sqrt{2\lambda/\bar{\iota}_{2,k}}\},\forall j\in\mathcal{I}_{\bar{Y}^{k+1}}$.
\end{remark}

Then, based on \cite[Lemma 2]{Bolte2014Proximal} and similar to the proof idea of Proposition \ref{limprop}, we can easily obtain the following proposition for which we omit its proof here.
\begin{proposition}\label{limprop2}
Let $\{(X^k,Y^k)\}$ be the sequence generated by Alg.\,\ref{FGL0alg2}. Define $\iota_{\max}:=\max\{\underline{\iota},\rho(2c_{\max}+L_fd\varsigma^2)\}$. Then,
\begin{itemize}
\item[(i)] $\bar{\iota}_{1,k}\leq\max\{\underline\iota,\rho (2c_{1,k}+L_f\|Y^k\|^2)\}\leq\iota_{\max}$ and $\bar{\iota}_{2,k}\leq\max\{\underline\iota,\rho (2c_{2,k}+L_f\|\bar{X}^{k+1}\|^2)\}\leq\iota_{\max}$;
\item[(ii)] $\lim_{k\rightarrow\infty}F_0(X^k,Y^k)$ exists and $\lim_{k\rightarrow\infty}\|\bar{X}^{k+1}-X^k\|_F=\lim_{k\rightarrow\infty}\|\bar{Y}^{k+1}-Y^k\|_F=0$.
\end{itemize}
\end{proposition}

\begin{assumption}\label{mu-nu-new}
$\nu$ in (\ref{FGL0R}) satisfies $\nu\in\Big(0,\min\{\lambda/\kappa,\Mbound,\sqrt{2\lambda/\iota_{\max}}\}\Big)$.
\end{assumption}
Under Assumption \ref{mu-nu-new} on $\nu$, we show the convergence results for Alg.\,\ref{FGL0alg2} as follows.
\begin{theorem}\label{alg2-conv-thm}
Let $\nu$ satisfy Assumption \ref{mu-nu-new}. Then, the following results hold for the sequence $\{(X^k,Y^k)\}$ generated by Alg.\,\ref{FGL0alg2}.
\begin{itemize}
\item[(i)] $\|\bar{X}^{k+1}_i\|\geq\min\{\Mbound,\sqrt{2\lambda/\iota_{\max}}\},\forall i\in\mathcal{I}_{\bar{X}^{k+1}}$ and $\|\bar{Y}^{k+1}_j\|\geq\min\{\Mbound,\sqrt{2\lambda/\iota_{\max}}\},\forall j\in\mathcal{I}_{\bar{Y}^{k+1}}$.
\item[(ii)] The statements (i), (ii) and (iii) in Theorem \ref{alg1-conv-thm} hold.
\item[(iii)] For the global convergence of $\{(X^k,Y^k)\}$, the statements in Corollary \ref{conv_rate_ag1} hold for Alg.\,\ref{FGL0alg2}.
\end{itemize}
\end{theorem}
\begin{proof}
(i) By Remark \ref{proj_ox} and $\bar{\iota}_{1,k}\leq\iota_{\max}$ and $\bar{\iota}_{2,k}\leq\iota_{\max}$ in Proposition \ref{limprop2} (i), we obtain the statement (i).
	
(ii) Based on (i) and Assumption \ref{mu-nu-new}, we deduce that $(X^k,Y^k)$ owns the column bound-$\nu$ property.
It follows from (i) that $\|\bar{X}^{k+1}-X^k\|_F\geq\min\{\Mbound,\sqrt{2\lambda/\iota_{\max}}\}$ if $\mathcal{I}_{\bar{X}^{k+1}}\neq\mathcal{I}_{X^k}$, and $\|\bar{Y}^{k+1}-Y^k\|_F\geq\min\{\Mbound,\sqrt{2\lambda/\iota_{\max}}\}$ if $\mathcal{I}_{\bar{Y}^{k+1}}\neq\mathcal{I}_{Y^k}$. Combining with $\lim_{k\rightarrow\infty}\|\bar{X}^{k+1}-X^k\|_F=\lim_{k\rightarrow\infty}\|\bar{Y}^{k+1}-Y^k\|_F=0$ in Proposition \ref{limprop2} (ii), we deduce that there exist $\tilde{K}>0$ such that $\mathcal{I}_{\bar{X}^{k+1}}=\mathcal{I}_{X^k}$ and $\mathcal{I}_{\bar{Y}^{k+1}}=\mathcal{I}_{Y^k},\forall k\geq\tilde{K}$. It then follows from $\mathcal{I}_{X^{k+1}}\subseteq\mathcal{I}_{\bar{X}^{k+1}}$ and $\mathcal{I}_{Y^{k+1}}\subseteq\mathcal{I}_{\bar{Y}^{k+1}}$ that $\mathcal{I}_{X^{k+1}}\subseteq\mathcal{I}_{X^{k}}$ and $\mathcal{I}_{Y^{k+1}}\subseteq\mathcal{I}_{Y^{k}},\forall k\geq\tilde{K}$. Consequently, there exist $\hat{K}>0$ and $\hat{\mathcal{I}}_X,\hat{\mathcal{I}}_Y\subseteq[d]$ such that $\mathcal{I}_{\bar{X}^k}=\mathcal{I}_{X^k}=\hat{\mathcal{I}}_X$ and $\mathcal{I}_{\bar{Y}^k}=\mathcal{I}_{Y^k}=\hat{\mathcal{I}}_Y,\forall k\geq\hat{K}$. This implies that $(X^k,Y^k)=(\bar{X}^k,\bar{Y}^k),\forall k\geq\hat{K}$. Then, the first statement of (ii) holds.
	
Similar to the proof of Theorem \ref{alg1-conv-thm} (ii), we have the second statement of (ii).
	
Let $(X^*,Y^*)$ be an arbitrary accumulation point of $\{(X^k,Y^k)\}$. By the analysis in \cite[Lemma 5]{Bolte2014Proximal}, we deduce that $(X^*,Y^*)\in\mathcal{S}_{(\ref{FGL0C})}$. It follows from the first two statements of (ii) that the last stated result holds.
	
(iii) Similar to the analysis in Corollary \ref{conv_rate_ag1}, we obtain the statement (iii).
\end{proof}
\begin{remark}\label{PALMcom}
When $\mathcal{K}=\emptyset$, $\iota_{1,k}^B=\max\{\underline{\iota},c_{\min}+L_f\|Y^k\|^2\}$, $\iota_{2,k}^B=\max\{\underline{\iota},c_{\min}+L_f\|\bar{X}^{k+1}\|^2\}$, $c_{1,k}=c_{2,k}=c_{\min}$, based on the theoretical analysis in Proposition \ref{limprop2} (i), there is only one inner iteration in Step 1 and Step 2 of Alg.\,\ref{FGL0alg2}, respectively, and hence Alg.\,\ref{FGL0alg2} becomes the PALM in \cite{Bolte2014Proximal}. Compared with the analysis in \cite{Bolte2014Proximal}, for the iterates generated by the PALM, there are some further progress on its convergence for problem (\ref{FGL0C}), as explained below.
\begin{itemize}
\item As shown in Theorem \ref{alg2-conv-thm} (ii), the column sparsity of the iterates does not change after finite iterations.
\item Any accumulation point of the iterates owns the column bound-$\nu$ property by Theorem \ref{alg2-conv-thm} (ii).
\item The global convergence can be guaranteed by the assumption on the loss function $f$ as shown in Theorem \ref{alg1-conv-thm} (iii), rather than that of the objective function $F_0$ as in \cite{Bolte2014Proximal}. This facilitates the estimation on the convergence rate of the proposed algorithm from the {\L}ojasiewicz exponent of $f$ rather than $F_0$.
\end{itemize}
Moreover, by Theorem \ref{alg2-conv-thm} (ii), the sparsity consistency processing procedure in Step 3 of Alg.\,\ref{FGL0alg2} makes the iterates also own the column consistency property after finite iterations.
\end{remark}
\begin{remark}
Alg.\,\ref{FGL0alg1} and Alg.\,\ref{FGL0alg2} belong to the class of alternating majorization-minimization (MM) methods. However, in view of the special design, such as the column consistency processing and adaptively updated proximal operators, it is not 
possible to directly apply the basic properties of MM method to Alg.\,\ref{FGL0alg1} and Alg.\,\ref{FGL0alg2}. Moreover, due to Theorem \ref{stat-loc} (ii) and Definition \ref{col-bd-p-def}, we have that any global minimizer of (\ref{FGL0C}) owns the column bound-$\alpha$ property for any $\alpha\in(0,\nu]$, where $\nu$ fulfills Assumption \ref{mu-nu}. Notably, the condition on $\nu$ in Assumption \ref{mu-nu-new} is stricter than those in Assumption \ref{mu-nu}. Then, comparing Theorem \ref{alg1-conv-thm} with Theorem \ref{alg2-conv-thm}, it can be seen that the accumulation points of the iterates generated by Alg.\,\ref{FGL0alg1} are more likely to own the stronger optimality conditions than those of Alg.\,\ref{FGL0alg2}.
\end{remark}

\subsection{Dimension reduction (DR) method for Alg.\,\ref{FGL0alg1} and Alg.\,\ref{FGL0alg2}}\label{CR_mtd}

In this subsection, we will give an adaptive dimension reduction method to reduce the computational complexity of Alg.\,\ref{FGL0alg1} and Alg.\,\ref{FGL0alg2} without affecting their theoretical results in subsections \ref{th_rs_ag1} and \ref{th_rs_ag2}.

In view of \cite[Theorem 1]{Yu2013NIPS}, the subproblem (\ref{FGL0sub1}) in Alg.\,\ref{FGL0alg1} is equivalent to (\ref{iter1_cal2}) with $\text{nnzc}$ replaced by $\theta_{\nu_k,I^k_i}$, and (\ref{FGL0sub2}) in Alg.\,\ref{FGL0alg1} is equivalent to (\ref{iter2_cal2}) with $\text{nnzc}$ replaced by $\theta_{\nu_k,J^k_i}$. Combining with the calculation of subproblems in Alg.\,\ref{FGL0alg2} as shown in Remark \ref{proj_ox}, we obtain that for Alg.\,\ref{FGL0alg1} and Alg.\,\ref{FGL0alg2}, ${\bar{X}}^{k+1}(\iota_{1,k})_i=\bm{0}$ if $Q^{k+1}_i=\bm{0}$ and ${\bar{Y}}^{k+1}(\iota_{2,k})_{j}=\bm{0}$ if $\bar{Q}^{k+1}_{j}=\bm{0}$, $\forall i,j\in[d]$ and $k\in\mathbb{N}$.

The low rankness of the matrix in problem (\ref{LRMPR}) corresponds to the column sparsity of the matrices in problem (\ref{FGL0C}). For Alg.\,\ref{FGL0alg1} and Alg.\,\ref{FGL0alg2} to solve problem (\ref{FGL0C}), as the iteration $k$ increases, $\mathcal{I}_{X^k}$ and $\mathcal{I}_{Y^k}$ become smaller. Let $\bar{\mathcal{I}}_k=\mathcal{I}_{X^k}\cup\mathcal{I}_{Y^k}$. It follows from $X^{k}_{\bar{\mathcal{I}}_k^c}=\bm{0}$ and $Y^{k}_{\bar{\mathcal{I}}_k^c}=\bm{0}$ that $Q^{k+1}_{\bar{\mathcal{I}}_k^c}=\bm{0}$ and hence $\bar{X}^{k+1}_{\bar{\mathcal{I}}_k^c}=\bm{0}$. Combining $Y^{k}_{\bar{\mathcal{I}}_k^c}=\bm{0}$ and $\bar{X}^{k+1}_{\bar{\mathcal{I}}_k^c}=\bm{0}$, we also obtain $\bar{Q}^{k+1}_{\bar{\mathcal{I}}_k^c}=\bm{0}$ and hence $\bar{Y}^{k+1}_{\bar{\mathcal{I}}_k^c}=\bm{0}$. These further imply that $X^{k+1}_{\bar{\mathcal{I}}_k^c}=\bm{0}$ and $Y^{k+1}_{\bar{\mathcal{I}}_k^c}=\bm{0}$. Consequently, in the $k$th iteration, $(X^{k+1}_{\bar{\mathcal{I}}_k},Y^{k+1}_{\bar{\mathcal{I}}_k})$ generated by Alg.\,\ref{FGL0alg1} (Alg.\,\ref{FGL0alg2}) corresponds to solving the following problem with reduced number of columns (which is $|\bar{\mathcal{I}}_k|$):
\begin{equation}\label{rd_subq}
\min_{(X^r,Y^r)\in\Omega^1_{\bar{\mathcal{I}}_k}\times\Omega^2_{\bar{\mathcal{I}}_k}}f(X^r(Y^r)^{\mathbb{T}})+\lambda\big(\text{nnzc}(X^r)+\text{nnzc}(Y^r)\big).
\end{equation}
As a result, in each iteration of Alg.\,\ref{FGL0alg1} and Alg.\,\ref{FGL0alg2}, we can adaptively reduce the dimension of the matrices in problem (\ref{FGL0C}) by determining $\mathcal{I}_{X^k}$ and $\mathcal{I}_{Y^k}$. Since this dimension reduction process does not affect the objective function value, it not only maintain the original calculation result, but also reduce the computational cost for the rank regularized problem. In particular, the sparsity consistency processing procedure can further reduce the dimension of problem \eqref{rd_subq} to further reduce the total computational cost of the proposed algorithms.

\section{Numerical Experiments}\label{section6}

In this section, we verify some theoretical results of Alg.\,\ref{FGL0alg1} and Alg.\,\ref{FGL0alg2}, and compare Alg.\,\ref{FGL0alg1}, Alg.\,\ref{FGL0alg2} (with the DR method in subsection \ref{CR_mtd}) with the algorithms in \cite{Bolte2014Proximal,Pan2022factor} for low-rank matrix completion problems under non-uniform sampling settings. Throughout this section, we use Alg.\,\ref{FGL0alg1} and Alg.\,\ref{FGL0alg2} to solve the following problem,
\begin{equation}\label{FGL0_ne}
\min_{(X,Y)\in\Omega^1\times\Omega^2}F_0(X,Y):=f(XY^\mathbb{T})+\lambda\big(\text{nnzc}(X)+\text{nnzc}(Y)\big),
\end{equation}
where $f(Z):=\frac{1}{2}\|\mathcal{P}_{\varGamma}(Z-\bar{Z}^\diamond)\|_F^2$, $\varGamma\subseteq[m]\times [n]$ is the index set of observed entries, $\bar{Z}^\diamond\in\mathbb{R}^{m\times n}$, $\Omega^1=\{X\in\mathbb{R}^{m\times d}:X_i\in B_{\Mbound},i\in[d]\}$ and $\Omega^2=\{Y\in\mathbb{R}^{n\times d}:Y_i\in B_{\Mbound},i\in[d]\}$ with $\varsigma=10^2\|\mathcal{P}_{\varGamma}(\bar{Z}^\diamond)\|_F^{1/2}$ and $d=\min\{100,\lceil 0.5\min\{m,n\}\rceil\}$.

The initial point $(X^0,Y^0)$ in Alg.\,\ref{FGL0alg1} and Alg.\,\ref{FGL0alg2} is set to be the 
same as that in \cite{Pan2022factor}, which is only related to the observation $\mathcal{P}_{\varGamma}(\bar{Z}^\diamond)$. By the singular value decomposition of $\mathcal{P}_{\varGamma}(\bar{Z}^\diamond)$, we obtain $\mathcal{P}_{\varGamma}(\bar{Z}^\diamond)=U\Sigma V^{\mathbb{T}}$ and then let $X^0=[\sqrt{\sigma_1}U_1,\ldots,\sqrt{\sigma_d}U_d]$ and $Y^0=[\sqrt{\sigma_1}V_1,\ldots,\sqrt{\sigma_d}V_d]$, where $\sigma_i$ is the $i$-th diagonal element of $\Sigma$. The stopping criterion for Alg.\,\ref{FGL0alg1} and Alg.\,\ref{FGL0alg2} is set as
\begin{equation*}
\frac{|F_0(X^k,Y^k)-F_0(X^{k-1},Y^{k-1})|}{\max\{1,F_0(X^k,Y^k)\}}\leq \text{Tol}
\end{equation*}
with $\text{Tol}=10^{-7}$ for simulated data and $\text{Tol}=10^{-4}$ for real data unless otherwise specified, and $k=100$ is also a stopping criterion in subsection \ref{comp_simu} and subsection \ref{real_data}.

{\bf{Parameter settings in Alg.\,\ref{FGL0alg1} and Alg.\,\ref{FGL0alg2} :}} According to the theoretical requirements, we fix the parameters in Alg.\,\ref{FGL0alg1} and Alg.\,\ref{FGL0alg2} as follows. Let $\nu=0.99\cdot\min\Big\{\Mbound,\frac{\lambda}{\Mbound(d\varsigma^2+\|\mathcal{P}_{\varGamma}(\bar{Z}^\diamond)\|_F)}\Big\}$. Furthermore, for the same parameters in Alg.\,\ref{FGL0alg1} and Alg.\,\ref{FGL0alg2}, set $\rho=2$, $\mathcal{K}=\{0,1,\ldots,9\}$, $c_{\min}=\underline{\iota}=10^{-5}$, $c_{\max}=\overline{\iota}=3d\varsigma^2$, $c_{1,k}=\max\{c_{\min},\|Y^k\|^2/5\}$ and $c_{2,k}=\max\{c_{\min},\|\bar{X}^{k+1}\|^2/5\}$. Moreover, $\nu_k$, $\iota_{1,k}^B$ and $\iota_{2,k}^B$ in each iteration are set as in Table \ref{para_set}.
\begin{table}[h]
\renewcommand{\arraystretch}{1.5}
\setlength{\belowcaptionskip}{-0.05cm}
\centering
\scriptsize
\caption{$\nu_k$ in Alg.\,\ref{FGL0alg1} and $\iota_{1,k}^B$, $\iota_{2,k}^B$ in both Alg.\,\ref{FGL0alg1} and Alg.\,\ref{FGL0alg2} for problem (\ref{FGL0_ne})}\label{para_set}
\begin{tabular}{|cc|ccc|}
\hline
iteration &$\nu_k$ &iteration &$\iota_{1,k}^B$&$\iota_{2,k}^B$\\
\hline
$k=0$
&$\sqrt{\sigma_1}$
&$k< 10$ 
&$\max\{\underline{\iota},\|Y^k\|^2/2\}$
&$\max\{\underline{\iota},\|\bar{X}^{k+1}\|^2/2\}$
\\
\hline
$k>0$
&$\nu$
&$k\geq10$ 
&$\max\{\underline{\iota},\|Y^k\|^2/4\}$
&$\max\{\underline{\iota},\|\bar{X}^{k+1}\|^2/4\}$
\\
\hline
\end{tabular}
\end{table}

Inspired by the setting of the largest $\lambda$ in \cite[subsection 5.2]{Pan2022factor}, we design the adaptive choices of $\lambda$ in the next paragraph. Before that, we introduce
some notations. For the given $\mathcal{P}_{\varGamma}(\bar{Z}^\diamond)$, set $Q^1$ as in Remark \ref{proj_ox} and define $q:=[\|Q^1_1\|,\ldots,\|Q^1_d\|]^\downarrow$. 
For $j\in[2]$, define $\beta_j : [d]\rightarrow[0,+\infty)$ as $\beta_1(r)=\iota_{1,0}^B\nu_0\sqrt{q_{r+1}}$, $\beta_2(r)=\iota_{1,0}^Bq_{r+1}/2,\forall r\in[d-1]$ and $\beta_j(d)=0.99\cdotp\beta_j(d-1)$. Note that $\beta_j$ will be applied to the setting of $\lambda$ in subsections \ref{lambd-cons} and \ref{comp_palm}, and the following tuning strategy for Alg.\,\ref{FGL0alg1} and Alg.\,\ref{FGL0alg2}, respectively.

{\bf{Adaptive tuning strategy on $\lambda$ for model (\ref{FGL0_ne}):}}
Let $q_{\text{diff}}=[q_1-q_2,\ldots,q_{d-1}-q_d]$, $q_{\text{rela}}=[(q_1-q_2)/\max\{q_2,0.001\},\ldots,(q_{d-1}-q_d)/\max\{q_d,0.001\}]$, 
$r_s=\min(\{i\in[d-1]:(q_{\text{diff}})_i\geq (q_{\text{diff}}^\downarrow)_5\}\cap\{i\in[d-1]:(q_{\text{rela}})_i\geq (q_{\text{rela}}^\downarrow)_2\})$. Set $r_0=r_s$ if $r_s\neq\emptyset$ and $1$ otherwise. Fix $n_\lambda\in[d]$ and $\delta_1,\delta_2\in(0,+\infty)$.
Generate $n_\lambda+1$ different $\lambda$ as follows.

Method 1 (average rank): For $j\in[2]$, $\lambda_{j,0}:=\beta_j(1)$ and $\lambda_{j,i}:=\beta_j(\lceil d\cdotp i/n_\lambda\rceil),\forall i\in[n_\lambda]$.

Method 2 (average distance): For $j\in[2]$, $\lambda_{j,i}:=\beta_j(r_0)-(\beta_j(r_0)-\beta_j(d))\cdotp i/n_\lambda, \forall i\in\{0\}\cup[n_\lambda]$.

Adopt Method 1 if $r_0=1$ and Method 2 otherwise to generate $\lambda_{1,i}$ ($\lambda_{2,i}$) for Alg.\,\ref{FGL0alg1} (Alg.\,\ref{FGL0alg2}). Moreover, denote $(X_{\lambda_{1,i}},Y_{\lambda_{1,i}})$ ($(X_{\lambda_{2,i}},Y_{\lambda_{2,i}})$) as the output of Alg.\,\ref{FGL0alg1} (Alg.\,\ref{FGL0alg2}) for problem (\ref{FGL0_ne}) with $\lambda=\lambda_{1,i}$ ($\lambda_{2,i}$).

{\bf{Output choice}}:
For the predefined functions $\zeta_1(i)$ and $\zeta_2(i)$, we check the following criteria for $i\in[n_\lambda]$ in order to determine the output.\\
Criterion 1: Output $(X_{\lambda_{j,i}},Y_{\lambda_{j,i}})$ if $\zeta_j(i)>\delta_1$.\\
Criterion 2: Output $(X_{\lambda_{j,i^*}},Y_{\lambda_{j,i^*}})$ if $\zeta_j(i)\neq0$ and $\zeta_j(i^*)/\zeta_j(i)>\delta_2$ with $i^*=\max\{t\in[i-1]:\zeta_j(t)\neq0\}$.

Next, for $j\in[2]$, we explain the definition of $\zeta_j:\{0\}\cup[n_\lambda]\rightarrow\mathbb{R}$ used in the experiments. Set $\zeta_j(0)=0$ and for any $i\in[n_\lambda]$, let
\begin{equation*}
\zeta_j(i)=\left\{
\begin{split}
&\frac{|(\text{loss}_j(i)-\text{loss}_j(i-1))|}{|(\text{loss}_j(i)(\text{nz}_j(i)-\text{nz}_j(i-1))|}&&\text{if }\text{loss}_j(i)(\text{nz}_j(i)-\text{nz}_j(i-1))\neq0\\
&0&&\text{otherwise},
\end{split}\right.
\end{equation*}
where $\text{nz}_j$ and $\text{loss}_j:\{0\}\cup[n_\lambda]\rightarrow\mathbb{R}$ are defined by $\text{nz}_j(0)=0$, $\text{nz}_j(1)=\min\{0.1d,\text{nnzc}(X_{\lambda_{j,i}})\}$, $\text{loss}_j(0)=\max\{0.9f(X_{\lambda_{j,1}}Y_{\lambda_{j,1}}^{\mathbb{T}}),f(\bm{0})\}$, $\text{loss}_j(1)=f(X_{\lambda_{j,1}}Y_{\lambda_{j,1}}^{\mathbb{T}})$ and for any $i\in[n_\lambda]\setminus\{1\}$,

\begin{equation*}
{\text{loss}_j(i)=\left\{
\begin{split}
&\text{loss}_j(i-1)&&\text{if }\lambda_{j,i}\geq\beta_j(\text{nz}_j(i-1))\\
&f(X_{\lambda_{j,i}}Y_{\lambda_{j,i}}^{\mathbb{T}})&&\text{otherwise},
\end{split}
\right.}
\end{equation*}
\begin{equation*}
{\text{nz}_j(i)=\left\{
\begin{split}
&\text{nz}_j(i-1)&&\text{if }\lambda_{j,i}\geq\beta_j(\text{nz}_j(i-1))\\
&\text{nnzc}(X_{\lambda_{j,i}})&&\text{otherwise}.
\end{split}
\right.}
\end{equation*}

{\bf{Non-uniform sampling schemes:}} The non-uniform sampling schemes are defined in \cite{Fang2018MP} and also used in \cite{Pan2022factor}. Denote $\pi_{kl}$ the probability that the $(k,l)$-th entry of $\bar{Z}^\diamond$ is sampled. For each $(k,l)\in[m]\times[n]$, set $\pi_{kl}=p_kp_l$ with $p_k$ (and $p_l$) as
\begin{equation*}
\textrm{Scheme 1}\!:\
p_k=\!\left\{\begin{array}{ll}
2p_0& {\rm if}\ k\le\frac{n}{10} \\
4p_0& {\rm if}\ \frac{n}{10}\le k\le \frac{n}{5}\\
p_0& {\rm otherwise}\\
\end{array}\right.\ {\rm or}\ \
\textrm{Scheme 2}\!:\ p_k=\!
\left\{\begin{array}{ll}
3p_0& {\rm if}\ k\le\frac{n}{10} \\
9p_0& {\rm if}\ \frac{n}{10}\le k\le \frac{n}{5}\\
p_0& {\rm otherwise},\\
\end{array}\right.
\end{equation*}
where $p_0>0$ is a normalizing constant such that $\sum_{k=1}^{n}p_k=1$.

{\bf{Algorithms for comparison:}} Recently, an alternating majorization-minimization (AMM) algorithm with extrapolation and a hybrid AMM (HAMM) algorithm are proposed in \cite{Pan2022factor} to solve the following column $\ell_{2,0}$-norm regularized model:
\begin{equation}\label{MS-FL20}
\min_{X\in \mathbb{R}^{n\times d},Y\in\mathbb{R}^{m\times d}}f(XY^\mathbb{T}\!)+\frac{\mu}{2}\big(\|X\|_F^2+\|Y\|_F^2\big)+\lambda\big(\|X\|_{2,0}+\|Y\|_{2,0}\big),
\end{equation}
where $\mu>0$ is a tiny constant, $f$ is smooth and its gradient is globally Lipschitz continuous. In the numerical experiments of \cite{Pan2022factor}, the authors also considered (\ref{MS-FL20}) with $f(Z)=\frac{1}{2}\Vert \mathcal{P}_{\varGamma}(Z-\bar{Z}^\diamond)\Vert_F^2$. As shown by the experiments in \cite{Pan2022factor}, compared with the nuclear norm regularized factorization model in \cite{hastie2015matrix} and the max-norm regularized convex model in \cite{Fang2018MP}, model (\ref{MS-FL20}) has the competitive advantage in offering solutions with lower error, smaller rank and less time for matrix completion problems under the non-uniform sampling schemes by the AMM and HAMM algorithms. Compared with (\ref{MS-FL20}), the considered model (\ref{FGL0C}) is more faithful in representing the rank regularized problem (\ref{LRMPR}). These motivate us to compare the performance of Alg.\,\ref{FGL0alg1}, Alg.\,\ref{FGL0alg2}, with the AMM and HAMM algorithms in \cite{Pan2022factor} for matrix completion problems.

Throughout this section, all computational results are obtained by running MATLAB 2022a on a MacBook Pro (3.2 GHz, 16 GB RAM). The running time (in seconds) of all algorithms is calculated in the same way as in \cite{Pan2022factor}, including the calculation time on the initial point. To measure the accuracy of an output solution $Z^{\rm out}$ from various algorithms, just as in \cite{Toh2010accelerated,Pan2022factor}, we adopt both the relative error (RE) defined by ${\|Z^{\rm out}-Z^\diamond\|_F}/{\|Z^\diamond\|_F}$ and the normalized mean absolute error (NMAE) defined by ${\rm{NMAE}}=\frac{\sum_{(i,j)\in\varDelta\backslash\varGamma}|Z^{\rm out}_{i,j}-Z^\diamond_{i,j}|}{|\varDelta\backslash\varGamma|(r_{\max}-r_{\min})}$,
where $Z^\diamond$ is the target matrix, $\varDelta$ is the index set of given entries, $\varGamma$ is the index set of observed entries, $r_{\max}$ and $r_{\min}$ are lower and upper bounds 
on the given entries. For real datasets, many entries are unknown and hence we adopt the metric of NMAE. We use the abbreviation ``SR'' for sampling ratio and ``LS'' for line search.

\subsection{Matrix completion on simulated data}\label{simu_data}

In this subsection, we test Alg.\,\ref{FGL0alg1} and Alg.\,\ref{FGL0alg2} on simulated data under non-uniform sampling setting. First, we will verify some theoretical results of Alg.\,\ref{FGL0alg1} and Alg.\,\ref{FGL0alg2}, show the acceleration effect of the DR method and the LS procedure in Alg.\,\ref{FGL0alg1} and Alg.\,\ref{FGL0alg2}, and present the comparison results between the PALM algorithm and Alg.\,\ref{FGL0alg2}. Moreover, we will compare the performance of Alg.\,\ref{FGL0alg1}, Alg.\,\ref{FGL0alg2}, AMM and HAMM algorithms in terms of speed and accuracy under the same non-uniform sampling scheme as in \cite{Pan2022factor}. The values reported 
in the tables of this subsection are the mean results for five different instances
generated under the same setting.

Throughout this subsection, the target matrix $Z^\diamond$ is generated by $Z^\diamond=Z^\diamond_{L}(Z^\diamond_{R})^{\mathbb{T}}$, where each entry of $Z^\diamond_{L}\in\mathbb{R}^{n\times d^\star}$ and $Z^\diamond_{R}\in\mathbb{R}^{n\times d^\star}$ is sampled independently from a standard normal distribution $N(0,1)$. Thus, $Z^\diamond\in\mathbb{R}^{m\times n}$ is a rank $d^\star$ matrix with $m=n$. The index set $\varGamma$ is obtained from non-uniform sampling Scheme 1. The noisy observation entries $\bar{Z}^\diamond_{i_t,j_t}$ with $(i_t,j_t)\in\varGamma$ for $t=1,2,\ldots,p$ are generated as in \cite{Pan2022factor} via the following setting
\begin{equation*}
\bar{Z}^\diamond_{i_t,j_t}=Z^\diamond_{i_t,j_t}+\sigma({\xi_{t}}/{\|\xi\|})\|\mathcal{P}_{\varGamma}(Z^\diamond)\|_F,\mbox{ where }\xi_t\sim N(0,1)\mbox{ and }\xi=(\xi_1,\ldots,\xi_p)^{\mathbb{T}}.
\end{equation*}

\subsubsection{Theoretical verification on Alg.\,\ref{FGL0alg1} and Alg.\,\ref{FGL0alg2}}\label{lambd-cons}
In this part, we verify some theoretical results of Alg.\,\ref{FGL0alg1} (Alg.\,\ref{FGL0alg2}) for problem (\ref{FGL0_ne}) with $\lambda=\beta_1(10)$ ($\lambda=\beta_2(10)$) when $n=1000$, $d^\star=10$, SR=$0.1,0.15,0.2,0.25$ and noise level $\sigma=0.1$. Set $\text{Tol}=10^{-15}$ in the stopping criterion to focus on the long-term behaviour of the iterates generated by Alg.\,\ref{FGL0alg1} and Alg.\,\ref{FGL0alg2}.

As shown in Fig.\,\ref{colu_SR}, $(X^k,Y^k)$ generated by Alg.\,\ref{FGL0alg1} and Alg.\,\ref{FGL0alg2} own the column bound property after finite iterations. Moreover, Fig.\,\ref{nnzc_rank_comp} shows that after finite iterations, one has $\text{nnzc}(X^k)=\text{nnzc}(Y^k)=\text{rank}(X^k{Y^k}^{\mathbb{T}})$ and hence $\mathcal{I}(X^k)=\mathcal{I}(Y^k)$ by (\ref{rank-rel}). This means that $(X^k,Y^k)$ generated by Alg.\,\ref{FGL0alg1} and Alg.\,\ref{FGL0alg2} own the column consistency property after finite iterations. Under the same target matrix and different sampling ratios, $\nabla_Zf(X^k{Y^k}^{\mathbb{T}})Y^k_{\mathcal{I}_{X^k}\cap\mathcal{I}_{Y^k}}$ and $\nabla_Zf(X^k{Y^k}^{\mathbb{T}})^{\mathbb{T}}X^k_{\mathcal{I}_{X^k}\cap\mathcal{I}_{Y^k}}$ converge to $\bm{0}$ as shown in Fig.\,\ref{stat_dF}. Combining with Definition \ref{CR-sstr}, these mean that the limit points of $\{(X^k,Y^k)\}$ generated by Alg.\,\ref{FGL0alg1} and Alg.\,\ref{FGL0alg2} are strong stationary points of problem (\ref{FGL0_ne}), which are consistent with the statements in Theorem \ref{alg1-conv-thm}, Corollary \ref{conv_rate_ag1} and Theorem \ref{alg2-conv-thm}. Moreover, by Fig.\,\ref{stat_dF}, under the same sampling ratio, $\nabla_Zf(X^k{Y^k}^{\mathbb{T}})Y^k_{\mathcal{I}_{X^k}\cap\mathcal{I}_{Y^k}}$ and $\nabla_Zf(X^k{Y^k}^{\mathbb{T}})^{\mathbb{T}}X^k_{\mathcal{I}_{X^k}\cap\mathcal{I}_{Y^k}}$ have  similar convergence rates. In addition, as the sampling ratio increases, Alg.\,\ref{FGL0alg1} and Alg.\,\ref{FGL0alg2} need fewer iterations to satisfy the termination condition.
\begin{figure}[h]
\centering
\subfigure[SR=0.1]{\includegraphics[height=1.8in,width=2.5in]{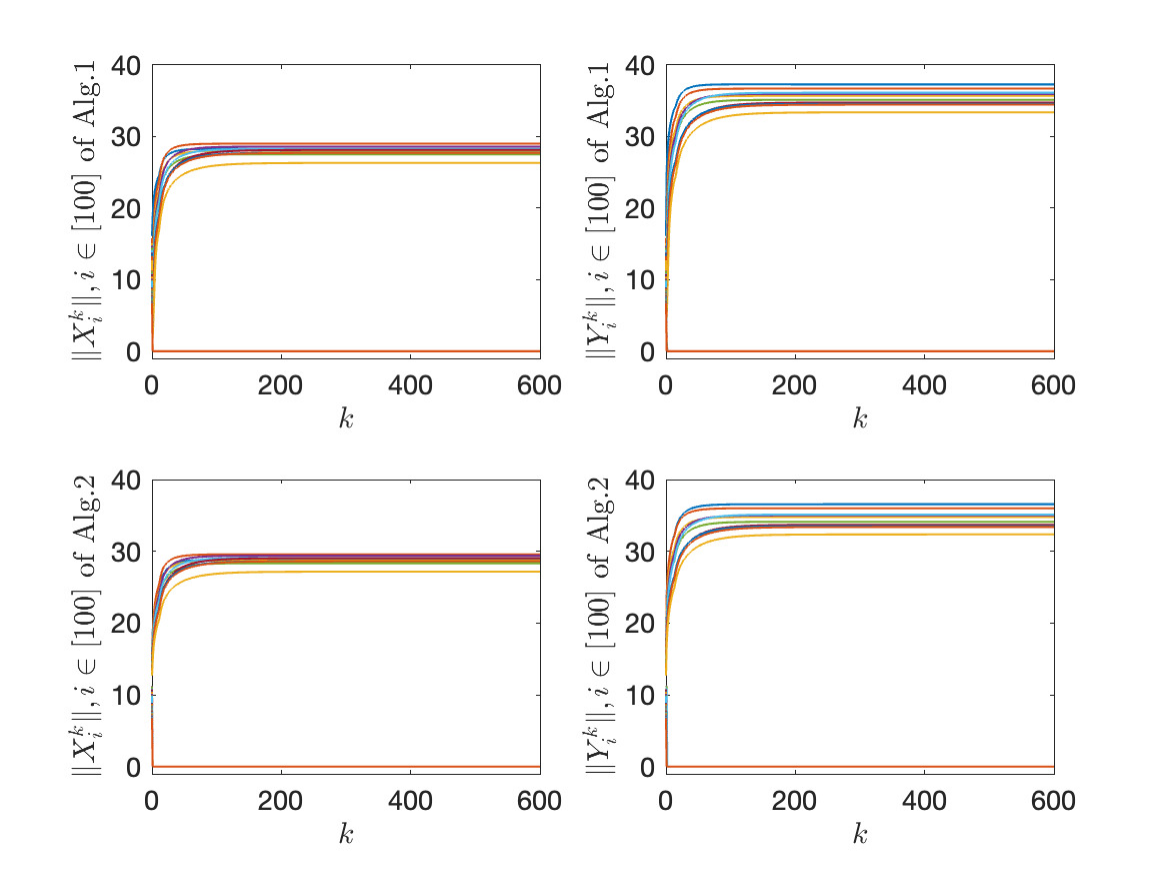}
\label{SR010}}
\subfigure[SR=0.25]{\includegraphics[height=1.8in,width=2.5in]{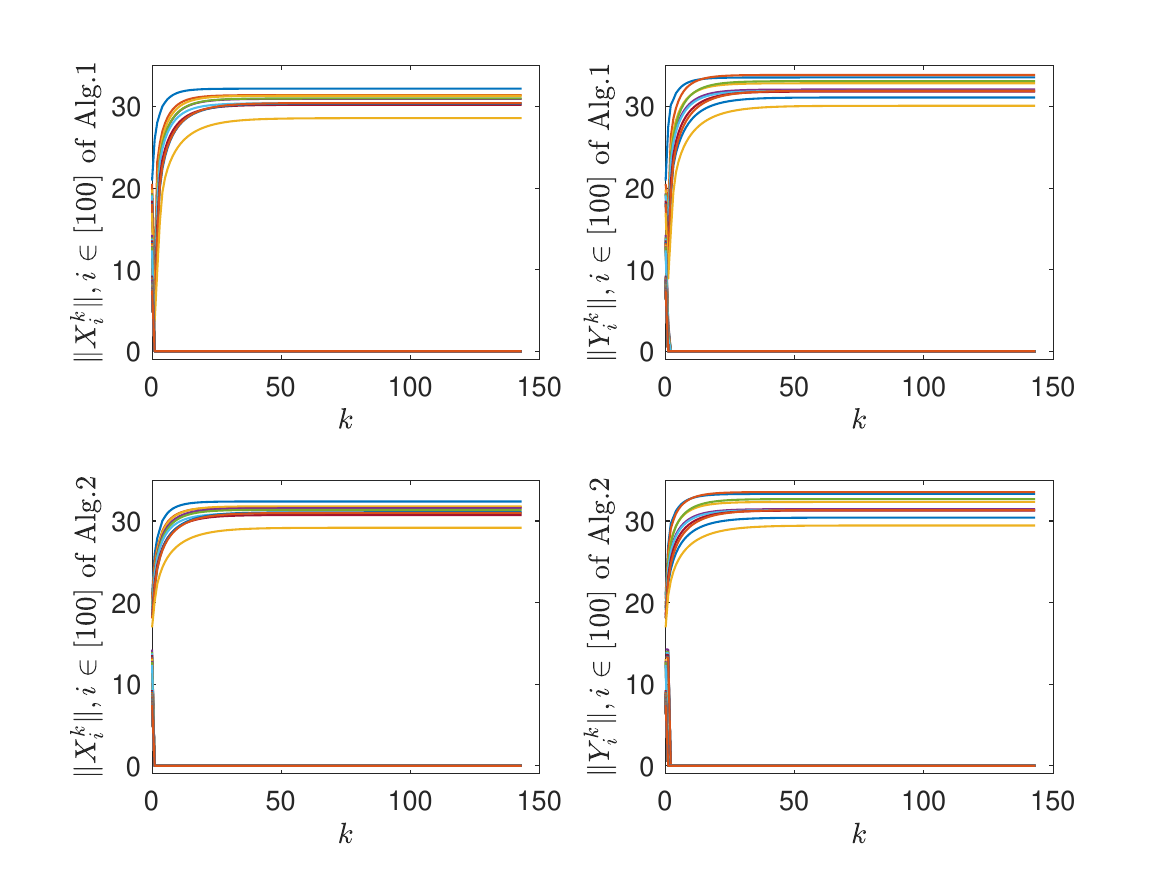}
\label{SR025}}
\caption{$\|X^k_i\|$ and $\|Y^k_i\|,i\in[d]$ generated by Alg.\,\ref{FGL0alg1} and Alg.\,\ref{FGL0alg2} against iteration $k$ when $d^\star=10$.}\label{colu_SR}
\end{figure}
\begin{figure}[h]
\centering
\subfigure[Alg.\,\ref{FGL0alg1}]{\includegraphics[height=1.5in,width=2.5in]{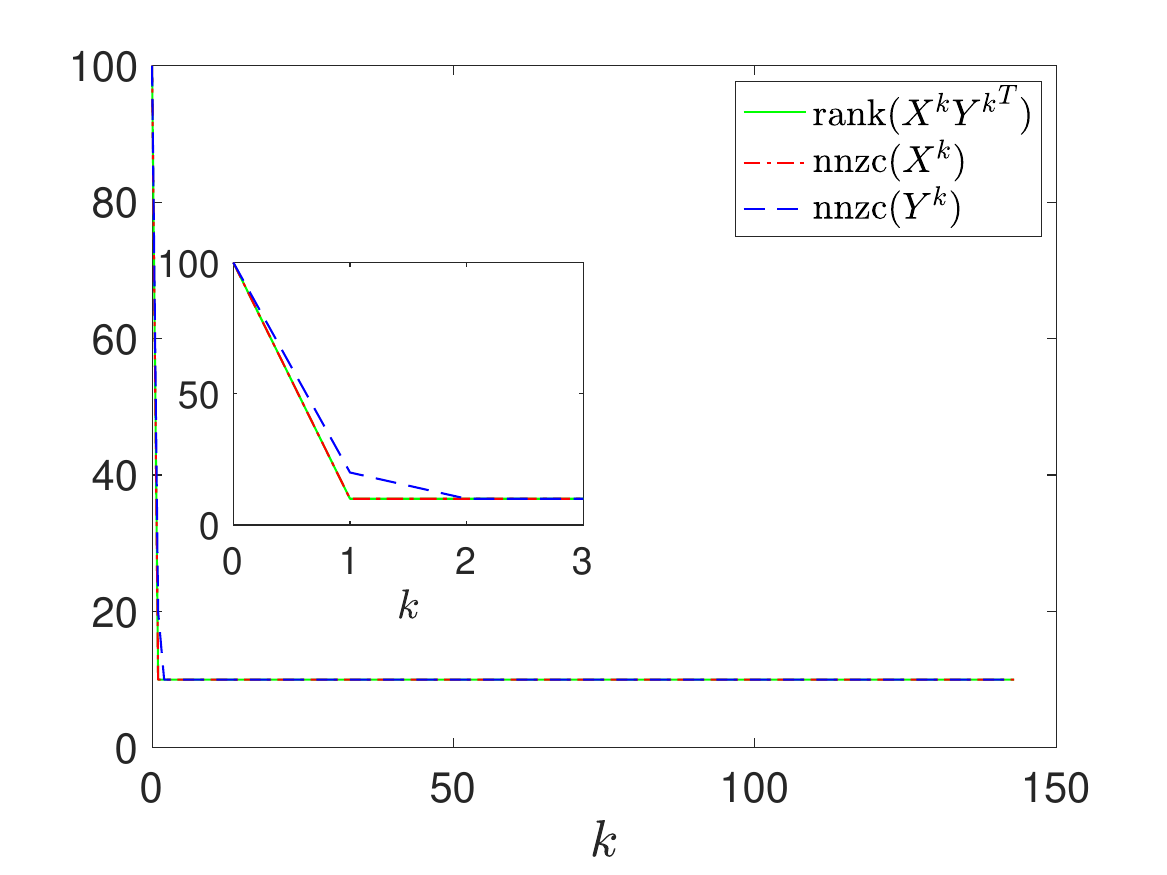}
\label{a1nnzc_rank}}
\hfil
\subfigure[Alg.\,\ref{FGL0alg2}]{\includegraphics[height=1.5in,width=2.5in]{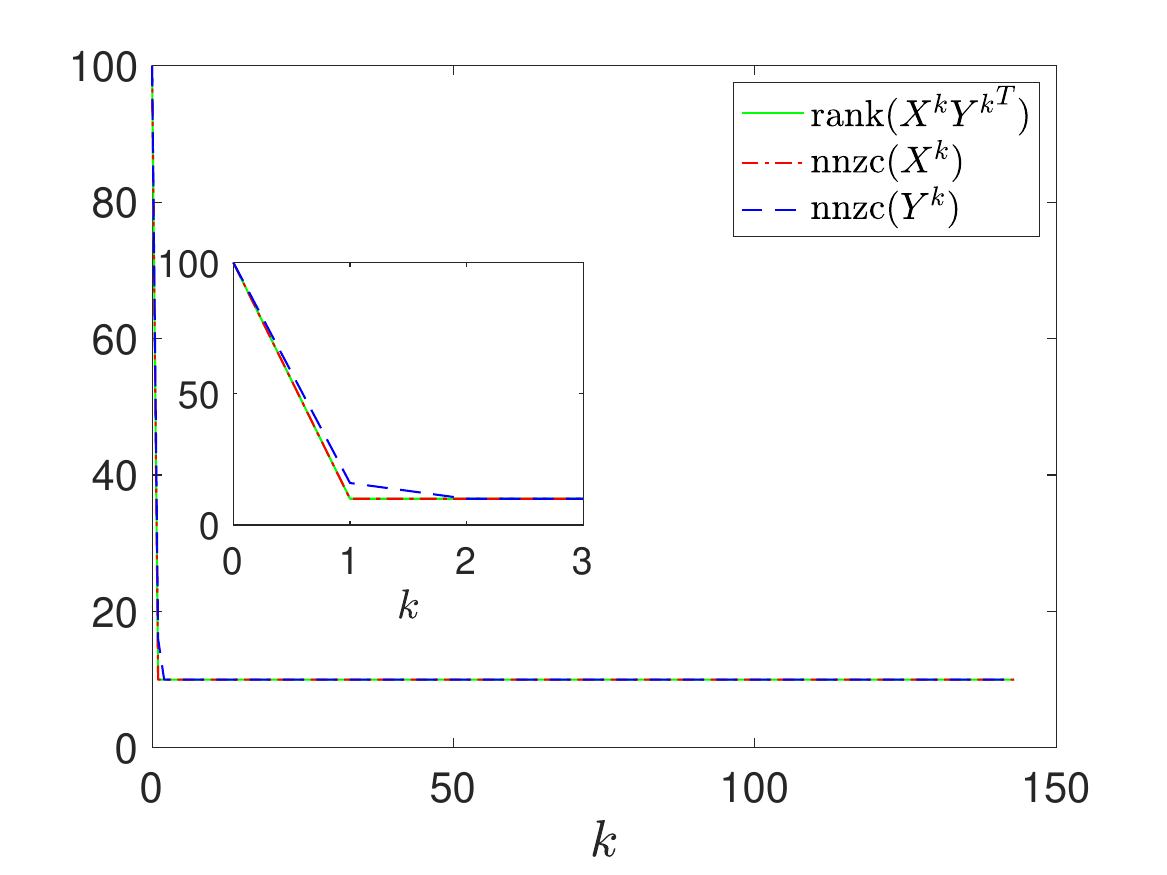}
\label{a2nnzc_rank}}
\caption{$\text{rank}(X^k{Y^k}^T)$, $\text{nnzc}(X^k)$ and $\text{nnzc}(Y^k)$ generated by Alg.\,\ref{FGL0alg1} and Alg.\,\ref{FGL0alg2} against iteration $k$ when $d^\star=10$ and SR=0.25.}\label{nnzc_rank_comp}
\end{figure}
\begin{figure}[h]
\centering
\subfigure[Alg.\,\ref{FGL0alg1}]{\includegraphics[height=1.6in,width=2.5in]{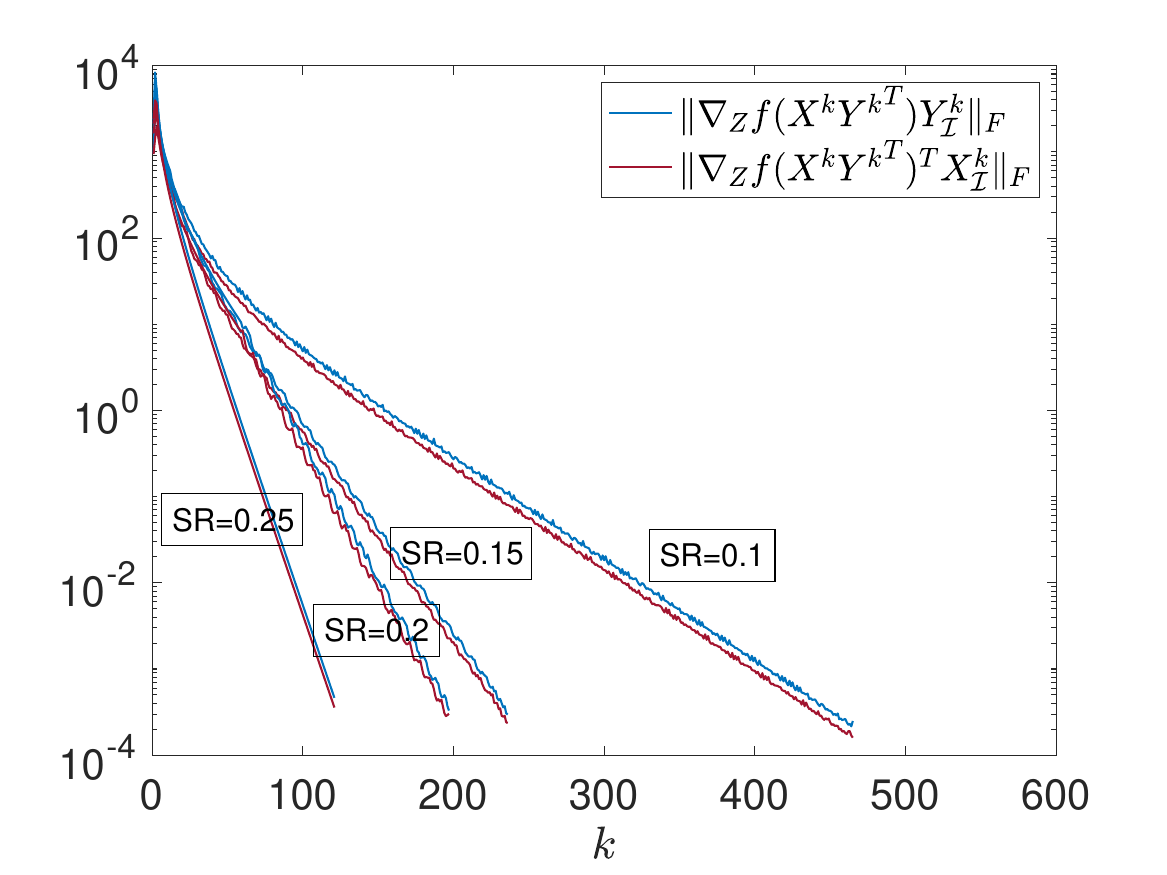}
\label{a1d8}}
\hfil
\subfigure[Alg.\,\ref{FGL0alg2}]{\includegraphics[height=1.6in,width=2.5in]{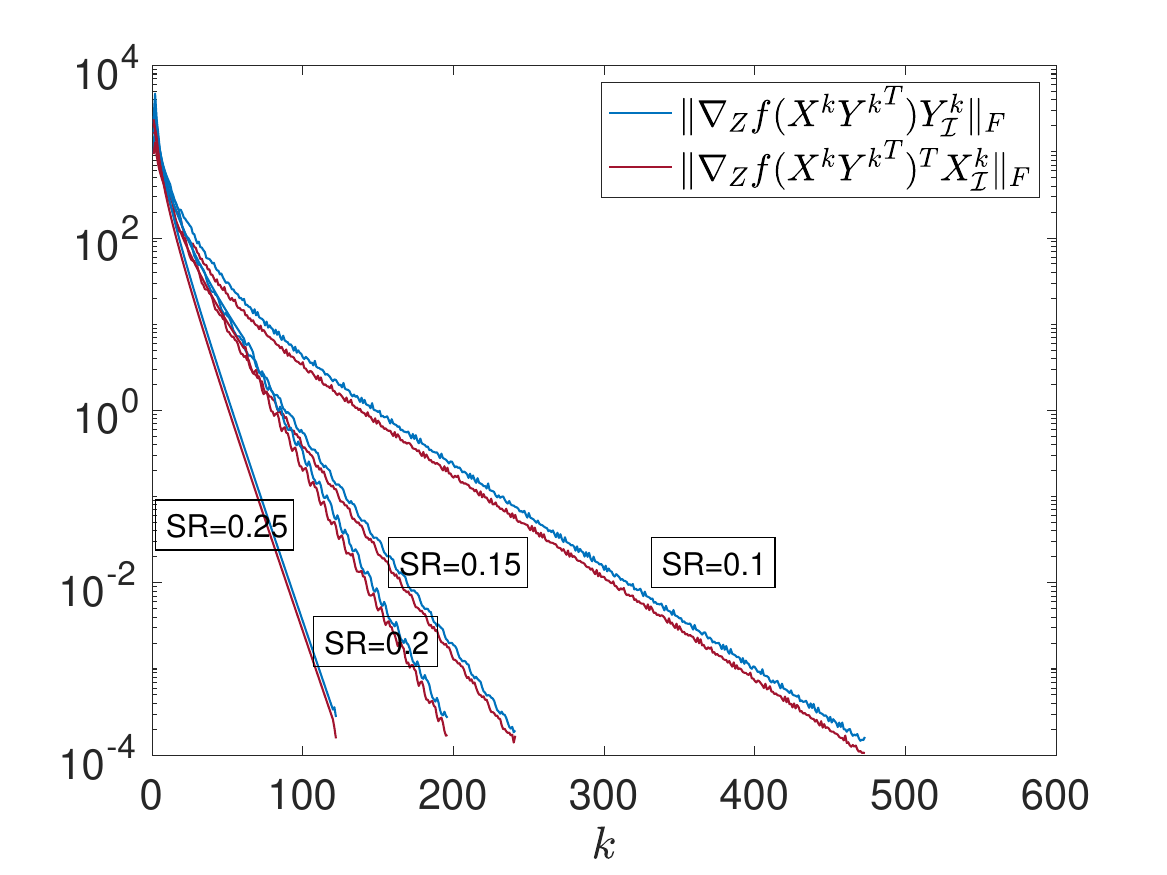}
\label{a2d8}}
\caption{$\|\nabla_Zf(X^k{Y^k}^T)Y^k_{\mathcal{I}}\|_F$ and $\|\nabla_Zf(X^k{Y^k}^T)^TX^k_{\mathcal{I}}\|_F$ with $\mathcal{I}=\mathcal{I}_{X^k}\cap\mathcal{I}_{Y^k}$ generated by Alg.\,\ref{FGL0alg1} and Alg.\,\ref{FGL0alg2} (in logarithmic scale) against iteration $k$ when $d^\star=5$ and SR=0.1, 0.15, 0.2 and 0.25.}\label{stat_dF}
\end{figure}

\subsubsection{The acceleration of DR method and LS procedure on Alg.\,\ref{FGL0alg1} and Alg.\,\ref{FGL0alg2}}\label{comp_palm}

In this part, we compare the performance of Alg.\,\ref{FGL0alg1} (without DR or LS), Alg.\,\ref{FGL0alg2} (without DR or LS) and PALM (with different values of $L_f$) for problem (\ref{FGL0_ne}) when the noise level $\sigma=0.1$.

In Table \ref{Comp_ag1} and Table \ref{Comp_ag2}, ``without LS (A)'' refers to Alg.\,\ref{FGL0alg1} (Alg.\,\ref{FGL0alg2}) with $\iota_{1,k}=\max\{\underline{\iota},\|Y^k\|^2/2\}$, $\iota_{2,k}=\max\{\underline{\iota},\|\bar{X}^{k+1}\|^2/2\}$ and without (1a), (1c), (2a) and (2c), and ``without LS (B)'' refers to Alg.\,\ref{FGL0alg1} with $\iota_{1,k}=\max\{\underline{\iota},\|Y^k\|^2\}$, $\iota_{2,k}=\max\{\underline{\iota},\|\bar{X}^{k+1}\|^2\}$ and without (1a), (1c), (2a) and (2c). In Table \ref{Comp_ag2}, PALM (0.5 or 1) is the PALM algorithm in Remark \ref{PALMcom} with $L_f=0.5$ or $1$. From Table \ref{Comp_ag1} and Table \ref{Comp_ag2}, it can be seen that the proposed DR method in subsection \ref{CR_mtd} and the designed LS procedure are beneficial for accelerating Alg.\,\ref{FGL0alg1} and Alg.\,\ref{FGL0alg2}, and Alg.\,\ref{FGL0alg2} runs faster than PALM for problem (\ref{FGL0_ne}). Through the DR method in subsection 5.3, we can see that when the rank of the restored matrix by Alg.\,\ref{FGL0alg1} is larger, the dimension of the variable in Alg.\,\ref{FGL0alg1} is larger. 
This results in a relatively large amount of calculation and a relatively small step size for each iteration. Therefore, Alg.\,1 without LS (B) is much slower than Alg.\,1 without LS (A) as shown in Table \ref{Comp_ag1}.

\begin{table}[h]
\setlength{\belowcaptionskip}{0.01cm}
\centering
\scriptsize
\caption{\small Average rank, RE and running time of Alg.\,\ref{FGL0alg1} (without DR or LS) for (\ref{FGL0_ne}) when SR=0.1}\label{Comp_ag1}
\begin{tabular}{ccccccc}
\hline
$(n,\,d^\star,\,\lambda)$& \multicolumn{3}{l}{\quad $(1000,\,5,\,\beta_1(5))$}&\multicolumn{3}{l}{\quad $(1000,\,10,\,\beta_1(10))$}\\
\cmidrule(lr){2-4} \cmidrule(lr){5-7}
& rank & RE& time& rank & RE&  time \\
\hline
Alg.\,\ref{FGL0alg1}
&5&0.0472 &{\color{blue} 1.5}
&10&0.0719 &{\color{blue}1.8}\\
\hline
Alg.\,\ref{FGL0alg1} without DR
&5&0.0472 &3.5
&10&0.0719 &4.2\\
\hline
Alg.\,\ref{FGL0alg1} without LS (A)
&5&0.0473 &2.2
&10&0.0720 &2.4\\
\hline
Alg.\,\ref{FGL0alg1} without LS (B)
&64&0.6570 &20.7
&100&0.6794 &14.3\\
\hline
$(n,\,d^\star,\,\lambda)$& \multicolumn{3}{l}{\quad $(2000,\,20,\,\beta_1(20))$}&\multicolumn{3}{l}{\quad $(4000,\,40,\,\beta_1(40))$}\\
\cmidrule(lr){2-4} \cmidrule(lr){5-7}
& rank & RE& time& rank & RE&  time \\
\hline
Alg.\,\ref{FGL0alg1}
&20&0.0707 &{\color{blue} 5.2}
&40&0.0701 &{\color{blue} 15.4}\\
\hline
Alg.\,\ref{FGL0alg1} without DR
&20&0.0707 &8.1
&40&0.0701 &23.0\\
\hline
Alg.\,\ref{FGL0alg1} without LS (A)
&20&0.0708 &7.5
&40&0.0702 &20.3\\
\hline
Alg.\,\ref{FGL0alg1} without LS (B)
&99.6&0.6523 &68.9
&96.8&0.6386 &255.8\\
\hline
\end{tabular}
\end{table}
\begin{table}[h]
\setlength{\belowcaptionskip}{0.01cm}
\centering
\scriptsize
\caption{\small Average rank, RE and running time of Alg.\,\ref{FGL0alg2} (without DR or LS) and PALM for (\ref{FGL0_ne}) when SR=0.1}\label{Comp_ag2}
\begin{tabular}{ccccccc}
\hline
$(n,\,d^\star,\,\lambda)$& \multicolumn{3}{l}{\quad $(1000,\,5,\,\beta_2(5))$}&\multicolumn{3}{l}{\quad $(1000,\,10,\,\beta_2(10))$}\\
\cmidrule(lr){2-4} \cmidrule(lr){5-7} 
& rank & RE& time& rank & RE&  time \\
\hline
Alg.\,\ref{FGL0alg2}
&5&0.0471 &{\color{blue} 1.6}
&10&0.0717 &{\color{blue} 1.9}\\
\hline
Alg.\,\ref{FGL0alg2} without DR
&5&0.0471 &3.8
&10&0.0717 &4.5\\
\hline
Alg.\,\ref{FGL0alg2} without LS (A)
&5&0.0471 &2.1
&10&0.0717 &2.5\\
\hline
PALM (0.5)
&5&0.0471 &5.4
&10&0.0718 &6.5\\
\hline
PALM (1)
&10&0.7554 &50.8
&20&0.6930 &190.6\\
\hline
$(n,\,d^\star,\,\lambda)$& \multicolumn{3}{l}{\quad $(2000,\,20,\,\beta_2(20))$}&\multicolumn{3}{l}{\quad $(4000,\,40,\,\beta_2(40))$}\\
\cmidrule(lr){2-4} \cmidrule(lr){5-7}
& rank & RE& time& rank & RE&  time \\
\hline
Alg.\,\ref{FGL0alg2}
&20&0.0706 &{\color{blue} 5.6}
&40&0.0701 &{\color{blue} 16.0}\\
\hline
Alg.\,\ref{FGL0alg2} without DR
&20&0.0706 &9.8
&40&0.0701 &23.8\\
\hline
Alg.\,\ref{FGL0alg2} without LS (A)
&20&0.0707 &8.0
&40&0.0701 &21.3\\
\hline
PALM (0.5)
&20&0.0707 &14.2
&40&0.0701 &32.2\\
\hline
PALM (1)
&39.8&0.6984 &367.9
&73.6&0.7101 &114.3\\
\hline
\end{tabular}
\end{table}

\subsubsection{Comparison between Alg.\,\ref{FGL0alg1}, Alg.\,\ref{FGL0alg2}, AMM and HAMM}\label{comp_simu}
In this part, we compare the recovery performance of Alg.\,\ref{FGL0alg1}, Alg.\,\ref{FGL0alg2}
versus the AMM and HAMM algorithms in \cite{Pan2022factor} under different noise levels and sampling ratios. We adopt the adaptive tuning strategy for $\lambda$  with $n_\lambda=20$, $\delta_1=1$ and $\delta_2=3$ to run Alg.\,\ref{FGL0alg1}, Alg.\,\ref{FGL0alg2} for problem (\ref{FGL0_ne}). All settings of AMM and HAMM are the same as in \cite[subsection 5.3]{Pan2022factor} for problem (\ref{MS-FL20}). The listed times in Table \ref{Simu_com} are the total running time for solving \eqref{MS-FL20} with the adaptive tuning strategy on $\lambda$. In the table, we highlight the entries that are obviously better than the other compared entries in blue color. As shown in Table \ref{Simu_com}, Alg.\,\ref{FGL0alg1} and Alg.\,\ref{FGL0alg2} have marginally better recovery performance than the AMM and HAMM algorithms. Importantly, the running times of Alg.\,\ref{FGL0alg1} and Alg.\,\ref{FGL0alg2} are all less than AMM and HAMM. In particular, the lower the rank of the target matrix or the higher the noise, the better is the performance of Alg.\,\ref{FGL0alg1}. Moreover, the outperformance of Alg.\,\ref{FGL0alg2} in terms of ``RE'' is almost always the best.

Finally, we give a simple  example with full sampling  to show the difference in performance between the obtained solution from models (\ref{FGL0_ne}) and (\ref{MS-FL20}), and the computable global minimizer of problem (\ref{LRMP}).
We consider problem (\ref{LRMP}) with $f(Z)=\frac{1}{2}\|Z-{Z}^\diamond\|_F^2$ 
and $\lambda=1$, and the corresponding models (\ref{FGL0_ne}) and (\ref{MS-FL20}) with different $\mu$, where $Z^\diamond$ and the unique global minimizer $Z^{\rm opt}$ of the rank regularized problem (\ref{LRMP}) are given below:  
\begin{equation*}
{Z}^\diamond=\begin{bmatrix}-1&0&0&0\\ 0&3&0&-1\\ 0&0&1&0\\ 0&1&0&3\end{bmatrix},
\quad
Z^{\rm opt}=\begin{bmatrix}0&0&0&0\\0&3&0&-1\\0&0&0&0\\0&1&0&3\end{bmatrix}.
\label{eq-Zdiamond}
\end{equation*}
We compare the difference between the  solutions obtained from (\ref{FGL0_ne}) and (\ref{MS-FL20}) and the global minimizer $Z^{\rm opt}$. The error results are shown in Table \ref{glob_comp}, where the error is defined by $\|Z^{\rm out}-Z^{\rm opt}\|_F$ with $Z^{\rm out}$ being the solution obtained by various algorithms. It can be seen that Alg.\,\ref{FGL0alg1} and Alg.\,\ref{FGL0alg2} obtain higher quality solutions, and the  errors of the solutions obtained by AMM and HAMM are dependent on the penalty parameter $\mu$ and they gradually approach $Z^{\rm opt}$ as  $\mu$ decreases toward $0$. The objective function of model (\ref{MS-FL20}) with $\mu=0$ is equivalent to that of model (\ref{FGL0_ne}), which is consistent with the result in Table \ref{glob_comp}.

\begin{table}[h]
\setlength{\belowcaptionskip}{0cm}
\centering
\scriptsize
\caption{\small Average RE, rank and running time of four algorithms for non-uniformly sampled synthetic data}\label{Simu_com}
\begin{tabular}{cccccccc}
\hline
Noise & SR 
& \multicolumn{3}{l}{\qquad\qquad Alg.\,\ref{FGL0alg1}}
&\multicolumn{3}{l}{\qquad\qquad AMM}\\
\cmidrule(lr){3-5} \cmidrule(lr){6-8} 
&  & \,\,\,\,\,RE & rank &time & \,\,\,\,\,RE & rank & time\\
\hline
\multicolumn{8}{l}{\qquad\qquad Case 1: $n=1000$, $d^\star=6$ with $\sigma=0, 0.1, 0.2$.}\\
\hline
$\sigma=0$   
&0.15
& {0.0025} & 6 & \blue{1.1}
& 0.0071 & 6 & 7.6\\
&0.25
& 0.0014 & 6 & \blue{1.0}
& \blue{0.0006} & 6 & 9.8\\
\hdashline
$\sigma=0.1$
&0.15
& \blue{0.0389} & 6 & \blue{1.2}
& 0.0393 & 6 & 7.8\\
&0.25
& \blue{0.0274} & 6 & \blue{1.2}
& 0.0275 & 6 & 9.9\\
\hdashline
$\sigma=0.2$
&0.15 
& \blue{0.0777} & 6 & \blue{2.1}
& \blue{0.0777} & 6 & 7.5\\
&0.25
& \blue{0.0548} & 6 & \blue{2.2}
& 0.0549 & 6 & 9.8\\
			
\hline
\multicolumn{8}{l}{\qquad\qquad Case 2: $n=1000$, $d^\star=13$ with $\sigma=0, 0.1, 0.2$.}\\
\hline
$\sigma=0$
&0.15
& {0.0045} & 13 & \blue{1.4}
& 0.0087 & 13 & 8.6\\
&0.25 
& 0.0021 & 13 & \blue{1.2}
& {0.0015} & 13 & 10.6\\
\hdashline
$\sigma=0.1$
&0.15 
& \blue{0.0604} & 13 & \blue{1.4}
& 0.0609 & 13 & 9.0\\
&0.25
& \blue{0.0415} & 13 & \blue{1.2}
& 0.0416 & 13 & 10.2\\
\hdashline
$\sigma=0.2$
&0.15 
& \blue{0.1209} & 13 & \blue{2.5}
& 0.1213 & 13 & 9.0\\
&0.25
& \blue{0.0830} & 13 & {2.6}
& 0.0832 & 13 & 10.5\\
			
\hline
\multicolumn{8}{l}{\qquad\qquad Case 3: $n=1000$, $d^\star=20$ with $\sigma=0, 0.1, 0.2$.}\\
\hline
$\sigma=0$
&0.15
& {0.0084} & 20 & \blue{1.6}
& 0.0192 & 20 & 10.3\\
&0.25
& 0.0025 & 20 & \blue{1.3}
& 0.0022 & 20 & 11.1\\
\hdashline
$\sigma=0.1$
&0.15
& {0.0807} & 20 & \blue{2.1}
& 0.0840 & 20 & 10.4\\
&0.25 
& \blue{0.0531} & 20 & \blue{1.2}
& 0.0532 & 20 & 10.9\\
\hdashline
$\sigma=0.2$
&0.15 
& \blue{0.1612} & 20 & \blue{2.9}
& 0.1656 & 20 & 10.8\\
&0.25
& \blue{0.1064} & 20 & \blue{2.6}
& 0.1066 & 20 & 11.9\\
\hline
Noise & SR 
&\multicolumn{3}{l}{\qquad\qquad  Alg.\,\ref{FGL0alg2}}
&\multicolumn{3}{l}{\qquad\qquad HAMM}\\
\cmidrule(lr){3-5} \cmidrule(lr){6-8}
&  & \,\,\,\,\,RE & rank &time & \,\,\,\,\,RE & rank & time\\
\hline
\multicolumn{8}{l}{\qquad\qquad Case 1: $n=1000$, $d^\star=6$ with $\sigma=0, 0.1, 0.2$.}\\
\hline
$\sigma=0$   
&0.15
& \blue{0.0016} & 6 & \blue{1.1}
& 0.0069 & 6 & 4.4\\
&0.25
& {0.0009} & 6 & \blue{1.0}
& 0.0027 & 6 & 5.7\\
\hdashline
$\sigma=0.1$   
&0.15
& \blue{0.0389} & 6 & \blue{1.2}
& 0.0398 & 6 & 4.4\\
&0.25
& \blue{0.0274} & 6 & \blue{1.2}
& 0.0276 & 6 & 6.4\\
\hdashline
$\sigma=0.2$   
&0.15 
& \blue{0.0777} & 6 & {2.2}
& 0.0787 & 6 & 4.3\\
&0.25
& \blue{0.0548} & 6 & {2.7}
& 0.0551 & 6 & 5.7\\
			
\hline
\multicolumn{8}{l}{\qquad\qquad Case 2: $n=1000$, $d^\star=13$ with $\sigma=0, 0.1, 0.2$.}\\
\hline
$\sigma=0$
&0.15
& \blue{0.0028} & 13 & \blue{1.4}
& 0.0088 &13 & 4.9\\
&0.25 
& \blue{0.0013} & 13 & \blue{1.2}
& 0.0033 & 13 & 5.9\\
\hdashline
$\sigma=0.1$
&0.15 
& \blue{0.0604} & 13 & {1.5}
& 0.0616 &13 & 5.0\\
&0.25
& \blue{0.0415} & 13 & \blue{1.2}
& 0.0417 & 13 & 5.9\\
\hdashline
$\sigma=0.2$
&0.15 
& 0.1210 & 13 & {2.7}
& 0.1235 &13 & 5.0\\
&0.25
& \blue{0.0830} & 13 & \blue{2.3}
& 0.0835 & 13 & 5.9\\
			
\hline
\multicolumn{8}{l}{\qquad\qquad Case 3: $n=1000$, $d^\star=20$ with $\sigma=0, 0.1, 0.2$.}\\
\hline
$\sigma=0$
&0.15
& \blue{0.0066} & 20 & \blue{1.6}
& 0.0110 & 20 & 5.0\\
&0.25
& \blue{0.0016} & 20 & \blue{1.3}
& 0.0038 & 20 & 6.4\\
\hdashline
$\sigma=0.1$
&0.15
& \blue{0.0806} & 20 & \blue{2.1}
& 0.0821 & 20 & 5.3\\
&0.25
& \blue{0.0531} & 20 & \blue{1.2}
& 0.0535 & 20 & 6.3\\
\hdashline
$\sigma=0.2$
&0.15 
& {0.1613} & 20 & \blue{2.9}
& 0.1660 & 20 & 4.9\\
&0.25
& \blue{0.1064} & 20 & {2.7}
& 0.1072 & 20 & 5.6\\
\hline
\end{tabular}
\end{table}

\begin{table}[h]
\renewcommand{\arraystretch}{1.5}
\setlength{\belowcaptionskip}{-0.05cm}
\centering
\scriptsize
\caption{The error results of Alg.\,\ref{FGL0alg1} and Alg.\,\ref{FGL0alg2} for (\ref{FGL0_ne}), and AMM and HAMM for (\ref{MS-FL20}) with different $\mu$}\label{glob_comp}
\begin{tabular}{|cc||cccccc|}
\hline
& & $\mu$ in (\ref{MS-FL20}) & $2$& $1$ & $10^{-1}$& $10^{-5}$& $10^{-10}$ \\
\hline
Alg.\,\ref{FGL0alg1}
&$1.6\cdotp10^{-15}$
&AMM
&$2.8$
&$1.4$
&$1.4\cdotp10^{-1}$
&$1.4\cdotp10^{-5}$
&$1.4\cdotp10^{-10}$
\\
\hline
Alg.\,\ref{FGL0alg2}
&$1.6\cdotp10^{-15}$
&HAMM
&$2.8$
&$1.4$
&$1.4\cdotp10^{-1}$
&$1.4\cdotp10^{-5}$
&$1.4\cdotp10^{-10}$
\\
\hline
\end{tabular}
\end{table}

\subsection{Matrix completion on real data}\label{real_data}
In this subsection, we use some real datasets (Jester joke dataset and MovieLens dataset) to compare the numerical performance of Alg.\,\ref{FGL0alg1} and Alg.\,\ref{FGL0alg2} for problem (\ref{FGL0_ne}), and AMM and HAMM algorithms in \cite{Pan2022factor} for problem (\ref{MS-FL20}). The comparison results are listed in Tables \ref{jester}-\ref{Movie-1M}. It is worth noting that for these real datasets, the low rankness of the output matrix is what we expect, but this does not mean that the output with the lowest rank-one is the best.

The adaptive tuning strategy for $\lambda$ with $n_\lambda=50$, $\delta_1=0.1$ and $\delta_2=1$ is adopted to run Alg.\,\ref{FGL0alg1} and Alg.\,\ref{FGL0alg2}. The parameter setting of AMM and HAMM are the same as used in Tables 2-4 of \cite{Pan2022factor}. For each dataset, the data matrix with $i$th row corresponding to the ratings given by the $i$th user is regarded as the original incomplete data matrix $Z^o$. We generate the target matrix $Z^\diamond$ from $Z^o$ by some random selection of columns and rows and set $\varDelta:=\{(i,j):Z^\diamond_{ij}\text{ is given}\}$. For the given observed index set $\varGamma$, $\mathcal{P}_{\varGamma}(\bar{Z}^\diamond)=\mathcal{P}_{\varGamma}(Z^\diamond)$ in (\ref{FGL0_ne}) and (\ref{MS-FL20}). Throughout this subsection, the processing on these real datasets are consistent with that in \cite{Fang2018MP,Pan2022factor}.

In Table \ref{jester}, the used dataset is the Jester joke dataset from \url{http://www.ieor.berkeley.edu/~goldberg/jester-data/}, which contains $4.1$ million ratings for $100$ jokes from $73421$ users. The whole dataset contains three subdatasets: (1) Jester-1: $24983$ users who have rated $36$ or more jokes; (2) Jester-2: $23500$ users who have rated $36$ or more jokes; (3) Jester-3: $24938$ users who have rated between $15$ and $35$ jokes.
Following \cite{Fang2018MP,Pan2022factor}, 
we randomly select $n_u$ rows from $Z^o$ and randomly permute these ratings to generate $Z^\diamond\in\mathbb{R}^{{n_u}\times 100}$ and generate the observed index set $\varGamma$ by Scheme 1. The rating range is from $r_{\min}=-10$ to $r_{\max}=10$.

The datasets in Table \ref{Movie-100K} and Table \ref{Movie-1M} are Movie-100K and Movie-1M from the MovieLens dataset, which is available through \url{http://www.grouplens.org/node/73}. The Movie-100K dataset contains $100000$ ratings for $1682$ movies by $943$ users and Movie-1M dataset contains $1000209$ ratings of $3952$ movies made by $6040$ users. We generate $Z^\diamond$ by randomly selecting $n_r$ users and their $n_c$ column ratings from $Z^o$, and generate $\varGamma$ by sampling the observed entries with Scheme 1 and Scheme 2 for Movie-100K dataset, and Scheme 1 for Movie-1M dataset; For more details, refer to \cite{Fang2018MP,Pan2022factor}. Since the rating range in $Z^\diamond_{ij}$ is $[1,5]$, we let $r_{\min}=1$ and $r_{\max}=5$.

In these tables, the listed values are the average results of five different instances for each setting. As reported in \cite{Pan2022factor} for real datasets, HAMM is shown to be superior to AMM \cite{Pan2022factor}, ALS \cite{hastie2015matrix} and  ADMM \cite{Fang2018MP} for most instances. From Table \ref{jester}, it can be seen that Alg.\,\ref{FGL0alg1} and Alg.\,\ref{FGL0alg2} have smaller NMAE and less running time than AMM and HAMM. Combining with Table \ref{Movie-100K} and Table \ref{Movie-1M}, Alg.\,\ref{FGL0alg1} has demonstrated to perform better than the other three algorithms in terms of the recovery error and running time. 

\begin{table}[h]
\setlength{\belowcaptionskip}{-0.01cm}
\centering
\scriptsize
\caption{\small Average NMAE, rank and running time of four algorithms for Jester joke dataset}\label{jester}
\begin{tabular}{cccccccc}
\hline
Dataset&\!\! ($n_u$,\,SR)& \multicolumn{3}{l}{\qquad\qquad Alg.\,\ref{FGL0alg1}}
&\multicolumn{3}{l}{\qquad\qquad AMM}\\
\cmidrule(lr){3-5} \cmidrule(lr){6-8}
& & NMAE & rank & time & NMAE & rank & time \\
\hline 
Jester-1 
&(1000,\,0.15) 
& \blue{0.1908}  & 2.0 & {\color{blue} 0.4}
& 0.1941 & 1.0 & 0.9\\
&(1000,\,0.25) 
& \blue{0.1815} & 1.6 &{0.2}
& 0.1870 & 1.0 & 0.5\\
			
&(2000,\,0.15)  
& {0.1901} & 1.4 & {\color{blue} 0.3}
& 0.1945 & 1.0 & 0.9\\
&(2000,\,0.25)  
& \blue{0.1825} & 1.6 & {\color{blue} 0.2}
& 0.1884 & 1.0 & 0.6\\
			
&(4000,\,0.15)  
& {0.1918} & 1.4 & {\color{blue} 0.4}
& 0.1966 & 1.2 & 2.0\\
&(4000,\,0.25)  
& \blue{0.1789} & 1.8 & {\color{blue} 0.3}
& 0.1827 & 1.4 & 1.9\\
\cmidrule(lr){1-8}
Jester-2 
&(1000,\,0.15)  
& \blue{0.1882} & 2.0 & {\color{blue} 0.4}
& 0.1926 & 1.2 & 1.4\\
&(1000,\,0.25)  
& \blue{0.1783} & 2.0 & {\color{blue} 0.2}
& 0.1870 & 1.0 & 0.5\\

&(2000,\,0.15) 
& \blue{0.1893} & 2.0 & {\color{blue} 0.4}
& 0.1938 & 1.0 & 0.9\\
&(2000,\,0.25)  
& \blue{0.1826} & 1.6 & {\color{blue} 0.2}
& 0.1878 & 1.0 & 0.6\\

&(4000,\,0.15)  
& \blue{0.1905} & 1.6 & {\color{blue} 0.4}
& 0.1960 & 1.0 & 1.1\\
&(4000,\,0.25)  
& \blue{0.1772} & 2.0 & {0.4} 
& 0.1861 & 1.0 & 0.7\\
			
\cmidrule(lr){1-8}
Jester-3 
&(1000,\,0.15) 
& \blue{0.2376} & 2.6 & {\color{blue} 0.6}
& 0.2767 & 3.4 & 7.6\\
&(1000,\,0.25)  
& \blue{0.2261} & 2.6 & {\color{blue} 0.4}
& 0.2852 & 7.8 & 12.3\\
			
&(2000,\,0.15)  
&\blue{0.2312} & 2.2 & {\color{blue} 0.6}
& 0.2788 & 1.6 & 4.8\\
&(2000,\,0.25)  
& \blue{0.2217} & 2.6 & {\color{blue} 0.5}
& 0.2456 & 6.2 & 12.2\\
			
&(4000,\,0.15)  
& \blue{0.2250} & 2.2 & {\color{blue} 0.7}
& 0.2621 & 1.4 & 5.4\\
&(4000,\,0.25)  
& \blue{0.2167} & 2.4 & {\color{blue} 0.6}
& 0.2554 & 6.6 & 17.1\\
\hline
Dataset&\!\! ($n_u$,\,SR) &\multicolumn{3}{l}{\qquad\qquad Alg.\,\ref{FGL0alg2}}
&\multicolumn{3}{l}{\qquad\qquad HAMM}\\
\cmidrule(lr){3-5} \cmidrule(lr){6-8} 
& & NMAE & rank & time & NMAE & rank & time \\
\hline 
Jester-1 
&(1000,\,0.15) 
& {0.1921}  & 1.8 & {\color{blue} 0.4}
& 0.1940 & 1.0 & {\color{blue} 0.4}\\
&(1000,\,0.25) 
& {0.1837} & 1.4 & {\color{blue} 0.1}
& 0.1870 &1.0 & 0.3\\
			
&(2000,\,0.15)  
& \blue{0.1892} & 1.8 & {0.4}
& 0.1944 &1.0 & 0.5\\
&(2000,\,0.25)  
& {0.1827} & 1.6 & {\color{blue} 0.2}
& 0.1884 & 1.0 & 0.4\\
			
&(4000,\,0.15)  
& \blue{0.1908} & 1.4 & {\color{blue} 0.4}
& 0.1970 & 1.2 & 1.1\\
&(4000,\,0.25)  
& 0.1827 & 1.4 & {\color{blue} 0.3}
& 0.1827 & 1.4 & 1.1\\
\cmidrule(lr){1-8}
Jester-2 
&(1000,\,0.15)  
& {0.1894} & 2.0 & {\color{blue} 0.4}
& 0.1928 & 1.2 & 0.6\\
&(1000,\,0.25)  
& {0.1800} & 1.8 & {\color{blue} 0.2}
& 0.1870 & 1.0 & 0.3\\

&(2000,\,0.15) 
& 0.1910 & 1.8 & {\color{blue} 0.4}
& 0.1938 & 1.0 & 0.5\\
&(2000,\,0.25)  
& 0.1845 & 1.4 & {\color{blue} 0.2}
& 0.1878 & 1.0 & 0.4\\

&(4000,\,0.15)  
& 0.1915 & 1.6 & {\color{blue} 0.4}
& 0.1959 & 1.0 & 0.6\\
&(4000,\,0.25)  
& {0.1810} & 1.6 & {\color{blue} 0.3}
& 0.1861 & 1.0 & 0.5\\
			
\cmidrule(lr){1-8}
Jester-3 
&(1000,\,0.15) 
& {0.2432} & 2.6 & {\color{blue} 0.6}
& 0.2640 & 3.6 & 1.5\\
&(1000,\,0.25)  
& {0.2348} & 2.4 & {\color{blue} 0.4}
& 0.2563 & 2.0 & 1.0\\
			
&(2000,\,0.15)  
& {0.2323} & 2.2 & {\color{blue} 0.6}
& 0.2537 & 2.8 & 1.9\\
&(2000,\,0.25)  
& {0.2301} & 2.8 & {0.6}
& 0.2437 & 1.6 & 1.1\\
			
&(4000,\,0.15)  
& {0.2298} & 2.4 & {0.9}
& 0.2633 & 1.8 & 1.9\\
&(4000,\,0.25)  
& {0.2226} & 2.6 & {0.8}
& 0.2381 & 4.8 & 2.6\\
\hline
\end{tabular}
\end{table}

\begin{table}[h]
\setlength{\belowcaptionskip}{-0.01cm}
\centering
\scriptsize
\caption{\small Average NMAE, rank and running time of four algorithms for Movie-100K dataset}\label{Movie-100K}
\begin{tabular}{cccccccc}
\hline
&SR& \multicolumn{3}{l}{\qquad\qquad Alg.\,\ref{FGL0alg1}}
&\multicolumn{3}{l}{\qquad\qquad AMM}\\
\cmidrule(lr){3-5} \cmidrule(lr){6-8}
& & NMAE & rank &time&  NMAE & rank & time  \\
\hline
Scheme 1
&0.10 
& \blue{0.2195} & 1.6 & {\color{blue} 1.4}
& 0.2427 & 1.0 & 8.0\\
&0.15  
& \blue{0.2112} & 1.6 & {\color{blue} 1.2}
& 0.2236 & 1.0 & 7.9\\
&0.20  
& \blue{0.2083} & 2.0 & {\color{blue} 1.5}
& 0.2158 & 1.0 & 8.3\\
&0.25  
& \blue{0.2033} & 1.8 & {\color{blue} 1.4}
& 0.2085 & 1.0 & 8.0\\
\cmidrule(lr){1-8}
Scheme 2
&0.10  
& \blue{0.2206} & 1.8 & {\color{blue} 1.5}
& 0.2453 & 1.0  & 8.3\\
&0.15  
& \blue{0.2120} & 1.8 & {\color{blue} 1.5}
& 0.2240 & 1.0 & 7.9\\
&0.20  
& \blue{0.2087} & 1.8 & {\color{blue} 1.4}
& 0.2156 & 1.0 & 8.6\\
&0.25  
& \blue{0.2042} & 1.6 & {\color{blue} 1.3}
& 0.2095 & 1.0 & 9.6\\
\hline
&SR &\multicolumn{3}{l}{\qquad\qquad Alg.\,\ref{FGL0alg2}} &\multicolumn{3}{l}{\qquad\qquad HAMM}\\
\cmidrule(lr){3-5} \cmidrule(lr){6-8} 
& & NMAE & rank &time&  NMAE & rank & time  \\
\hline
Scheme 1
&0.10 
& 0.2227 & 1.6 & {\color{blue} 1.4}
& 0.2292 & 1.0 & 1.5\\
&0.15  
& 0.2126 & 1.6 & {\color{blue} 1.2}
& 0.2173 & 1.0 & 1.5\\
&0.20  
& 0.2135 & 2.0 & {\color{blue} 1.5}
& 0.2123 & 1.0 & 1.6\\
&0.25  
& 0.2050 & 1.8 & {\color{blue} 1.4}
& 0.2067 & 1.0 & 1.6\\
\cmidrule(lr){1-8}
Scheme 2
&0.10  
& 0.2245 & 1.8 & 1.6
& 0.2301 &1.0 & 1.6\\
&0.15  
& 0.2153 & 1.8 & {\color{blue} 1.5}
& 0.2178 & 1.0 & 1.5\\
&0.20  
& 0.2122 & 1.8 & {\color{blue} 1.4}
& 0.2123 & 1.0 & 1.5\\
&0.25  
& 0.2074 & 1.6 & {\color{blue} 1.3}
& 0.2074 & 1.0 & 1.6\\
\hline
\end{tabular}
\end{table}

\begin{table}[h]
\setlength{\belowcaptionskip}{-0.01cm}
\centering
\scriptsize
\caption{\small Average NMAE, rank and running time of four algorithms for Movie-1M dataset \label{Movie-1M}}
\begin{tabular}{cccccccc}
\hline
($n_r$,\,$n_c$)&\!\! SR & \multicolumn{3}{l}{\qquad\qquad Alg.\,\ref{FGL0alg1}}
&\multicolumn{3}{l}{\ \ \  \qquad\qquad AMM}\\
\cmidrule(lr){3-5} \cmidrule(lr){6-8} 
& &\!\!\!\! NMAE&\!\!\! rank \!\!\!\! &time &  NMAE \!\!\!& rank \!\!\!& time\\
\hline
(1500,\,1500)
&0.15 
& \blue{0.2124} & 1.8 & {\color{blue} 1.6}
& 0.2244 & 1.0 & 9.9\\
&0.25 
& \blue{0.2010} & 1.8 & {\color{blue} 1.7}
& 0.2061 & 1.0 & 10.8\\
\cmidrule(lr){1-8}
(2000,\,2000)
&0.15 
& \blue{0.2041} & 1.8 & {\color{blue} 2.5}
& 0.2125 & 1.0 & 14.5\\
&0.25 
& \blue{0.1952}  & 1.6 & {\color{blue} 2.4}
& 0.2010 & 1.0 & 11.0\\
\cmidrule(lr){1-8}
(3000,\,3000)
&0.15 
& \blue{0.1993} & 1.4 & 3.7
& 0.2035 & 1.0 & 26.3\\
&0.25 
& \blue{0.1883} & 2.0 & 5.2
& 0.1948 & 1.0  & 20.0\\
\cmidrule(lr){1-8}
(6040,\,3706)  
&0.15 
& \blue{0.1956} & 1.6 & 30.4
& 0.1971 & 1.0 & 248.0\\
&0.25 
& \blue{0.1852} & 2.0 & 30.7
& 0.1916 & 1.0 & 128.0\\
\hline
($n_r$,\,$n_c$)&\!\! SR &\multicolumn{3}{l}{\qquad\qquad Alg.\,\ref{FGL0alg2}}
&\multicolumn{3}{l}{\qquad\qquad HAMM}\\
\cmidrule(lr){3-5} \cmidrule(lr){6-8} 
& &\!\!\!\! NMAE&\!\!\! rank \!\!\!\! &time &  NMAE \!\!\!& rank \!\!\!& time\\
\hline
(1500,\,1500)
&0.15 
& 0.2157 & 1.8 & {\color{blue} 1.6}
& 0.2171 & 1.0 & 1.8\\
&0.25 
& 0.2047 & 1.8 & {\color{blue} 1.7}
& 0.2037 & 1.0 & 1.9\\
\cmidrule(lr){1-8}
(2000,\,2000)
&0.15 
& 0.2069 & 1.8 & {\color{blue} 2.5}
& 0.2082 & 1.0 & 2.7\\
&0.25 
& 0.1963 & 1.6 & {\color{blue} 2.4}
& 0.1997 & 1.0 & 2.9\\
\cmidrule(lr){1-8}
(3000,\,3000)
&0.15 
& 0.2020 & 1.4 & {\color{blue} 3.6}
& 0.2014 & 1.0 & 5.1\\
&0.25 
& 0.1912 & 1.8 & {\color{blue} 4.7}
& 0.1944 & 1.0 & 5.5\\
\cmidrule(lr){1-8}
(6040,\,3706)
&0.15 
& 0.1978 & 1.4 & {\color{blue} 29.8}
& 0.1963 & 1.0 & 54.7\\
&0.25 
& 0.1881 & 1.8 & {\color{blue} 28.1}
& 0.1915 & 1.0 & 54.3\\
\hline
\end{tabular}
\end{table}

\section {Conclusions}\vspace{0.5ex}
In this paper, we focus on the factorized column-sparse regularized problem, which is an equivalent low-dimensional transformation of the rank regularized problem in the sense of global minimizers. 
By adding the bound constraints, the solution set of the considered problem (\ref{FGL0C}) covers all global minimizers of the rank regularized problem \eqref{LRMPR} that satisfy a predicted bound on the target matrix. Limited by the non-convexity from the loss and regularization functions, most of the previous work involved the general stationary points of the considered problem. 
Differently, by exploring some properties from its global minimizers but not all local minimizers, we strengthened the stationary point of the considered problem and designed two algorithms to illustrate the availability of this class of strong stationary points in theory and experiments. 
This work mainly contains the following three key points. First, we established the equivalent relation on the rank regularized problem (\ref{LRMPR}) and the considered problem (\ref{FGL0C}) in the sense of global minimizers. This combined with the dimension reduction from the matrix factorization makes the  model (\ref{FGL0C}) a promising alternative to problem (\ref{LRMPR}). Second, we gave some optimality analysis and a variety of relations between the problem (\ref{FGL0C}) and its relaxation problem (\ref{FGL0R}). 
Further, we established the equivalence on solving the strong stationary points of problems (\ref{FGL0C}) and (\ref{FGL0R}). Third, we designed two algorithms from different perspectives, and gave their commonalities and distinctions from both theoretical and numerical aspects. Moreover, through numerical experiments, we showed the 
superiority of the considered model (\ref{FGL0C}) and proposed algorithms. Meanwhile, we verified the acceleration of  our dimension reduction technique and line search procedure on the proposed algorithms.

\backmatter





\bmhead{Acknowledgments}

The authors would like to thank  Shaohua Pan and Ethan X. Fang for providing some codes in \cite{Fang2018MP} and \cite{Pan2022factor} for the numerical comparison in Section 6. The authors would also like to thank the Associate Editor and the referees for their helpful comments and suggestions to improve this paper.

\section*{Declarations}


Part of this work was done while Wenjing Li was with Department of Mathematics, National University of Singapore. The research of Wenjing Li is supported by the National Natural Science Foundation of China Grant (No.\,12301397) and the China Postdoctoral Science Foundation Grant (No.\,2023M730876). The research of Wei Bian is supported by the National Natural Science Foundation of China Grants (No.\,12271127, 62176073), the National Key Research and Development Program of China (No.\,2021YFA1003500) and the Fundamental Research Funds for the Central Universities of China (No.\,2022FRFK0600XX). The research of Kim-Chuan Toh is supported by the Ministry of Education, Singapore, under its Academic Research Fund Tier 3 grant call (MOE-2019-T3-1-010).

\bibliography{references}

\end{document}